\documentclass[11pt]{amsart}

\usepackage[svgnames,x11names]{xcolor}
\usepackage{amssymb,amsfonts}
\usepackage[pagebackref,colorlinks,urlcolor={DarkOliveGreen4},
linkcolor={DarkSlateBlue},
citecolor={Chocolate4}]{hyperref}
\usepackage[capitalize]{cleveref}
\usepackage{amssymb,textcomp}
\usepackage{graphicx}
\numberwithin{equation}{section}

\makeatletter
\def\ps@myplain{
  \let\@oddhead\@empty
  \let\@evenhead\@empty
  \let\@oddfoot\@empty
  \let\@evenfoot\@empty
  \def\@oddhead{{\tiny\hfill\thepage}}  % Tiny font, right-aligned
  \def\@evenhead{{\tiny\hfill\thepage}} % Same for even pages
}
\makeatother

\usepackage{comment}
\usepackage[official]{eurosym}
\usepackage[greek,english]{babel}
\usepackage{bbm}
\usepackage[utf8]{inputenc}
\usepackage[T1]{fontenc}
\usepackage{pdfpages}
\usepackage{tensor}
\usepackage{enumerate}
\usepackage{relsize}
\usepackage[normalem]{ulem}
\usepackage[stable]{footmisc}
\usepackage{mathtools}
\usepackage[mathscr]{eucal}
\usepackage[a4paper, twoside=false, vmargin={2cm,3cm}, includehead]{geometry}
\usepackage{pgfgantt}
\usepackage{graphicx}
\usepackage{xcolor}
\ganttset{group/.append style={orange},
	milestone/.append style={red},
	progress label node anchor/.append style={text=red}}

\usepackage{cancel}

\definecolor{commentpurple}{RGB}{140, 40, 170}

\newcommand{\rsout}{%
  \bgroup
  \markoverwith{\textcolor{red}{\rule[0.8ex]{2pt}{0.8pt}}}%
  \ULon
}

\usepackage{amsmath}

\usepackage{enumerate}

\usepackage[utf8]{inputenc}
\usepackage[T1]{fontenc}

\usepackage[mathscr]{eucal}


\newtheorem{lemma}{Lemma}[section]
\newtheorem{theorem}[lemma]{Theorem}
\newtheorem*{theorem*}{Theorem}
\newtheorem{corollary}[lemma]{Corollary}

\newtheorem*{question*}{Question}
\newtheorem{proposition}[lemma]{Proposition}
\newtheorem*{proposition*}{Proposition}

\newtheorem{conjecture}{Conjecture}

\newtheorem*{problem*}{Problem}

\theoremstyle{definition}
\newtheorem{definition}{Definition}[section]
\newtheorem*{claim*}{Claim}

\newtheorem{rmk}{Remark}

\newtheorem*{example}{Example}

\newtheorem*{remark}{Remark}

\newcommand{\C}{{\mathbb C}}
\newcommand{\E}{{\mathbb E}}
\newcommand{\D}{{\mathbb D}}

\newcommand{\N}{{\mathbb N}}
\renewcommand{\P}{{\mathbb P}}
\newcommand{\Q}{{\mathbb Q}}
\newcommand{\R}{{\mathbb R}}
\renewcommand{\S}{\mathbb{S}}
\newcommand{\T}{{\mathbb T}}
\newcommand{\Z}{{\mathbb Z}}
\newcommand{\U}{{\mathbb U}}
\newcommand{\X}{{\mathcal{X}}}
\newcommand{\cesE}{\mathlarger{\mathbb{E}}}
\newcommand{\inflim}{\underline{\lim}}

\newcommand{\CP}{{\mathcal P}}

\newcommand{\CX}{{\mathcal X}}

\newcommand{\M}{{\mathcal{M}}}
\newcommand{\A}{{\mathcal{A}}}
\newcommand{\Oh}{{\rm O}}
\newcommand{\oh}{{\rm o}}

\newcommand{\e}{\varepsilon}% for the function exp

\newcommand{\Folner}{F\o{}lner}

\title[Recurrence for pretentious systems along generalized Pythagorean triple]{Recurrence for pretentious systems along generalized Pythagorean triples}
\date{\today}
\author{Nikos Frantzikinakis and Andreas Mountakis}
	\address[Nikos Frantzikinakis]{University of Crete, Department of mathematics and applied mathematics, Voutes University Campus, Heraklion 70013, Greece} \email{frantzikinakis@gmail.com}
		\address[Andreas Mountakis]{Northwestern University, Department of mathematics, Evanston, IL 60208, USA} \email{andreas.mountakis@northwestern.edu}

\begin{document}

\vspace*{-0.015cm}
\begin{abstract} 
We establish multiple recurrence results for pretentious measure-preserving multiplicative  
actions along generalized Pythagorean triples, that is, solutions to the equation  $ax^2 + b y^2 = c z^2$.  
This confirms the ergodic-theoretic form of the generalized Pythagorean partition regularity  
conjecture in this  critical case of structured  measure-preserving actions. As a consequence of our main theorem, any finite coloring  
of $\mathbb{N}$ generated by the level sets of finitely many pretentious completely  
multiplicative functions must contain a monochromatic generalized Pythagorean triple.  
\end{abstract}

\thanks{The authors were supported  by the Hellenic Foundation for Research 
and Innovation ELIDEK HFRI-NextGenerationEU-15689.}

\maketitle
\tableofcontents

\section{Introduction and main results}
\subsection{Introduction}
A central principle in arithmetic Ramsey theory asserts that any finite 
partition of the natural numbers must contain at least one cell exhibiting 
significant internal structure. Here, by structure we mean the 
existence of nontrivial solutions to algebraic equations or systems of 
equations, with all variables lying in the same cell of the partition.

\begin{definition}
	An equation $P(x, y, z) = 0$ is \emph{partition regular} if, for every finite 
	partition of $\mathbb{N}$, there exist distinct $x, y, z$ in the same cell that 
	satisfy the equation.
\end{definition}

The partition regularity of linear equations is classically illustrated by the 
results of Schur~\cite{schur_thm} and van der Waerden~\cite{van_der_Waerden_thm}, 
and was substantially 
extended by Rado~\cite{Rado-original}, who gave a complete classification of 
systems of linear equations that are partition regular.

\begin{definition}
	We say that $(a,b,c) \in \mathbb{N}^3$ is a \emph{Rado triple} if $a = c$, 
	$b = c$, or $a + b = c$.
\end{definition}

In the case of a single linear equation, Rado’s criterion asserts that
\begin{equation}\label{E:RadoLinear}
	a x + b y = c z
\end{equation}
is partition regular if and only if $(a,b,c)$ is a Rado triple.

Polynomial equations pose far greater challenges in the study of partition 
regularity. The most prominent example is the long-standing open problem of 
Erd\H{o}s and Graham~\cite{Graham-prob-1, Graham-prob-2}, which asks whether the 
Pythagorean equation $x^2 + y^2 = z^2$ is partition regular.
More generally, one may ask about the partition regularity of broader classes 
of homogeneous quadratic equations in three variables. This problem captures 
the essential structural difficulties of the nonlinear setting and has become 
a focus of considerable recent research. In this direction we state the 
following conjecture:

\begin{conjecture}\label{Rado}
	If $a, b, c \in \mathbb{N}$, then
	\begin{equation}\label{gen-pyth-eq}
		a x^2 + b y^2 = c z^2
	\end{equation}
	is partition regular if and only if $(a,b,c)$ is a Rado triple.
\end{conjecture}

The necessity of the Rado triple condition follows from the fact that 
partition regularity of \eqref{gen-pyth-eq} implies partition regularity of its 
linear counterpart \eqref{E:RadoLinear}. The sufficiency, however, is far less 
understood: no instance of a Rado triple $(a,b,c)$ is currently known for which 
partition regularity of \eqref{gen-pyth-eq} has been established. Depending on 
the coefficients, the problem exhibits a range of distinct and often 
deeply challenging behaviors.
 In \cite{Fra-Host-structure-multiplicative}, among other results, the authors 
 prove partition regularity with respect to the pair $(x,y)$ for the equation 
 $16x^2 + 9y^2 = z^2$; that is, for every finite partition of $\mathbb{N}$ there 
 exist distinct $x, y$ in the same cell and $z \in \mathbb{N}$ such that the 
 equation is satisfied. In \cite{Fra-Klu-Mor}, partition regularity with respect 
 to $(x,y)$ and with respect to $(y,z)$ is established for the Pythagorean equation 
 $x^2 + y^2 = z^2$. { This result was later extended in \cite{fra-klu-mor-2} to
 partition regularity of \eqref{gen-pyth-eq} with respect to
 $(x,y)$ and $(y,z)$ when $a=c$, and with respect to
 $(x,y)$ and $(x,z)$ when $b=c$.}
For further results on the partition regularity of quadratic 
 equations, see 
 \cite{bergelson-ergodic-ramsey, browning-prendiville-roth, chapman-part-reg, 
 	chow-lindqvist-prendiville, csikvari-gyarmati-sarkozy, dinasso-luperi, 
 	green-lindqvist, heule-kullmann-marek, khalfalah-szemeredi, moreira-monochromatic}.
 
 The techniques developed in 
 \cite{Fra-Host-structure-multiplicative, fra-klu-mor-2, Fra-Klu-Mor} are 
 inherently limited to pair configurations and do not extend to partition 
 regularity problems involving triples. To overcome this obstacle, we employ the 
 framework of ergodic theory to formulate sufficient conditions for the partition 
 regularity of \eqref{gen-pyth-eq}.

\begin{definition}
	A \emph{multiplicative action} is a quadruple $(X, \mathcal{X}, \mu, T_n)$,
	where $(X, \mathcal{X}, \mu)$ is a probability space and for
	$n \in \mathbb{N}$, the transformations $T_n \colon X \to X$ are
	invertible and measure-preserving, have measurable inverses, and satisfy
	$T_1 = \mathrm{id}$ and $T_m \circ T_n = T_{mn}$ for all
	$m, n \in \mathbb{N}$. This uniquely determines a
	$(\mathbb{Q}_{>0}, \times)$-action on $(X, \mathcal{X}, \mu)$
	defined by $T_{m/n} := T_m \circ T_n^{-1}$. The action is said to be
	\emph{finitely generated} if the set $\{T_p : p \in \mathbb{P}\}$ is finite.
\end{definition}

Using a variant of Furstenberg's correspondence
principle~\cite{furstenberg_szemeredi_thm} for dilation-invariant
densities~\cite{Bergelson_multiplicative}, together with the
characterization of partition regularity through multiplicatively
syndetic sets by Chapman~\cite{ChapmanSyndetic},  { Conjecture~\ref{Rado} is
equivalent to the following multiple recurrence statement.
A proof of this equivalence is given in the Appendix.}

\begin{conjecture}[Ergodic variant of Conjecture~\ref{Rado}]\label{RadoErgodic}  
	Let $(a,b,c)$ be a Rado triple. Then for every multiplicative action 
	$(X,\mathcal{X},\mu,T_n)$ and every set $A \in \mathcal{X}$ with  
	$X = \bigcup_{j=1}^k T_j^{-1} A$ for some $k \in \mathbb{N}$, there exist 
	distinct $x, y, z \in \mathbb{N}$ such that  
	$$
	\mu\bigl(T_x^{-1} A \cap T_y^{-1} A \cap T_z^{-1} A\bigr) > 0  
	\quad\text{and}\quad  
	a x^2 + b y^2 = c z^2.
	$$
\end{conjecture}  
If $a+b\neq c$, simple examples show that this multiple recurrence
can fail for sets $A$ of positive measure. For instance, when
$a=b=c=1$, consider the measure-preserving dilations
$T_nx=n^2x\pmod{1}$ on $\mathbb{T}$ with Haar measure
$m_{\mathbb{T}}$, and take $A=[1/3,2/3]$.
 { Passing to a simultaneous natural extension and replacing $A$
by its inverse image gives an invertible multiplicative action
with the same failure of recurrence.}
 
 \subsection{Main results} 
In this paper we establish Conjecture~\ref{RadoErgodic} for a  
class of structured systems; these are
the \emph{pretentious multiplicative actions}. Although not the 
most technically demanding case, they are the most plausible setting for 
counterexamples to nonlinear partition regularity of homogeneous equations, 
making them a critical test case. Verifying the 
conjecture in this setting  provides strong evidence for both 
Conjectures~\ref{Rado} and~\ref{RadoErgodic}.

Pretentious systems in multiplicative dynamics play a role analogous to that of 
rational Kronecker systems in the additive setting. The analogy is useful but 
imperfect: in Kronecker systems, multiple recurrence is typically  straightforward, 
whereas pretentious systems possess a richer structure, making their analysis 
substantially more delicate. The notation for the next definition is given in 
Sections~\ref{SS:mf} and~\ref{sec: spectral measures}.

\begin{definition}
	A multiplicative action $(X,\CX,\mu,T_n)$ is \emph{pretentious} if the spectral 
	measure of every $F \in L^2(\mu)$ is supported on the set of pretentious 
	multiplicative functions.\footnote{Given a measure space $(X, \X, \mu)$, we say 
		that $\mu$ is supported on $E \in \X$ if $\mu(X \setminus E) = 0$.} 
\end{definition}

The simplest example of a pretentious multiplicative action is a rotation by a 
pretentious multiplicative function: if $f \colon \N \to \S^1$ is pretentious 
and $\S^1$ carries the Haar measure, define $T_n \colon \S^1 \to \S^1$ by  
$T_n z := f(n) \cdot z$. If $(X,\CX,\mu,T_n)$ is finitely generated, Charamaras 
showed in \cite{Charamaras-multiplicative} that pretentious systems are spanned by 
``pretentious 
rational eigenfunctions''. This fails for general pretentious multiplicative 
actions, as shown in \cite{Fra-Hom_Triples}.

Our main result, stated below, remains nearly as difficult to prove even in the 
simpler setting of multiplicative rotations by several pretentious multiplicative 
functions. We choose to work in the broader framework of pretentious multiplicative actions, 
as this is likely to be a necessary intermediate step toward establishing 
Conjecture~\ref{RadoErgodic}. Further discussion of this perspective is given 
in Section~\ref{SS:Problem}.

\begin{theorem}\label{T:Main}
	Let $(a,b,c)$ be a Rado triple. For every pretentious multiplicative action 
	$(X,\CX,\mu,T_n)$ and every set $A \in \CX$ with $\mu(A) > 0$, 
    there exist distinct $x, y, z \in \N$ such that  
	$$
	\mu\bigl( T^{-1}_x A \cap T^{-1}_y A \cap T^{-1}_z A\bigr) 
	>0 
	\quad \text{and} \quad a x^2 + b y^2 = c z^2.  
	$$
\end{theorem}  

{
\begin{rmk}\label{rmk_lower_bound}
Our proofs give in fact something stronger: 
we show that the conclusion holds for many triples $x, y, z \in \N$, and we 
provide lower bounds for the measure of the intersection.
If $a=c$ or $b=c$, then for any 
$\e>0 $, we show  there exist distinct  $x, y, z \in \N$ such that  
$\mu ( T^{-1}_x A \cap T^{-1}_y A \cap T^{-1}_z A ) > 
\mu(A) ^3 - \e$, while in case $a+b=c$, the exponent is either $3$
or $4$, depending on whether all of the quadratic forms appearing 
in the parametrization of the solutions of 
$ax^2+ by^2 = cz^2$ are irreducible or not. Including these cases and bounds would make the statement of 
Theorem \ref{T:Main} lengthy, so we give the explicit bounds later.
\end{rmk}
}

\begin{rmk}
%	Our argument extends, without substantive changes, to the %setting of three 
%	pretentious multiplicative actions $T_n, S_n, R_n$ acting on %the same 
%	probability space $(X,\CX,\mu)$, without assuming %commutativity. Moreover, we 
%	show that the conclusion holds for many triples $x, y, z \in %\N$; precise 
%	formulations appear in Theorems~\ref{T: main a=c} and \ref{T: %main a+b=c}.
	For a general multiplicative action $(X,\CX,\mu,T_n)$, the functions 
	$F \in L^2(\mu)$ whose spectral measure is supported on pretentious 
	multiplicative functions form the \emph{pretentious factor} of the system (see 
	\cite[Theorem~2.4]{Fra-Hom_Triples}). Theorem~\ref{T:Main} therefore establishes 
	multiple recurrence along generalized Pythagorean triples within this factor.
\end{rmk}

We deduce the following corollary, which  confirms 
Conjecture~\ref{Rado} in the case where the partition is defined by the level 
sets of finitely many pretentious multiplicative functions.

\begin{corollary}
	Let $(a,b,c)$ be a Rado triple, and let $f_1, \ldots, f_s \colon \N \to \S^1$ 
	be pretentious completely multiplicative functions. Then for every open arc $I$ 
	of $\S^1$ containing $1$ there exist distinct $x, y, z \in \N$ such that 
	$$
	a x^2 + b y^2 = c z^2 \quad \text{and} \quad  
	f_j(x), f_j(y), f_j(z) \in I \quad \text{for } j = 1, \ldots, s. 
	$$
\end{corollary}

\begin{proof}
Let 
\begin{equation}\label{E:Id}
I_\delta := \{ e^{2\pi i t} \colon t \in (-\delta, \delta) \}
\end{equation}
 for some  
$\delta \in (0,1/4)$ chosen small enough so that 
{$I_{3\delta} \subset I$}.  
Let $X:= (\S^1)^s$, let $\CX$ be the Borel $\sigma$-algebra of $X$, and  
let $\mu$ denote the Haar measure on $X$. For $n \in \N$  and $z_1,\ldots, z_s\in \S^1$, define  
$$  
T_n(z_1, \ldots, z_s) := (f_1(n) z_1, \ldots, f_s(n) z_s).  
$$   

We claim that $(X,\CX,\mu,T_n)$ is a pretentious multiplicative action.  
Indeed, the spectral measure of each character of $X$ is supported on  the set
$$
\{ f_1^{k_1} \cdots f_s^{k_s} \colon k_1,\ldots,k_s \in \Z \},  
$$  
which consists entirely of pretentious, completely multiplicative functions.  
Since linear combinations of characters are dense in $L^2(\mu)$, and the class of  
functions in $L^2(\mu)$ whose spectral measures are supported entirely on  
pretentious multiplicative functions is closed, the claim follows.

{ 
Let $A := I_\delta \times \cdots \times I_\delta$. Then $\mu(A) > 0$, and  
Theorem~\ref{T:Main} provides distinct $\tilde{x},\tilde{y}, \tilde{z} \in \N$ such that $a \tilde{x}^2 + b \tilde{y}^2 = c \tilde{z}^2 $ and 
$$
\mu(T_{\tilde{x}}^{-1}A \cap T_{\tilde{y}}^{-1}A \cap 
T_{\tilde{z}}^{-1}A) > 0.  $$  
Consequently, for each $j=1,\ldots,s$ we have  
$$
(f_j(\tilde{x})^{-1} I_\delta) \cap (f_j(\tilde{y})^{-1} I_\delta) \cap  
(f_j(\tilde{z})^{-1} I_\delta) \neq \emptyset,  
$$  
which implies that there is $w_j \in \mathbb{S}^1$ such that 
\begin{equation}\label{eq_cor}
    f_j(\tilde{x}), \ f_j(\tilde{y}), \ f_j(\tilde{z}) \in w_j I_{\delta}  
\quad \text{for } j=1,\ldots,s.  
\end{equation}
Consider the point $(f_1(\tilde{x}), \ldots, f_s(\tilde{x})) \in 
(\mathbb{S}^1)^s$. By the pigeonhole principle there is $n\in \N$ such that 
$$ (f_1 ^n (\tilde{x}), \ldots , f_s ^n (\tilde{x})) = 
(f_1 (\tilde{x} ^n), \ldots , f_s (\tilde{x} ^n)) \in A=
I_\delta \times \cdots \times I_\delta. $$
Let $k= \tilde{x} ^{n-1}\in \N$, and $x= k\tilde{x}, y= k\tilde{y}, 
z= k\tilde{z}$.
Indeed, two powers of this point with distinct nonnegative exponents
lie in the same cell of a sufficiently fine finite partition of the
torus, so their quotient belongs to $A$.
 Then $ax^2 + by^2 =cz^2$, and for each $j=1, \ldots, s$ we 
have $f_j (x)= f_j (\tilde{x} ^n) \in I_{\delta} \subset I$. 
From \eqref{eq_cor} we see that 
$f_j(\tilde{y}) \in f_j (\tilde{x}) I_{2\delta}$, so 
\begin{equation*}
 f_j (y) = f_j(k) f_j (\tilde{y})\in f_j(k) f_j (\tilde{x}) I_{2\delta}
 = f_j(x) I_{2\delta} \subset I_{3\delta} \subset I,   
\end{equation*}
where we used $f_j(x)\in I_\delta$ to obtain
$f_j(x)I_{2\delta}\subset I_{3\delta}$.
Analogously we have that for each $j$, $f_j (z) \in I_{3\delta}
\subset I$. This completes the proof. }
\end{proof}

\subsection{Proof strategy} 
We begin with a rough outline of our general proof strategy; for the 
Pythagorean equation, a more detailed sketch appears in 
Section~\ref{Sec_proof_plan}.  

To prove Theorem~\ref{T:Main}, we work with solutions of  
$ a x^2 + b y^2 = c z^2 $ in the parametric form  
$$
x = k \, P_1(m,n), \quad y = k\,  P_2(m,n), \quad z = k \, P_3(m,n),  
\quad k, m, n \in \N,  
$$  
where $P_j$, $j \in \{1,2,3\}$, are binary quadratic forms that will be explicitly described later.  
The key averages to analyze are  
(averaging notation is given in Section~\ref{SS:Notation}):  
{ \begin{multline}\label{proof_strategy_eq_1}
	\cesE_{(m,n) \in  [N]^2}  
	\int T_{P_1(m,n)}F \cdot T_{P_2(m,n)}F \cdot T_{P_3(m,n)}F \, d\mu \\
    =\cesE_{(m,n) \in  [N]^2}  
	\int F \cdot T_{\frac{P_1(m,n)}{P_3(m,n)}}F \cdot T_{\frac{P_2(m,n)}{P_3(m,n)}}F \, d\mu.  
\end{multline}  }
When the action is ergodic and finitely generated, we restrict the averaging 
to a two-dimensional grid $(Qm+1, Qn)$, where $Q \in \N$ is taken to be highly 
divisible.  
Applying the spectral theorem together with the concentration estimates for 
multiplicative functions given in Propositions~\ref{linear_conc} 
and~\ref{quad_conc}, we find that {$T_{{P_2(m,n)}/{P_3(m,n)}}F$ can be replaced by $F$, while 
$T_{{P_1(m,n)}/{P_3(m,n)}}F$ can be replaced either by $F$ itself or by $T_{2aQn}F$.} 
Taking a multiplicative average over $Q$ then yields a grid for which the following  lower bound is achieved
$$
\int F\, d\mu \cdot \int F^2 \, d\mu - \varepsilon 
\ \geq\ \Big(\int F\, d\mu\Big)^3 - \varepsilon,
$$  
completing the proof. 

This outline describes the argument in the finitely generated case.  
For general actions, two significant new difficulties arise:  

\begin{enumerate}[(i)]
	\item \emph{Archimedean characters.}  
	Non-concentration phenomena appear for functions whose spectral measures 
	are supported on Archimedean characters. As in earlier work, we address 
	this by introducing a cutoff $S_{\delta}$ (see \eqref{def S_d a=c}, \eqref{def S_d a+b=c}) that, 
	for sufficiently small $\delta$,  localizes  the 
	Archimedean characters, so the corresponding component of the functions becomes concentrated either 
    around itself {or around $T_{1/q}F$ for a highly divisible $q$}.
	
	\item \emph{$1$-pretentious oscillating multiplicative functions.}  
	A greater challenge comes from functions whose spectral measures are 
	supported on $1$-pretentious oscillating multiplicative functions (see the example after Proposition \ref{linear_conc}).  
	Our new crucial idea here is to average over a grid $(Qm+1, Qn+v)$, carefully 
	choosing $Q$ and $v=v(Q_1,Q_2,Q_3)$ to satisfy congruences such as those in  
	Lemma \ref{lemma_choice_of_v_1}. This allows us to factor out highly divisible, 
	pairwise coprime terms $Q_j$ from each $P_j(Qm+1, Qn+v)$ for  $j \in \{1,2,3\}$.  
	We then perform a triple multiplicative averaging over the $Q_j$, 
	thereby eliminating the contribution from the oscillating component of $F$.  
	The resulting bound is  
	$$
	\int  F_1 \cdot F_2 \cdot F_3 \, d\mu - \varepsilon,  
	$$  
	where for  $j \in \{1,2,3\}$ the functions  $F_j$  are projections of $F$ onto suitable 
	factors. The proof is completed by applying the estimate in Proposition
	\ref{chu-lemma}.
\end{enumerate}

The fact that $(a,b,c)$ is a Rado triple plays a crucial role in the 
preceding arguments. It ensures that we can parametrize the solutions to 
\eqref{gen-pyth-eq} so that at least two of the quadratic forms 
$P_j$, $j \in \{1,2,3\}$, have leading coefficient $1$ with respect to the variable $m$.  
This property is essential even in the finitely generated case, as it allows 
us to replace { $T_{P_j(Qm+1,Qn+v)/P_3(Qm+1,Qn+v)}F$ by $F$ for at least one value of $j\in \{1,2\}$.}

To illustrate with a concrete example why our method fails  
when $(a,b,c)$ is not a Rado triple, consider the equation  
  $x^2 + y^2 = 4 z^2$, which admits the parametrization  
$$  
x = k \, 2(m^2-n^2), \quad y = k \, 4 m n, \quad  
z = k \, (m^2 + n^2), \quad k,m,n \in \N.  
$$  
This equation is partition regular with respect to each of the pairs  
$(x,y)$,  $(y,z)$, and $(x,z)$ (see \cite[Theorem~1.2]{fra-klu-mor-2}), but it 
is not partition regular with respect to $(x,y,z)$ since $(1,1,4)$ is not a 
Rado triple.  
If the action is finitely generated and ergodic, our argument yields a lower 
bound of the form  
$$  
\int F \, d\mu \cdot \int F \cdot T_2 F \, d\mu,  
$$  
and the presence of the term $T_2F$ in place of $F$ prevents positivity, as 
expected.

\subsection{An open problem}\label{SS:Problem}
For finitely generated actions, it appears plausible that the multiple 
recurrence property in Conjecture~\ref{RadoErgodic} holds for every 
set $A \in \mathcal{X}$ of positive measure.  
In the special case of the Pythagorean equation, this would follow directly 
from Theorem~\ref{T:Main} together with a positive solution to the problem 
below, which is a special case of \cite[Problem~4]{Fra-Hom_Triples}.  

\begin{problem*}
	Let $(X,\mathcal{X}, \mu, T_n)$ be a finitely generated multiplicative action, 
	and let $F, G, H \in L^\infty(\mu)$. If either $F$ or $G$ is 
	orthogonal to the pretentious factor of the system,\footnote{Equivalently, 
		as shown in \cite[Proposition~5.6]{Fra-Hom_Triples}, the spectral measure of 
		this function is entirely supported on aperiodic completely multiplicative 
		functions.} then  
	$$  
	\lim_{N\to\infty} \cesE_{1\leq n<m\leq N}  
\int
	T_{mn}F \cdot T_{m^2-n^2}G \cdot T_{m^2+n^2}H \, d\mu= 0. 
	$$  
\end{problem*}  
If one of the functions is constant, it follows from  
\cite{Fra-Host-structure-multiplicative} that the problem admits a positive 
answer for general actions, without assuming finite generation.  
However, the action by dilations  described after 
Conjecture~\ref{RadoErgodic} shows that, in general, finite generation cannot 
be omitted unless the orthogonality assumption is strengthened.

\subsection{Notational conventions}\label{SS:Notation}
We let $\N:=\{1, 2, \ldots\}$, $\N_0:=\{0, 1, 2, \ldots\}$,  $\Q_{>0}:=\Q\cap  (0,+\infty)$,  $\mathbb{S}^1$  
be the unit circle in $\C$, $\mathbb{U}$ the closed complex unit disk,  
and $\T:=\R/\Z$ the torus. For $x\in \R$, we set $e(x):=e^{2\pi i x}$. For  
$z\in \C$, we denote by $\Re(z)$ the real part of $z$. For $N\in \N$, we  let $[N]:=\{1, \dots, N\}$.  

We use $\P$ for the set of prime numbers, and throughout, $p$ denotes a  
prime. For $p \in \P$, $\ell \in \N_0$, and  $a\in \Z\setminus\{0\}$, we say that  
$p^{\ell}$ fully divides $a$, written $p^{\ell} \parallel a$, if  
$p^{\ell} \mid a$ and $p^{\ell+1} \nmid a$. We denote this exponent $\ell$  
by $\theta_p(a)$. The notation $(a_1,\ldots, a_k)$ stands for the greatest  
common divisor of $a_1,\ldots, a_k$. The inverse of an element $a$ in  
$\Z_{m}^{\times}$ is denoted by $a^{-1}$.  

If $A$ is a finite non-empty set and $a\colon A\to \C$, we  
define  
$$  
\E_{n\in A}\, a(n) := \frac{1}{|A|}\sum_{n\in A} a(n).  
$$  

For two functions $f$ and $g$, if there exists a constant $C>0$ such  
that $|f(X)| \leq C |g(X)|$ for all $X$ sufficiently large, we write  
$f(X) = \Oh(g(X))$ or $f(X) \ll g(X)$. When the constant $C$ depends on  
a quantity $K$, we write $f(X) = \Oh_{K}(g(X))$ or  
$f(X) \ll_{K} g(X)$.   Given  a collection of variables $X_1, \dots, X_r$ and $l_1, \dots, l_r \in  
\R \cup \{\infty, -\infty\}$, we write  
$\oh_{X_1, \dots, X_s, X_{s+1}\to l_{s+1}, \dots, X_r \to l_r}(1)$ for an  
error term $\e(X_1, \dots, X_r)$ such that for each $X_1, \dots, X_s$  
we have  
$$  
\lim_{X_{s+1}\to l_{s+1}} \limsup_{X_{s+2}\to l_{s+2}} \dots  
\limsup_{X_{r}\to l_{r}} |\e(X_1, \dots, X_r)| = 0.  
$$  

Finally, we abbreviate  
$$  
\lim_{X_{1}\to l_1} \limsup_{X_{2}\to l_2} \dots \limsup_{X_{r}\to l_r}  
\quad \text{by} \quad \lim_{X_1, \dots, X_r}  
$$  
when the value of the expression is zero; 
and 
$$  
\liminf_{X_{1}\to l_1} \liminf_{X_{2}\to l_2} \dots \liminf_{X_{r}\to l_r}  
\quad \text{by} \quad \inflim_{\, X_1, \dots, X_r}.  
$$

\subsection{Acknowledgements}
{  The authors would like to thank James Leng for pointing out
an error in a previous version of the manuscript, which led to a modification of the weight
sequences used in the proofs. The authors are also thankful to the anonymous referee for several useful comments and suggestions.}

\section{Background and preliminary results}\label{S:Background}
In this section we collect notions and results from number theory and ergodic 
theory that are going to be useful for us later on.

\subsection{Multiplicative averaging}
In our setting, a binary quadratic form is a homogeneous polynomial $P(m,n)$
of degree $2$ with integer coefficients. 
Recall that a binary quadratic form 
$P(m,n)=\alpha m^2 + \beta  mn + \gamma n^2\in \Z[m,n]$, where $\alpha,\beta, \gamma\in \Z$,  is
irreducible if and only if its discriminant 
$\Delta=\beta^2 -4 \alpha \gamma $ is not a perfect square.
Following \cite{fra-klu-mor-2}, given
a binary quadratic form $P$ and $r \in \N$, we define
\begin{equation}\label{omega_def}
    \omega_{P}(r):=|\{ n \in \Z_r: P(1,n)\equiv 0 \pmod{r} \}|.
\end{equation}
In \cite{fra-klu-mor-2}, $\omega_P$ is defined using 
$P(n,1)$ instead of $P(1,n)$, but the same results hold because of symmetry. 
From \cite[Lemma 4.3]{fra-klu-mor-2} we know that if $P$ is irreducible, 
then for all but finitely many $p\in \P$, $\omega_P(p)=0$ or $2$, and 
$\sum_{p\in \P} \frac{\omega_P(p)}{p} =+ \infty.$

We will also need the notion of multiplicative \Folner{} sequences. 
Given $r,K\in \N$ with $K>r$, let 
\begin{equation}\label{eq_foln_def_1}
    \Phi_{r}:=\Big\{ \prod_{p\leq r} p^{k_p} \colon 
    r< k_p \leq 3r/2\Big \} \:\:
\text{ and } \:\:
    \Phi_{r,K}:=\Big\{ \prod_{r<p\leq K} p^{k_p} \colon 
    K< k_p \leq 3K/2\Big \}.
\end{equation}
In addition, given a binary quadratic form $P$ and $r,K \in \N$ with 
$K>r$, we denote by $\Phi_{r,K,P}$ the set 
\begin{equation}\label{eq_foln_def_2}
    \Phi_{r,K, P}:=
    \Bigg\{ \prod_{\substack{r< p\leq K\\ \omega_P (p)>0}} 
    p^{k_p} \colon K<k_p \leq {3K}/{2} \Bigg\}.
\end{equation}
The sets $\Phi_r$ are asymptotically invariant as $r\to \infty$ under dilation, 
while the sets 
$\Phi_{r,K}$ (respectively $\Phi_{r,K, P}$) are asymptotically invariant as 
$K\to \infty$ under dilation by $p\in \P$ with $p>r$ (respectively with $p>r$
such that $\omega_P (p)>0$).

For each $L\in \N$, let 
\begin{equation}\label{Q_l_def}
    Q_L:= \prod_{p\leq L} p^{2L}.
\end{equation}

\subsection{Background on multiplicative functions}\label{SS:mf}
Recall that a function $f \colon \N \to \C$ is called  
\emph{multiplicative} if $f(mn) = f(m) \cdot f(n)$ whenever  
$(m,n) = 1$, and \emph{completely multiplicative} if the same identity  
holds for all $m,n \in \N$.  
The periodic completely multiplicative functions are precisely the  
{\em Dirichlet characters}, which we typically denote by $\chi$.  For our context, we  define the {\em  conductor} of $\chi$  to be its smallest period. 
We let  
\begin{equation*}
	\M := \{ f \colon \N \to \mathbb{S}^1 \; : \; f \text{ is completely multiplicative} \} 
\end{equation*}
and endow $\M$ with the topology of pointwise  
convergence, under which $\M$ is a compact metric space.

Following the terminology of Granville and Soundararajan  
\cite{Granville-Sound-book}, given two multiplicative functions  
$f, g \colon \N \to \U$, we define their \emph{pretentious distance}  
between two real numbers $A,B$ by  
\begin{equation}\label{E: definition of pretentious distance}
	\D(f,g;A,B)
	:= \Big( \sum_{A < p \leq B} 
	\frac{1}{p}\big(1 - \Re \big( f(p)\,  \overline{g(p)} \big)
	\Big)^{1/2}.
\end{equation}
We set $\D(f,g) := \lim\limits_{X \to \infty} \D(f,g; 1, X)$, and say  
that $f$ \emph{pretends} to be $g$, denoted $f \sim g$, if  
$\D(f,g) < +\infty$.  
\begin{definition} A multiplicative function $f$ is called \emph{pretentious} if there  
exist a Dirichlet character $\chi$ and a real number $t$ such that  
$f \sim \chi \cdot n^{it}$.  
We denote by $\M_p$ the set of pretentious multiplicative functions in $\M$.  
\end{definition}
By \cite[Lemma~3.6]{Fra-Klu-Mor}, the set $\M_p$ is a Borel subset of $\M$.   

Given $L, N \in \N$ with $L < N$, a binary quadratic form $P$, and a  
multiplicative function $f$, we define  
\begin{equation}\label{def_F_N}
	F_N(f,L) := \sum_{L < p \leq N} 
	\frac{1}{p} \, \big(f(p)\, \overline{\chi(p)} \, p^{-it} - 1\big)
\end{equation}
and  
\begin{equation}\label{def_G_P,N}
	G_{P,N}(f,L) := \sum_{L < p \leq N} 
	\frac{\omega_P(p)}{p}\,  \big( f(p) \, \overline{\chi(p)}\,  p^{-it} - 1 \big).
\end{equation}

We will need the following concentration estimates for pretentious  
multiplicative functions, first established in \cite{Klu-Man-Po-Ter}  
and subsequently adapted and applied in numerous works  
(see, for example, \cite{cha-mou-tsin,fra-klu-mor-2,Fra-Klu-Mor}).
Proposition \ref{linear_conc} 
is obtained from \cite[Lemma 2.5]{Klu-Man-Po-Ter} by replacing its concentration exponent by $F_N(f,K)$
and accounting for the primes $p>K$ dividing $Q$.
\begin{proposition}{\cite[Lemma 2.5]{Klu-Man-Po-Ter}}\label{linear_conc}
	Let $f \colon \N \to \U$ be a multiplicative function such that  
	$f \sim \chi \cdot n^{it}$ for some $t \in \R$ and some Dirichlet  
	character $\chi$ of period $q$.  
	Let $Q, K \geq 1$ be integers with  
	$\prod_{p \leq K} p \mid Q$, $q \mid Q$.
	%and all prime factors of $Q$  
	%lying in the interval $[1, K]$.  
	Then, for any $a \in \Z$ with $(a, Q) = 1$, we have  
	\begin{multline}\label{lin_conc_eq_1}
		\limsup_{N \to \infty} 
		\cesE_{n \in [N]}  
		\big| f(Qn+a) - \chi(a) \cdot (Qn)^{it}  
		\cdot \exp(F_N(f,K)) \big| \\
		\ll \D(f, \chi \cdot n^{it}, K, \infty) + K^{-1/2}+ \sum_{p>K, \, p\mid Q}\frac{1}{p},
	\end{multline}
	where $\D$ and $F_N(f,K)$ are as in  
	\eqref{E: definition of pretentious distance} and \eqref{def_F_N},  
	and the implicit constant is absolute.
\end{proposition}
\begin{example}[$1$-pretentious oscillatory multiplicative function]
If  $f(p) := e(1 / \log\log p)$ for all  
	$p \in \P$, then $f \sim 1$, but the imaginary part of the  
	series $\sum_{p \in \P} \frac{1}{p} \big( 1 - f(p) \big)$ diverges.  
	Consequently, the term $\exp\!\big(F_N(f,K)\big)$ in~\eqref{lin_conc_eq_1} is  
	oscillatory.
\end{example}

\begin{proposition}{\cite[Proposition 4.9]{fra-klu-mor-2}}\label{quad_conc}
	Let $P \in \Z[m,n]$ be an irreducible binary quadratic form  and 
	let $f \colon \Z \to \U$ be an even multiplicative function such that 
	$f \sim \chi \cdot n^{it}$ for some $t \in \R$ and some
	Dirichlet character $\chi$ of period $q$. 
	Let $Q, K \geq 1$ be integers, suppose that all  the prime factors of $Q$ 
	belong to $[1,K]$, and let $K$ be sufficiently large, depending only 
	on $P$, and so that
	$\D(f,\chi \cdot n^{it}, K, \infty) \leq 1$. Then 
	\begin{multline}\label{quad_conc_eq_1}
		\limsup_{N\to \infty} \cesE _{m,n\in [N]} 
		1_{P(Qm+a,Qn+b)>0 } \, 1_{P(m,n)>0 } \,
        |f(P_c (Qm+a, Qn+b)) - \\
		\chi(P_c(a,b)) \cdot (P_c(Qm, Qn))^{it} \cdot \exp(G_{P,N}(f,K)) |
		\ll_{P} \D(f, \chi \cdot n^{it},K,\infty) + K^{-1/2},
	\end{multline}
	where in the previous
	$c:=(P(a,b),Q)$, $P_c:=P/c$, and $a,b$ are such that $q$ and $\prod_{p\leq K}p$
	divide $Q/c$. Finally,  $\D$ and  $G_{P,N}$ are as in   
	\eqref{E: definition of pretentious distance} and \eqref{def_G_P,N}, and the implicit constant only 
	depends on $P$. 
\end{proposition}
\begin{rmk}\label{uniform_in_Q}
	Note that in \eqref{lin_conc_eq_1}, when all prime factors of $Q$ lie in $[1,K]$,  and in \eqref{quad_conc_eq_1}, the bounds 
	on the right-hand side depend only on $K$ and not on $Q$. Therefore, 
	if $\mathcal{C}_K$ is a finite collection of $Q$'s and, for each 
	$Q \in \mathcal{C}_K$, $a_Q$ is such that \eqref{lin_conc_eq_1} holds
	(respectively $a_Q, b_Q$ are such that \eqref{quad_conc_eq_1} holds), 
	then in \eqref{lin_conc_eq_1} (respectively
	\eqref{quad_conc_eq_1}) one may replace  
    ``$\limsup_{N\to \infty} \cesE_{n\in [N]}$'' by  
	``$\limsup_{N\to \infty} \max_{Q \in \mathcal{C}_K} \cesE_{n\in [N]}$''
	(respectively the limit ``$\limsup_{N\to \infty} \cesE_{m,n\in [N]}$'' by  
	``$\limsup_{N\to \infty} \max_{Q \in \mathcal{C}_K} \cesE_{m,n\in [N]}$''). 
\end{rmk}

\subsection{Spectral measures}\label{sec: spectral measures}
Note that $\mathbb{S}^1$-valued completely multiplicative functions on 
$\Q_{>0}$ are uniquely determined by their values on $\N$, so we may 
identify $\M$ with the Pontryagin dual of the discrete group 
$(\Q_{>0}, \times)$.

Let $(X, \CX, \mu, T_n)$ be a multiplicative action and let $F \in L^2(\mu)$. 
The map 
$$
(\Q_{>0}, \times) \to \C, \quad r/s \mapsto \int_X T_r F \cdot T_s \overline{F} \, d\mu
$$
is well-defined and positive definite. By the Bochner-Herglotz theorem,  
there exists a finite Borel measure $\sigma_F$ on $\M$  
(called the \emph{spectral measure} of $F$) such that  
$$
\int_X T_r F \cdot T_s \overline{F} \, d\mu 
= \int_{\M} f(r) \cdot \overline{f(s)} \, d\sigma_F(f) 
\quad \text{for all } r,s \in \Q_{>0}.
$$

From the above, we can easily deduce the following identity, which we will use frequently:
\begin{equation}\label{spec_eq_1}
	\Bigg\lVert \sum_{k=1} ^{l} c_k T_{r_k} F \Bigg\rVert_{L^2(\mu)}=
	\Bigg\lVert \sum_{k=1} ^{l} c_k f({r_k}) \Bigg\rVert_{L^2(\sigma_F(f))}
\end{equation}
	 for all  $l\in \N$, $c_1, \ldots, c_l \in \C$,
	$r_1, \ldots, r_l \in \Q_{>0}$, where, for $f\in \M$ and $r=m/n \in \Q_{>0}$,
$f(r)=f(m)/f(n)=f(m)\overline{f(n)}$.
From \eqref{spec_eq_1}, we immediately obtain that for all $r\in \Q_{>0}$,
\begin{equation}\label{spec_eq_2}
	\lVert T_{r} F -F \rVert_{L^2(\mu)}=
	\lVert f({r}) - 1 \rVert_{L^2(\sigma_F(f))}.
\end{equation}
\subsection{Factors of multiplicative actions}\label{SS:factors}
In the proof of our main result, we will require  
that certain  subsets of $L^{\infty}(\mu)$ form factors  
of the original system { $(X, \X, \mu, T_n)$}. 
Recall that $\mathcal{M}_p$ is the set of pretentious multiplicative functions. We denote by $\mathcal{A}$  
the set of Archimedean characters, namely
\begin{equation}\label{archim_def}
    \A:=\{n^{it}: t\in \R\},
\end{equation}
and by $\M_{p,\textup{fin.supp.}}$ the set 
\begin{multline}\label{def_pretent_conc}
    \M_{p,\textup{fin.supp.}}:= \{ f\in \M_p \colon \text{ there are } \chi, t \text{ such that } f(p)=\chi(p) \, p ^{it} \\
    \text{ for all but finitely many primes }p\}.
\end{multline}
Finally, given an irreducible binary quadratic form $P$, let 
\begin{multline}\label{def_pretent_conc_1}
    \M_{p,\textup{fin.supp.}, P}:= \{ f\in \M_p \colon \text{ there are } \chi, t \text{ such that } f(p)=\chi(p) \, p ^{it} \\
    \text{ for all but finitely many 
    primes }p \text{ with }  \omega_P(p)>0 \}.
\end{multline}
\begin{rmk}\label{useful_rmk_1}
	If $f\in \M_{p,\textup{fin.supp.}, P}$, then there exist
	$\chi_1, \chi_2, t_1, t_2$ such that $f\sim \chi_1 \cdot n^{it_1}$ and
	$f(p)=\chi_2 (p)\,  p^{it_2}$ for all but finitely many primes $p$ with
	$\omega_P(p)>0$. By \cite[Lemma 4.5]{fra-klu-mor-2}, this implies that
	$t_1 = t_2$ and $\chi_1(p) = \chi_2(p)$ for all but finitely many such $p$.
	Consequently, we have $f\sim \chi_1\cdot  n^{it_1}$ and
	$f(p) = \chi_1 (p)\, p^{it_1}$ for all but finitely many primes $p$ with
	$\omega_P(p)>0$.
\end{rmk}
The set $\A$ is a Borel subset of $\M$, since it is a countable union of compact
sets. We will now prove that $\M_{p,\textup{fin.supp.}}$ and 
$\M_{p,\textup{fin.supp.}, P}$ are also Borel subsets of $\M$.

\begin{lemma}
	The sets $\M_{p,\textup{fin.supp.}}$ and 
	$\M_{p,\textup{fin.supp.}, P}$ are Borel subsets of $\M$.
\end{lemma}

\begin{proof}
	We first prove that $\M_{p,\textup{fin.supp.}}$ is Borel.  
	Given $n, k \in \N$ and a Dirichlet character $\chi$, 
	consider the map  
	$$
	H_{n,k,\chi}\colon \M \to [0,+\infty), \quad 
	H_{n,k,\chi}(f):=\inf_{t\in \Q \cap [-k,k]} 
	\sum_{p\geq n} \frac{1}{p^2} \, |f(p)-\chi(p) \, p^{it}|.
	$$
	The function $H_{n,k,\chi}$ is Borel measurable, since it is the 
	countable infimum of measurable functions.\footnote{Each series in this 
		expression is a Borel measurable function with respect to $f$, as it is the pointwise 
		limit of continuous functions.}  
	It follows that  
	$$
	E_{n,k,\chi}:=\{f\in \M\colon H_{n,k,\chi}(f)=0\}
	$$
	is Borel measurable.  
	For each $f\in \M$, the map $\R \to [0,+\infty)$ given by 
	$$
	t\mapsto \sum_{p\geq n} \frac{1}{p^2}\,  |f(p)-\chi(p) 
\, p^{it}|
	$$
	is continuous. Using this observation, it is straightforward to verify 
	that  
	$$
	\M_{p,\textup{fin.supp.}}=\M _p \cap \Big( \bigcup_{n,k\in \N}
	\bigcup_{\chi} E_{n,k,\chi}\Big),
	$$
	where the innermost union is taken over all Dirichlet characters $\chi$.  
	Since the set of Dirichlet characters is countable, we deduce that 
	$\M_{p,\textup{fin.supp.}}$ is Borel measurable.  
	
	If, in the previous argument, we instead define $H_{n,k,\chi}(f)$ by  
	$$
	H_{n,k,\chi}(f):=\inf_{t\in \Q \cap [-k,k]} 
	\sum_{p\geq n} \frac{\omega_P(p)}{p^2} |f(p)-\chi(p) p^{it}|,
	$$
	then the same reasoning shows that 
    $\M_{p,\textup{fin.supp.},P}$ is 
	also Borel measurable.  
	This completes the proof.
\end{proof}

Given a pretentious multiplicative action $(X, \X, \mu, T_n)$, consider the 
subsets $X_{\A}$, \\ 
$X_{p,\textup{fin.supp.}}$ and $X_{p,\textup{fin.supp.},P}$ of $L^{2}(\mu)$ defined by 
\begin{equation}\label{def_X_A}
X_{\A}:=\{ F \in L^{2}(\mu)\colon \sigma_F \text{ is supported on } \A\},
\end{equation}
\begin{equation}\label{def_X_p_fin_supp}
X_{p,\textup{fin.supp.}}:=\{ F \in L^{2}(\mu)\colon  \sigma_F \text{ is supported on } \M_{p,\textup{fin.supp.}}\},
\end{equation}
and
\begin{equation}\label{def_X_p_fin_supp_P}
X_{p,\textup{fin.supp.},P}:=\{ F \in L^{2}(\mu)\colon \sigma_F \text{ is supported on } \M_{p,\textup{fin.supp.},P}\}.
\end{equation}

We will prove that $X_{\A}, X_{p,\textup{fin.supp.}}$, and  
$X_{p,\textup{fin.supp.},P}$ define factors $\X_{\A}, \X_{p,\textup{fin.supp.}},$ and $\X_{p,\textup{fin.supp.},P}$ of the original system.
In view of \cite[Lemma 4.6 (i)-(iii)]{Fra-Hom_Triples}  we have that 
$X_{\A}, X_{p,\textup{fin.supp.}}$, and $X_{p,\textup{fin.supp.},P}$ 
are closed $T_n$-invariant subspaces of $L^2(\mu)$.
It remains to show that $X_{\A} \cap L^{\infty}(\mu), X_{p,\textup{fin.supp.}}\cap L^{\infty}(\mu)$, and  
$X_{p,\textup{fin.supp.},P}\cap L^{\infty}(\mu)$ are conjugation closed 
subalgebras. 

One easily sees that they are conjugation closed, so
to prove that they define factors, it remains to show that for 
$\mathcal{H}\in \{ X_{\A}\cap L^{\infty}(\mu), 
X_{p,\textup{fin.supp.}}\cap L^{\infty}(\mu), 
X_{p,\textup{fin.supp.},P}\cap L^{\infty}(\mu) \}$, the following implication holds
\begin{equation}\label{E:F12}
 F_1, F_2 \in \mathcal{H}  \implies F_1 \cdot F_2 \in \mathcal{H}.
\end{equation}

For $r,K \in \N$ with $r<K$ and $Q \in \Phi_{r,K}$, we use the Chinese 
remainder theorem to pick $v=v(r, K, Q)\in \N$ such that
\begin{itemize}
    \item $v\equiv Q^{-1} \pmod {p^r}$ for $p\leq r$,
    \item $v\equiv 1 \pmod {p}$ for $r<p\leq K$.
\end{itemize}
Then it is easy to see that 
\begin{equation}\label{eq_cong_lw_1}
    \big(v, \prod_{p\leq K} p \big)=1 \: \text{ and } \: 
    v\equiv Q^{-1}  \! \!\!\! \pmod{\prod_{p\leq r} p^r}.
\end{equation}
On the other hand, given an irreducible binary quadratic form $P$,
$r, K\in \N$ with $r<K$ and $Q\in \Phi_{r, K, P}$,
use the Chinese remainder theorem to pick $u=u(r, K, Q,P)\in \N$ with
\begin{itemize}
    \item $u\equiv Q^{-1} \pmod {p^r}$ for $p\leq r$,
    \item $u\equiv 1 \pmod {p}$ for $r<p\leq K$.
\end{itemize}
Similarly, as before, we have 
\begin{equation}\label{eq_cong_lw_2}
    \big(u, \prod_{p\leq K} p \big)=1 \: \text{ and } \: 
    u\equiv Q^{-1} \! \! \! \!  \pmod{\prod_{p\leq r} p^r}.
\end{equation}

To deduce the implication \eqref{E:F12}, we will use the following concentration
estimates for functions in $L^2(\mu)$.

\begin{proposition}\label{P: conc_L_2_factors}
{Let $(X, \X, \mu, T_n)$ be a pretentious multiplicative action,} let $F\in L^{2}(\mu)$, let $P$ be an irreducible binary 
quadratic form, and let $\Phi_r$, $\Phi_{r,K}$, 
$\Phi_{r,K,P}$, $Q_K$
be as in 
\eqref{eq_foln_def_1}, \eqref{eq_foln_def_2},
\eqref{Q_l_def}.
Then 
    \begin{enumerate}
        \item \label{factors_i} $ F \in X_{\A}$ if and only if  
            \begin{equation}\label{eq_arch}
        \lim_{r\to \infty} \limsup_{N\to \infty} \max_{Q \in \Phi_r} \cesE_{n\in [N]}
        \Big \lVert T_{\frac{Qn+1}{Qn}}F -F
         \Big \rVert _{L^2(\mu)}=0.
    \end{equation}
    \item \label{factors_ii} $ F \in X_{p,\textup{fin.supp.}}$ if and only if
    \begin{equation}\label{eq_fin_sup_factor}
        \lim_{r\to \infty} \limsup_{K\to \infty} \limsup_{N\to \infty} 
        \max_{Q \in \Phi_{r,K}} \cesE_{n\in [N]}
        \Big \lVert T_{\frac{Q(Q_K Q n+v(r,K,Q))}{Q_K Q^2 n+1}}F -F
         \Big \rVert _{L^2(\mu)}=0.
    \end{equation}
    \item \label{factors_iii} $ F \in X_{p,\textup{fin.supp.},P}$ if and only if
    \begin{equation}\label{eq_fin_sup_P_factor}
        \lim_{r\to \infty} \limsup_{K\to \infty} \limsup_{N\to \infty} 
        \max_{Q \in \Phi_{r,K,P}} \cesE_{n\in [N]}
        \Big \lVert T_{\frac{Q(Q_K Q n+u(r,K,Q,P))}
        {Q_K Q^2 n+1}}F -F
         \Big \rVert _{L^2(\mu)}=0.
         \end{equation}
       \end{enumerate}
\end{proposition}

\begin{proof}
In what follows,  $\D$ and $F_N(f,K)$  are always as in
\eqref{E: definition of pretentious distance} and    \eqref{def_F_N} respectively.
Since the proofs of \eqref{factors_i}, \eqref{factors_ii}, 
\eqref{factors_iii}
are similar, we give a detailed proof of 
\eqref{factors_iii}, and then we sketch the proofs of 
\eqref{factors_i}, \eqref{factors_ii}.

Let $F\in L^{2}(\mu)$. In the following, we simply write $u$ instead of 
$u(r,K,Q,P)$.
Given $n,r,K\in \N$ and $Q\in \Phi_{r,K,P}$, using \eqref{spec_eq_2} we have
\begin{equation}\label{factors_eq_1}
    \Big \lVert T_{\frac{Q(Q_K Q n+u)}{Q_K Q^2 n+1}}F -F  
    \Big \rVert _{L^2(\mu)} ^2
    =
    \int_{\M _p} \big | {f(Q)}f(Q_K Qn+u)\overline{f(Q_K Q^2 n+1)} - 1
    \big | ^2 d \sigma_F(f). 
\end{equation}
If $F\in X_{p, \textup{fin.supp.},P}$, 
then $\sigma_F$ is supported on $\M_{p, \textup{fin.supp.},P}$, so using 
\eqref{factors_eq_1}, Cauchy-Schwarz and Fatou's lemma one sees that in 
order to prove 
\eqref{eq_fin_sup_P_factor}, it suffices to prove that for each 
$f\in \M_{p, \textup{fin.supp.},P}$ we have
\begin{equation}\label{factors_eq_3}
    \lim_{r\to \infty} 
    \limsup_{K\to \infty} 
    \limsup_{N\to \infty} 
    \max_{Q \in \Phi_{r,K,P}} \cesE_{n\in [N]}
    | {f(Q)}f(Q_K Qn+u)\overline{f(Q_K Q^2 n+1)} - 1
    \big |=0.
\end{equation}
Let $f\in \M_{p, \textup{fin.supp.},P}$. From \eqref{def_pretent_conc_1}
and Remark \ref{useful_rmk_1} there exist a Dirichlet character $\chi$, 
$t\in \R$, and $r_0\in \N$ such that 
$f\sim \chi \cdot n^{it}$ and $f(p)= \chi(p) \cdot p^{it}$ for all primes 
$p>r_0$ with $\omega_P(p)>0$.

Let $q_{\chi}$ be the conductor of $\chi$, and suppose  that $r>r_0$ and 
$r$ is large enough so that $q_{\chi} \mid \prod_{p\leq r} p^r$.
From \eqref{eq_cong_lw_2} we deduce that $(u,Q_KQ)=1$ and 
$u\equiv Q^{-1}\pmod{q_{\chi}}$, so using Proposition \ref{linear_conc}  
and Remark \ref{uniform_in_Q} we infer that 
\begin{multline}\label{factors_eq_2}
    \limsup_{N\to \infty} \max_{Q \in \Phi_{r,K,P}} \cesE_{n\in [N]}
    | {f(Q)}f(Q_K Qn+u)\overline{f(Q_K Q^2 n+1)} - 1
    \big | \ll 
     \\
    \limsup_{N\to \infty} \max_{Q \in \Phi_{r,K,P}} \cesE_{n\in [N]}
    \Big | {f(Q)}\, \overline{\chi(Q)} \, (Q_K Qn)^{it}  \,  
    \exp(F_N(f,K))
    {(Q_K Q^2 n)^{-it}}\cdot\\ \overline{\exp(F_N(f,K))}  - 1 \Big |
     + \oh_{K\to \infty}(1) = 
    |\exp(-2\D^2(f, \chi \cdot n^{it}; K, \infty))-1| + 
    \oh_{K\to \infty}(1),
\end{multline}
where in the last equality we used that every prime dividing $Q$ is 
greater than $r>r_0$ and has $\omega_P(p)>0$, so 
$f(Q)=\chi(Q)Q^{it}$, and that $2\Re(F_N(f,K))=-2\D^2(f, \chi \cdot n^{it};K, N)$, which converges to  
$-2\D^2(f, \chi \cdot n^{it}; K, \infty)$ as $N\to \infty$. Since 
$\D^2(f, \chi \cdot n^{it};1, \infty)< \infty$, taking 
$K\to \infty$ in \eqref{factors_eq_2} we see that \eqref{factors_eq_3}
holds,
and therefore $F$ satisfies \eqref{eq_fin_sup_P_factor}.

For the opposite direction,  suppose that $F$ satisfies 
\eqref{eq_fin_sup_P_factor}. Using \cite[Lemma 4.6 (v)]{Fra-Hom_Triples}
we can write $F=F_1 + F_2$, with $\sigma_{F_1}$ supported on 
$\M_{p, \textup{fin.supp.},P}$ and $\sigma_{F_2}$ supported on 
$\M_p \setminus \M_{p, \textup{fin.supp.},P}$.
Then using the reverse triangle inequality we have 
\begin{multline*}
 \Big \lVert \cesE_{Q \in \Phi_{r,K,P}} \cesE_{n\in [N]}
        T_{\frac{Q(Q_K Q n+u)}{Q_K Q^2 n+1}}F -F  
         \Big \rVert _{L^2(\mu)}
 \geq    
 \lVert F_2  \rVert_{L^2(\mu)} \\ - 
\Big \lVert \cesE_{Q \in \Phi_{r,K,P}} \cesE_{n\in [N]}
         T_{\frac{Q(Q_K Q n+u)}{Q_K Q^2 n+1}}F_1  
         -F_1  \Big \rVert _{L^2(\mu)} - \Big \lVert
         \cesE_{Q \in \Phi_{r,K,P}} \cesE_{n\in [N]}
         T_{\frac{Q(Q_K Q n+u)}{Q_K Q^2 n+1}}F_2
         \Big \rVert _{L^2(\mu)},
\end{multline*}
which in turn implies that 
\begin{multline}\label{factors_eq_4}
\lVert F_2  \rVert_{L^2(\mu)} \leq  \max_{Q \in \Phi_{r,K,P}} 
\cesE_{n\in [N]}
\Big \lVert  T_{\frac{Q(Q_K Q n+u)}{Q_K Q^2 n+1}}F  -F
         \Big \rVert _{L^2(\mu)} \\ + 
       \max_{Q \in \Phi_{r,K,P}} \cesE_{n\in [N]}
\Big \lVert  T_{\frac{Q(Q_K Q n+u)}{Q_K Q^2 n+1}}F_1  -F_1
         \Big \rVert _{L^2(\mu)} 
         +        
\Big \lVert \cesE_{Q \in \Phi_{r,K,P}} \cesE_{n\in [N]}
         T_{\frac{Q(Q_K Q n+u)}{Q_K Q^2 n+1}}F_2 
        \Big \rVert _{L^2(\mu)}.
\end{multline}
Since $F_1 \in X_{p, \textup{fin.supp.},P}$, from the already established direction we get that 
\begin{equation}\label{factors_eq_5}
    \lim_{r\to \infty} \limsup_{K\to \infty} \limsup_{N\to \infty} 
        \max_{Q \in \Phi_{r,K,P}} \cesE_{n\in [N]}
        \Big \lVert T_{\frac{Q(Q_K Q n+u)}{Q_K Q^2 n+1}}F_1 -F_1
         \Big \rVert _{L^2(\mu)}=0.
\end{equation}
For $F_2$, using \eqref{spec_eq_1} we infer that
\begin{multline}\label{factors_eq_8}
\Big \lVert \cesE_{Q \in \Phi_{r,K,P}} \cesE_{n\in [N]}
         T_{\frac{Q(Q_K Q n+u)}{Q_K Q^2 n+1}}F_2 
        \Big \rVert _{L^2(\mu)} ^2 \\ =
    \int_{\M_p \setminus \M_{p, \textup{fin.supp.},P}} 
    \Big | \cesE_{Q \in \Phi_{r,K,P}} \cesE_{n\in [N]}
    f(Q) f(Q_K Q n+u) \overline{f(Q_K Q^2 n+1)}
        \Big | ^2 d \sigma_{F_2}(f).
\end{multline}
Let $f\in \M_{p}\setminus \M_{p, \textup{fin.supp.},P}$ with
$f\sim \chi \cdot n^{it}$. Let 
$q_{\chi}$ be the conductor of $\chi$ and suppose that 
$r$ is large enough so that 
$q_{\chi} \mid \prod_{p\leq r} p^r$.
From \eqref{eq_cong_lw_2} we deduce that $(u,Q_KQ)=1$ and 
$u\equiv Q^{-1}\pmod{q_{\chi}}$, so using Proposition \ref{linear_conc}  
and Remark \ref{uniform_in_Q} we infer that 
\begin{multline}\label{factors_eq_6}
\cesE_{Q \in \Phi_{r,K,P}} \cesE_{n\in [N]}
    f(Q) \, f(Q_K Q n+u) \, \overline{f(Q_K Q^2 n+1)}
        \\ =
        \exp(2\Re(F_N(f,K)))\cdot 
        \cesE_{Q \in \Phi_{r,K,P}} 
    f(Q) \, Q ^{-it} \, \overline{\chi(Q)} + 
\oh_{K\to \infty, N\to \infty}(1).
\end{multline}
Since $f \notin \M_{p, \textup{fin.supp.},P}$,
there is $p_0 \in \P$ with 
$p_0>r$ so that $\omega_P(p_0)>0$ and
$f(p_0)\neq p_0 ^{it}\cdot \chi({p_0})$. Note that since 
$p_0>r$ and $q_{\chi} \mid \prod_{p\leq r}p^r$, we have $|\chi(p_0)|=1$,
and using that $\Phi_{r,K,P}$ is asymptotically invariant under 
dilation by $p_0$ as $K\to \infty$ we get that 
\begin{align*}
   f(p_0) \,  p_0 ^{-it} \,  \overline{\chi({p_0})}\, 
   \cesE_{Q \in \Phi_{r,K,P}} 
    f(Q) \, Q ^{-it} \, \overline{\chi(Q)}&=
    \cesE_{Q \in p_0 \Phi_{r,K,P}} 
    f(Q)\,  Q ^{-it} \, \overline{\chi(Q)} \\& =
    \cesE_{Q \in \Phi_{r,K,P}} 
    f(Q)\,  Q ^{-it}\,  \overline{\chi(Q)} + \oh_{K \to \infty}(1),
\end{align*}
which implies that 
\begin{equation}\label{factors_eq_7}
    \lim_{K\to \infty}  
    \cesE_{Q \in \Phi_{r,K,P}} 
    f(Q)\,  Q ^{-it} \, \overline{\chi(Q)}=0.
\end{equation}
From \eqref{factors_eq_6} and \eqref{factors_eq_7} we get that for 
$f\in \M_p \setminus \M_{p, \textup{fin.supp.},P}$ we have 
\begin{equation*}
   \lim_{r\to \infty} \limsup_{K\to \infty} 
   \limsup_{N\to \infty} \Big|
   \cesE_{Q \in \Phi_{r,K,P}} \cesE_{n\in [N]}
    f(Q)\,  f(Q_K Q n+u)\,  \overline{f(Q_K Q^2 n+1)}
   \Big|=0,
\end{equation*}
which in view of \eqref{factors_eq_8} implies
\begin{equation}\label{factors_eq_9}
    \lim_{r\to \infty} 
    \limsup_{K\to \infty} 
    \limsup_{N\to \infty} 
    \Big \lVert \cesE_{Q \in \Phi_{r,K,P}} \cesE_{n\in [N]}
         T_{\frac{Q(Q_K Q n+u)}{Q_K Q ^2 n+1}}F_2 
        \Big \rVert _{L^2(\mu)} =0.
\end{equation}
Taking $N\to \infty$, then $K\to \infty$ and finally $r\to \infty$ in 
\eqref{factors_eq_4} and using \eqref{factors_eq_5},
\eqref{factors_eq_9} and the fact that 
$F$ satisfies \eqref{eq_fin_sup_P_factor},
we get that $\lVert F_2 \rVert _{L^2(\mu)}=0$, so $F=F_1 \in X_{p, \textup{fin.supp.},P}$.
This concludes the proof of \eqref{factors_iii}. 

If in the previous one replaces $u(r,K,Q,P)$ by $v(r,K,Q)$ and 
$\Phi_{r,K,P}$ by $\Phi_{r,K}$, then \eqref{factors_ii} follows with
the same proof. 

Finally, we sketch the proof of \eqref{factors_i}.
If $F\in X_{\A}$, 
then $\sigma_F$ is supported on $\A$, so in order to establish
\eqref{eq_arch} it suffices to prove that for each $f\in \A$ we have 
\begin{equation}\label{factors_eq_10}
    \lim_{r\to \infty} \limsup_{N\to \infty} \max_{Q \in \Phi_r} \cesE_{n\in [N]}
        | f(Qn+1)\, \overline{f(Qn)} -1 |=0.
\end{equation}
For $f\in \A$ with $f=n^{it}$ we have 
\begin{equation*}
    \max_{Q \in \Phi_r} \cesE_{n\in [N]}
    | \overline{f(Qn)} \, f(Qn+1)-1  |=
    \max_{Q \in \Phi_r} \cesE_{n\in [N]}
    | (Qn)^{-it}\, (Qn+1)^{it}-1 |=\oh_{t, r, N\to \infty}(1),
\end{equation*}
and the result follows.

For the opposite direction,  suppose that $F$ satisfies 
\eqref{eq_arch}. Again, write $F=F_1 + F_2$, with $\sigma_{F_1}$ 
supported on $\A$ and $\sigma_{F_2}$ supported on $\M_p \setminus \A$.
From reverse triangle inequality we have 
\begin{multline}\label{factors_eq_11}
\lVert F_2  \rVert_{L^2(\mu)} \leq  \max_{Q \in \Phi_{r}} 
\cesE_{n\in [N]}
\Big \lVert  T_{\frac{Qn+1}{Qn}}F  -F
         \Big \rVert _{L^2(\mu)} + \\
       \max_{Q \in \Phi_{r}} \cesE_{n\in [N]}
\Big \lVert  T_{\frac{Qn+1}{Qn}}F_1  -F_1
         \Big \rVert _{L^2(\mu)} 
         +        
\Big \lVert \cesE_{Q \in \Phi_{r}} \cesE_{n\in [N]}
         T_{\frac{Qn+1}{Qn}}F_2 
        \Big \rVert _{L^2(\mu)}.
\end{multline}
Since $F_1 \in X_{\A}$, from the already established direction we get 
that 
\begin{equation}\label{factors_eq_12}
    \lim_{r\to \infty} \limsup_{N\to \infty} 
        \max_{Q \in \Phi_{r}} \cesE_{n\in [N]}
        \Big \lVert T_{\frac{Qn+1}{Qn}}F_1 -F_1
         \Big \rVert _{L^2(\mu)}=0.
\end{equation}
For $f\in \M_{p}\setminus \A$ with $f\sim \chi \cdot n^{it}$, 
using Proposition \ref{linear_conc}  
and Remark \ref{uniform_in_Q} we infer that 
\begin{multline*}
\cesE_{Q \in \Phi_{r}} \cesE_{n\in [N]}
     f(Qn+1) \, \overline{f(Qn)}
        \\ =
        \exp(F_N(f,r)) \Big(\cesE_{n\in [N]} \overline{f(n)}\, n^{it} \Big)
        \Big( \cesE_{Q \in \Phi_{r}}    \overline{f(Q)}\,  Q ^{it}  \Big) + 
\oh_{r\to \infty, N\to \infty}(1).
\end{multline*}
Since $f \notin \A$ there is $p_0 \in \P$ so that 
$f(p_0)\neq p_0 ^{it}$. Again, using that $\Phi_{r}$ is asymptotically 
invariant under dilation by $p_0$ as $r\to \infty$ we get that 
\begin{equation*}
    \lim_{r\to \infty}  
    \cesE_{Q \in \Phi_{r}} 
    \overline{f(Q)}\,  Q ^{it} =0.
\end{equation*}
Note that $|\exp(F_N(f,r))|=\exp(-\D^2(f,\chi \cdot n^{it};r, N))$, which 
converges to $1$ as $N\to \infty$ and then $r\to \infty$, and combining
with the previous we get that for $f\in \M_p \setminus \A$ we have 
\begin{equation*}
    \lim_{r\to \infty} \limsup_{N \to \infty}
   \Big|  \cesE_{Q \in \Phi_{r}} \cesE_{n\in [N]}
     f(Qn+1)\,  \overline{f(Qn)}  \Big| =0.
\end{equation*}
Since $\sigma_{F_2}$ is supported on $\M_p \setminus \A$, we 
infer that 
\begin{equation}\label{factors_eq_14}
\lim_{r\to \infty} \limsup_{N \to \infty}
\Big \lVert \cesE_{Q \in \Phi_{r}} \cesE_{n\in [N]}
         T_{\frac{Qn+1}{Qn}}F_2 
        \Big \rVert _{L^2(\mu)}=0 .
\end{equation}
Using \eqref{factors_eq_11}, \eqref{factors_eq_12},
\eqref{factors_eq_14}, and the fact that $F$ satisfies \eqref{eq_arch},
we get that $\lVert F_2 \rVert _{L^2(\mu)}=0$, so $F=F_1 \in X_{\A}$.
This concludes the proof of \eqref{factors_i}. 
\end{proof}

Finally, we have the following lemma, which concludes the proof that 
$X_{\A}, X_{p,\textup{fin.supp.}}$, $X_{p,\textup{fin.supp.},P}$ 
are factors. 

\begin{lemma}
For $\mathcal{H}\in \{ X_{\A}\cap L^{\infty}(\mu), 
X_{p,\textup{fin.supp.}}\cap L^{\infty}(\mu), 
X_{p,\textup{fin.supp.},P}\cap L^{\infty}(\mu) \}$, if $F_1, F_2 \in \mathcal{H}$, 
then $F_1 \cdot F_2 \in \mathcal{H}$.
\end{lemma}

\begin{proof}
Since the proofs are identical, we will only prove the case when 
$\mathcal{H}=X_{\A}\cap L^{\infty}(\mu)$. For $F_1, F_2 \in 
X_{\A}\cap L^{\infty}(\mu)$, $Q\in \Phi_r$ and $n\in \N$ we have 
\begin{align*}
   \Big \lVert T_{\frac{Qn+1}{Qn}}(F_1 \cdot F_2) - F_1 \cdot F_2 
         \Big \rVert _{L^2(\mu)} 
         \leq 
          &\lVert  F_1
         \rVert _{L^{\infty}(\mu)} 
        \Big \lVert T_{\frac{Qn+1}{Qn}} F_2  -  F_2 
         \Big \rVert _{L^2(\mu)} 
         + \\
          &\lVert  F_2
         \rVert _{L^{\infty}(\mu)} 
        \Big \lVert T_{\frac{Qn+1}{Qn}} F_1-  F_1 
         \Big \rVert _{L^2(\mu)}.
\end{align*}
In the previous, {averaging over $n \in [N]$,} taking the maximum with respect to $Q\in \Phi_r$,  
sending $N\to \infty$ and then $r\to \infty$, and using that 
$F_1, F_2$ satisfy \eqref{eq_arch}, we see that $F_1 \cdot F_2$ also 
satisfies \eqref{eq_arch}, which in view of Proposition \ref{P: conc_L_2_factors}
\eqref{factors_i} implies that $F_1 \cdot F_2 \in X_{\A}$. This 
concludes the proof. 
\end{proof}

\subsection{Polynomial visit times of Archimedean characters}

{In the proof of our main result, we will average over suitable 
sectors, so that if $P_j$, $j=1,2,3,$ are the homogeneous binary quadratic forms that parametrize the solutions 
of the equation $ax^2 +by^2=cz^2$, then $(P_l(m,n)/P_j(m,n))^{it}$
concentrate around suitable values. Recall that for a set $E\subset \N^2$, its 
upper and lower densities are defined respectively by 
\begin{equation}\label{density def}
    \overline{\textup{d}}(E):=\limsup_{N\to \infty} \frac{|E\cap [N]^2|}{N^2} \quad \text{and} \quad 
\underline{\textup{d}}(E):=\liminf_{N\to \infty} \frac{|E\cap [N]^2|}{N^2},
\end{equation}
where $[N]^2: = [N]\times [N]$, and if $\overline{\textup{d}}(E)= \underline{\textup{d}}(E)$, then                    
we say that the density of $E$ 
exists, and we denote it by $\textup{d}(E)$.

\begin{lemma}\label{L: S_d a=c}
Let $a,b \in \N$, let $P_1(m,n)= 2amn$, $P_2(m,n)=m^2-abn^2$, $P_3(m,n)=m^2+abn^2$, 
and for each $\delta>0$ consider the set 
\begin{equation}\label{def S_d a=c}
    S_\delta:= \Big\{ (m,n) \in \N ^2 \colon \frac{\delta e ^{- \delta /2}}{2a}  < 
    \frac{n}{m} < \frac{\delta e ^{\delta /2}}{2a} \Big\}.
\end{equation}
Then  $S_\delta$ has positive density, and if $\delta$ is small enough
depending on $a,b$, then for each $(m,n) \in S_\delta$ and $t\in \R$ we have
$P_j(m,n) >0$, $j=1,2,3$, and  
\begin{equation}\label{prop S_d a=c}
 \big| (P_1(m,n)/P_3(m,n)) ^{it} - \delta^{it} \big| \leq |t| \delta, 
 \quad \big| (P_2(m,n)/P_3(m,n)) ^{it} -1 \big| \leq |t| \delta.
\end{equation}
\end{lemma}

\begin{proof}
For any $\delta>0$, we have $\frac{\delta e ^{- \delta /2}}{2a} < \frac{\delta e ^{\delta /2}}{2a}$, so $ S_\delta$ is a non-empty 
sector, and therefore its density exists and is positive. For any $(m,n) \in \N^2$, $P_1(m,n), P_3(m,n)>0$, while 
if $\delta< \min \{ 2, \frac{a}{e \sqrt{ab}}\}$, then for each $(m,n)\in  S_\delta$ we have 
that $m> 2 \sqrt{ab} n$, so also $P_2(m,n) >0$. If we furthermore restrict 
$\delta < \min \{ \frac{1}{2}, \frac{a}{2b}\}$, then it is not too difficult to see that for $(m,n) \in  S_\delta$,
\begin{equation}\label{S_d useful eq}
    \Big|\log \frac{P_1(m,n)}{P_3(m,n)} - \log \delta \Big| \leq \delta \quad \text{and} \quad 
    \Big|\log \frac{P_2(m,n)}{P_3(m,n)}\Big| \leq \delta. 
\end{equation}
We then obtain that 
\begin{equation*}
    \big| (P_1(m,n)/P_3(m,n)) ^{it} - \delta^{it} \big|= \big| e^{it \log \frac{P_1(m,n)}{P_3(m,n)}} - e^{it \log \delta}\big| \leq 
    \Big|t \log \frac{P_1(m,n)}{P_3(m,n)} - t \log \delta \Big| \leq |t| \delta,
\end{equation*}
where in the first inequality we used that for all $x,y \in \R$, $|e^{ix} - e^{iy}| \leq |x-y|$, and in the 
second inequality we used \eqref{S_d useful eq}. Analogously, we have that 
\begin{equation*}
    \big| (P_2(m,n)/P_3(m,n)) ^{it} - 1 \big|= \big| e^{it \log \frac{P_2(m,n)}{P_3(m,n)}} - 1\big| \leq 
    \Big|t \log \frac{P_2(m,n)}{P_3(m,n)} \Big| \leq |t| \delta.
\end{equation*}
This concludes the proof of the lemma.
\end{proof}

\begin{lemma}\label{L: S_d a+b=c}
Let $a,b \in \N$, let $P_1(m,n):=m^2 + ab n^2$, $P_2(m,n):=m^2 -2bmn - ab n^2$, $P_3(m,n):= m^2 +2amn - ab n^2$, 
and for each $\delta>0$ consider the set 
\begin{equation}\label{def S_d a+b=c}
     S_\delta:= \Big\{ (m,n) \in \N ^2 \colon \frac{n}{m} < \frac{1- e^{-\delta}}{4ab} \Big\}.
\end{equation}
Then $ S_\delta$ has positive density, and for each $(m,n) \in  S_\delta$ and $t\in \R$ we have
$P_j(m,n) >0$, $j=1,2,3$, and  
\begin{equation}\label{prop S_d a+b=c}
 \big| (P_l(m,n)/P_j(m,n)) ^{it} - 1 \big| \leq |t| \delta \quad \text{for}\,\,\,  l, j\in \{1,2,3\}.
\end{equation}
\end{lemma}

The proof of this lemma is very similar to the proof of Lemma \ref{L: S_d a=c}, so it is omitted.
}

\subsection{Two useful estimates} 
In order to obtain not only positivity but also the asserted lower  
bounds in Theorem \ref{T:Main} (see Remark \ref{rmk_lower_bound}), 
we will use the following estimate.
\begin{proposition}[{\cite[Lemma 1.6]{chu-2-commuting}}]\label{chu-lemma}
    Let $(X, \X, \mu)$ be a probability space, let $\X_1, \dots, \X_{\ell} $ be sub-$\sigma$-algebras of $\X$ and $F\in L^{\infty}(\mu)$ be non-negative. Then
    \begin{equation*}
        \int_X F\cdot \E(F \mid \X_{1}) \cdots \E(F\mid \X_{\ell}) \, d \mu \geq \Big( \int_X F \, d \mu\Big) ^{\ell+1}.
    \end{equation*}
\end{proposition}
We will also frequently use the following elementary result.
\begin{lemma}[{\cite[Lemma 4.11]{Fra-Hom_Triples}}]\label{bilinear_lemma}
Let $V$ be a normed space, let $v: \Z \to V$ be a $1$-bounded even\footnote{This means that for all $n\in \Z$, $v(-n) = v(n)$.}
sequence, and $l_1, l_2\in \Z$, not both of them $0$. Let $l=|l_1|+|l_2|$
and suppose 
that for some $\e>0$ and for some $1$-bounded sequence $(v_{N})_{N\in \N}$ of elements in $V$ we have 
\begin{equation*}
    \limsup_{N\to \infty} \cesE_{n\in [N]} \lVert v(n) - v_N \rVert 
    \leq \e.
\end{equation*}
Then 
\begin{equation*}
   \limsup_{N\to \infty} \cesE_{m,n \in [N]} \lVert v({l_1m + l_2 n})
   -v_N \rVert \leq 4 l \e.
\end{equation*}
\end{lemma}

\section{Proof sketch for the Pythagorean equation}\label{Sec_proof_plan}
In this section, we provide a sketch of the proof of Theorem \ref{T:Main} in 
case $a=b=c=1$, which corresponds to the classical Pythagorean equation
\begin{equation}\label{proof plan eq 1}
	x^2 + y^2 =z^2.
\end{equation}
We use solutions of \eqref{proof plan eq 1} given in parametric form by 
\begin{equation}\label{proof plan eq 2}
	x=k\, 2mn , \quad y=k\, (m^2 -n^2), \quad z=k\, (m^2 + n^2), \quad k,m,n\in \N.
\end{equation}
If we let 
\begin{equation}\label{proof plan eq 4}
	P_1(m,n):=2mn, \quad P_2(m,n):=m^2 -n^2, \quad P_3(m,n):=m^2 + n^2,
\end{equation}
we see that the case $a=b=c=1$ of Theorem \ref{T:Main} follows from the next result. 
\begin{theorem}\label{proof plan thm 1}
	Let $(X,\CX,\mu,T_n)$  be a pretentious multiplicative action, let 
	$F\in L^{\infty}(\mu)$ with $F\geq 0$, 
	and let $P_1,P_2,P_3$ be as in 
	\eqref{proof plan eq 4}. Then for every $\e>0$ we have
	\begin{equation}\label{proof plan eq 3}
		\int_X T_{P_1(m,n)} F \cdot T_{P_2(m,n)} F \cdot T_{P_3(m,n)} F \, d\mu
		\geq \Big(\int_X F d \mu \Big)^3 -\e
	\end{equation}
	for a set of $(m,n)\in \N^2$, $m>n$, with positive upper density.
\end{theorem}
All notions used in this proof sketch are defined in Section~\ref{S:Background}. {Using $T$-invariance of $\mu$, 
we may rewrite the expression in the left hand side of \eqref{proof plan eq 3} as 
\begin{equation}\label{proof plan eq 10}
    \int_X F \cdot T_{\frac{P_1(m,n)}{P_3(m,n)}} F \cdot T_{\frac{P_2(m,n)}{P_3(m,n)}} F \, d\mu.
\end{equation}}

{\em Finitely generated case.} For simplicity, first suppose that the action $T_n$ is ergodic and finitely  
generated. We will  average over the grid $(Qm+1, Qn)$, where $Q$ is highly  
divisible, and apply the $Q$-trick (as in \cite{Fra-Klu-Mor}) to obtain Theorem  
\ref{proof plan thm 1}. More precisely, by \cite[Lemma~2.6]{Charamaras-multiplicative}
for each $F \in L^2(\mu)$, its  
spectral measure is supported on  
$$  
\M_{p,\mathrm{f.g.}} := \{ f \colon \N \to \mathbb{S}^1, \; \text{$f$ is  
	pretentious and finitely generated} \},  
$$  
where finitely generated means that $\{ f(p) \colon p \in \P \}$ is finite.  

For every $f \in \M_{p,\mathrm{f.g.}}$, if $f \sim \chi \cdot n^{it}$, then  
$t = 0$, while for the oscillating factors $F_N(f,L)$ and $G_{P,N}(f,L)$  
from \eqref{def_F_N} and \eqref{def_G_P,N}, we have  
$\exp(F_N(f,L)), \exp(G_{P,N}(f,L)) \to 1$ as $N\to \infty$ and then  $L \to \infty$; finite  
generation is essential here and the claimed properties are established in \cite[Lemmas~B.3 and B.4]{Charamaras-multiplicative}. These facts make it easier to handle the  
terms $T_{{P_j(m,n)}/{P_3 (m,n)}}F$ for $j \in \{1,2\}$.
{ Indeed, the preceding observations and
Propositions~\ref{linear_conc} and \ref{quad_conc} give
 concentration at $1$ for $f(Qm+1)$ and
$f(P_3(Qm+1,Qn))$, uniformly over $Q\in\Phi_K$,
as $N\to\infty$ and then $K\to\infty$.
Moreover,
\begin{align*}
	P_1(Qm+1,Qn)&=2Qn(Qm+1),\\
	P_2(Qm+1,Qn)&=(Q(m-n)+1)(Q(m+n)+1).
\end{align*}
Complete multiplicativity, 
%and the fact that each value of
%$m\pm n$ occurs at most $N$ times give the %corresponding
%scalar estimates for $P_1$ and $P_2$.
the spectral identity \eqref{spec_eq_1}, the Cauchy-Schwarz inequality,
and the bounded convergence theorem,  then give 
$$  
\lim_{K \to \infty} \limsup_{N \to \infty} \max_{Q \in \Phi_K}  
\cesE_{(m,n) \in [N]^2,\, m>n}  \,
\lVert T_{\frac{P_1(Qm+1, Qn)}{P_3(Qm+1, Qn)}} F -T_{2Qn} F \rVert_{L^2(\mu)} = 0,  
$$  
while 
$$  
\lim_{K \to \infty} \limsup_{N \to \infty} \max_{Q \in \Phi_K}  
\cesE_{(m,n) \in [N]^2,\, m>n}  \,
\lVert T_{\frac{P_2(Qm+1, Qn)}{P_3(Qm+1, Qn)}} F - F \rVert_{L^2(\mu)} = 0,  
$$  
where $(\Phi_K)_{K \in \N}$ is a multiplicative \Folner{} sequence as in  
\eqref{eq_foln_def_1}.  
Let $F_0 := F - \int_X F\, d\mu$. Using the spectral theorem again and the fact that multiplicative averages of non-trivial  
multiplicative functions vanish, we deduce  
$$\lim_{K \to \infty} \limsup_{N \to \infty}  
\lVert \cesE_{Q \in \Phi_K} \cesE_{(m,n) \in [N]^2,\, m>n} \, 
T_{2Qn} F_0 \rVert_{L^2(\mu)} = 0,
$$
which implies that 
$$  
\lim_{K \to \infty} \limsup_{N \to \infty}  
\lVert \cesE_{Q \in \Phi_K} \cesE_{(m,n) \in [N]^2,\, m>n}  \,
T_{\frac{P_1(Qm+1, Qn)}{P_3(Qm+1, Qn)}} F_0 \rVert_{L^2(\mu)} = 0.  
$$  
Combining these results we obtain  
$$  
\liminf_{K \to \infty} \liminf _{N \to \infty}  
\cesE_{Q \in \Phi_K} \cesE_{(m,n) \in [N]^2,\, m>n}  
\int_X F \cdot  T_{\frac{P_1(Qm+1,Qn)}{P_3(Qm+1,Qn)}} F \cdot T_{\frac{P_2(Qm+1,Qn)}{P_3(Qm+1,Qn)}} F \, d\mu  
$$  
is equal to  
$$  
\int_X F\, d\mu \cdot \int_X F^2\, d\mu  
\geq \Big( \int_X F\, d\mu \Big)^3.
$$  }

{\em General  case.} 
We now return to the general case. Since the action $T_n$ is not  
necessarily finitely generated, it is no longer true that the oscillating  
factors $\exp(F_N(f,L))$ and $\exp(G_{P,N}(f,L))$ converge to $1$, and the  
parameter $t$ need not be zero. 

{To control the $n^{it}$ terms, we restrict  
our $(m,n)$-averaging to a suitably chosen set $S_{1/q}$ (see for  
instance \eqref{def S_d a=c}), then we take multiplicative averages $\cesE_{q \in \Phi_{\ell}}$, and finally we 
send $\ell \to \infty$.}

To treat the oscillating factors, we take a triple multiplicative average  
over parameters $Q_1 \in \Phi_r$, $Q_2 \in \Phi_{r,K}$, and  
$Q_3 \in \Phi_{K,L,P_3}$, where the sets $\Phi_r$, $\Phi_{r,K}$, and  
$\Phi_{K,L,P_3}$ are defined in \eqref{eq_foln_def_1} and  
\eqref{eq_foln_def_2}. The primes dividing $Q_1$, $Q_2$, and $Q_3$ lie in  
disjoint ranges $[1,r]$, $(r,K]$, and $(K,L]$ respectively. Using the  
Chinese Remainder Theorem, we choose $v = v(Q_1,Q_2,Q_3) \in \N$ as in Lemma
\ref{lemma_choice_of_v_1}, and average over the grid  
$(Q_{L} m + 1,\, Q_{L} n + v)$, where $Q_{L}$ is the
highly divisible integer in \eqref{Q_l_def}. The disjointness of the  
prime factors of $Q_1$, $Q_2$, and $Q_3$ is essential for the  
congruences in Lemma \ref{lemma_choice_of_v_1} to be solvable.  

{
With this choice of $v$, averaging over the grid  
$(Q_{L} m + 1,\, Q_{L} n + v)$ allows us to factor out $Q_j$  
from $P_j(Q_{L} m + 1,\, Q_{L} n + v)$ for each  
$j \in \{1,2,3\}$, and ensures that  
$T_{P_j(Q_{L} m + 1,\, Q_{L} n + v)} F$ is asymptotically  
independent of the specific value of $Q_l$ for $l \neq j$  
(see Lemma \ref{L: main a=c P_2 reducible}).  
Then for every $j \in \{1,2\}$, we show that there exists a factor  
$\X_j$ of the system $(X, \X, \mu, T_n)$ such that if we write  
$F = F_{1,j} + F_{2,j}$, where $F_{1,j} := \E(F \mid \X_j)$ and  
$F_{2,j} \perp F_{1,j}$, then  
\begin{equation}\label{proof plan eq 8}
	\max_{q\in \Phi_{\ell}}  \,\, \cesE_{(m,n) \in S_{1/q} \cap [N]^2} 
	\lVert \cesE_{Q_j} T_{P_j(Q_{L} m + 1, Q_{L} n + v)}  
	F_{2,j} \rVert_{L^2(\mu)}
\end{equation}
goes to $0$ when all relevant parameters go to infinity  
(see for example \eqref{aux_eq_3}).  
Moreover, for $j=1$, the term $T_{P_1(Q_{L} m + 1, Q_{L} n + v)/P_3(Q_{L} m + 1, Q_{L} n + v)}  
F_{1,1}$ concentrates around $T_{1/q}F_{1,1}$, namely  
\begin{equation}\label{proof plan eq 9}
	\max_{q \in \Phi_{\ell}} \,\, \max_{Q_j}\,\, \cesE_{(m,n) \in S_{1/q} \cap [N]^2} 
	\Big \lVert T_{\frac{P_1(Q_{L} m + 1, Q_{L} n + v)}{P_3(Q_{L} m + 1, Q_{L} n + v)}} F_{1,1}  
	- T_{1/q} F_{1,1} \Big \rVert_{L^2(\mu)}
\end{equation}
is asymptotically $0$ when all relevant parameters go to infinity  
(see for example \eqref{eq_useful_fix_5}).  
On the other hand, for $j=2$, the term $T_{P_2(Q_{L} m + 1, Q_{L} n + v)/P_3(Q_{L} m + 1, Q_{L} n + v)}  
F_{1,2}$ concentrates around $F_{1,2}$, namely  
\begin{equation}\label{proof plan eq 11}
	\max_{q \in \Phi_{\ell}} \,\, \max_{Q_j}\,\, \cesE_{(m,n) \in S_{1/q} \cap [N]^2} 
	\Big \lVert T_{\frac{P_2(Q_{L} m + 1, Q_{L} n + v)}{P_3(Q_{L} m + 1, Q_{L} n + v)}} F_{1,2}  
	-  F_{1,2} \Big \rVert_{L^2(\mu)}
\end{equation}
is asymptotically $0$ when all relevant parameters go to infinity  
(see for example \eqref{eq_useful_4}).

The factors $\X_j$ are chosen so that whenever $F$ is  $\X_j$-measurable and $f$ lies  
in the support of the spectral measure  $\sigma_F$, the oscillating factor $\exp(F_N(f,L))$  
(or $\exp(G_{P,N}(f,L))$) arising from the concentration estimates  
(Proposition \ref{linear_conc} or \ref{quad_conc}) equals $1$ for  
$N > L$ and $L$ sufficiently large. This property is crucial in proving  
\eqref{proof plan eq 9}, \eqref{proof plan eq 11}.  
The specific factors are  
$\X_1 := \X_{\A}$, 
$\X_2 := \X_{p,\textup{fin.supp.}}$
defined in  
\eqref{def_X_A}, \eqref{def_X_p_fin_supp}
respectively.

Using \eqref{proof plan eq 8}, \eqref{proof plan eq 9} and 
\eqref{proof plan eq 11}, and averaging  
over $q\in \Phi_{\ell}$, $Q_1 \in \Phi_r$, $Q_2 \in \Phi_{r,K}$, and  
$Q_3 \in \Phi_{K,L,P_3}$, together with $(m,n) \in S_{1/q}$ in  
\eqref{proof plan eq 3}, and then letting  
$N \to \infty$, $L \to \infty$, $K \to \infty$, $r \to \infty$, and  
finally $\ell \to \infty$, in this order, yields  
\begin{multline*}
	\inflim_{\,\ell,r,K,L,N} \cesE_{q\in \Phi_{\ell},Q_1\in \Phi_r, Q_2\in \Phi_{r,K}, Q_3\in \Phi_{K,L,P_3}} 
	\cesE_{(m,n) \in S_{1/q} \cap [N]^2} 
	\int_X  \prod_{j=1}^3  
	T_{P_j(Q_{L} m + 1, Q_{L} n + v)} F \, d\mu   \\
	\geq  \inflim_{\,\ell} \ 
	\int_X  F \cdot F_{1,2} \cdot \cesE_{q\in \Phi_{\ell}} T_{1/q} F_{1,1}
    \, d\mu  = \int_X  F \cdot  \E(F \mid \X_2) \cdot
    \E(F \mid \mathcal{I}) \,    d\mu
	\geq \Big( \int_X F \, d\mu \Big)^3,
\end{multline*}
where the last estimate follows from Proposition \ref{chu-lemma}, and 
$\mathcal{I}:=\{B\in \CX\colon T_n^{-1}B=B\, \text{ for all } n\in\N\}$ is the invariant factor. 
This completes the rather detailed proof sketch of Theorem~\ref{proof plan thm 1}.}

\section{Proof of the main theorem when  $a=c$ or $b=c$}\label{Sec a=c or b=c}
In this section we prove Theorem \ref{T:Main} when  $a=c$ or $b=c$.  
The cases are symmetric, so we can assume that $a=c$. 
One  easily verifies that 
\begin{equation}\label{eq_param_a=c_2}
    x=k \,(m^2 - ab n^2),\:\: y=k \, 2 a mn, \:\:
    z=k \, ( m^2 + ab n^2),\:\:\: k, m,n\in \N,
\end{equation}
are solutions of $a x^2 + b y^2=c z^2$ when $a=c$. Let
\begin{equation}\label{quadratic_forms_a=c}
P_1(m,n):=2a mn, \:\:\:
P_2(m,n):=m^2 - ab n^2, \:\:\:
P_3(m,n):= m^2 + ab n^2.
\end{equation} 
It is easy to see that, since $a, b \neq 0$,  
the polynomials $P_1, P_2, P_3$ are pairwise independent,  
a property that is implicitly used in our arguments when we try  
to solve some simultaneous congruences. Also, for all $(m,n)\in \N^2$,  
$P_1(m,n), P_3(m,n) >0$, and if $m> \sqrt{ab}n$,  
we also have that $P_2(m,n) >0$.

The case $a=c$ of Theorem \ref{T:Main} follows from the next result.
\begin{theorem}\label{T: main a=c}
Let $(X,\CX,\mu,T_n)$  be a pretentious multiplicative action, let 
$F\in L^{\infty}(\mu)$ with $F\geq 0$, 
and let $P_1,P_2,P_3$ be as in \eqref{quadratic_forms_a=c}. Then for every $\e>0$ we have
\begin{equation}\label{T: main a=c eq}
    \int_X  T_{P_1(m,n)} F \cdot T_{P_2(m,n)} F \cdot T_{P_3(m,n)} F \, d\mu
    \geq  \Big(\int_X F \, d \mu \Big)^3 -\e
\end{equation}
for a set of $(m,n)\in \N^2$, $m>\sqrt{ab} n$, with positive upper 
density.
\end{theorem}

Note that  $P_1$ is reducible and 
$P_3$ is irreducible. We split the proof of Theorem \ref{T: main a=c} into two cases, 
according to whether $P_2$ is irreducible or not.\footnote{{Reducibility simplifies the individual concentration estimates, since these can be reduced to estimates along linear forms. However, coordinating the three polynomial iterates and eliminating their oscillatory factors remains delicate, requiring similar congruence and spectral arguments in both the reducible and irreducible cases.}}

Throughout Section \ref{Sec a=c or b=c}, $\Phi_r$, $\Phi_{r,K}$, 
$\Phi_{r,K,P}$ are as in \eqref{eq_foln_def_1}, \eqref{eq_foln_def_2},
$Q_{L}$ is as in \eqref{Q_l_def}. Furthermore, 
$\omega_P$ is as in \eqref{omega_def}, and $\D$, $F_N(f,L)$, $G_{P,N}(f,L)$
are as in \eqref{E: definition of pretentious distance}, 
\eqref{def_F_N}, \eqref{def_G_P,N}. 
Finally, $X_{\A}$, $X_{p,\textup{fin.supp.}}$, and $X_{p,\textup{fin.supp.},P}$
are the factors defined respectively in \eqref{def_X_A}, 
\eqref{def_X_p_fin_supp}, and \eqref{def_X_p_fin_supp_P}.
Throughout Sections~\ref{Sec a=c or b=c} and~\ref{Sec a+b=c}, we also assume that $2K<L$.
This restriction does not affect the iterated limits, since
$L\to\infty$ is taken with $K$ fixed.

{
For each $q \in \N$ consider the set $S_{1/q}$ as in \eqref{def S_d a=c}, i.e. 
\begin{equation}\label{S_delta_a=c_P_2_reduc_fixed}
    S_{1/q}=\Big\{ (m,n) \in \N ^2 \colon \frac{ e ^{- 1 /(2q)}}{2aq}  < 
    \frac{n}{m} < \frac{ e ^{1 /(2q)}}{2aq} \Big\}.
\end{equation}
As we saw in Lemma \ref{L: S_d a=c}, $S_{1/q}$ has positive density, and if $q$ is large enough 
depending on $a,b$ (so that $1/q$ is small enough), then for $(m,n) \in S_{1/q}$ we have $m> 2 \sqrt{ab} n$, and
\begin{equation}\label{arch conc S_d a=c}
 \big| (P_1(m,n)/P_3(m,n)) ^{it} - q^{-it} \big| \leq |t| / q, 
 \quad \big| (P_2(m,n)/P_3(m,n)) ^{it} -1 \big| \leq |t| /q.
\end{equation}
Later we will average $\cesE_{q \in \Phi_{\ell}}$, with $\Phi_{\ell}$ as in \eqref{eq_foln_def_1},
and we will take $\ell \to \infty$, so we always assume $\ell$ to be large enough 
so that for each $q \in \Phi_{\ell}$ and $(m,n)\in S_{1/q}$, $m> 2 \sqrt{ab}n$ and  
\eqref{arch conc S_d a=c} holds. 
}

\subsection{$P_2$ reducible}\label{Sec : a=c_P_2_reducible}
Here we handle the case when $P_2$ is reducible.  
A representative example is when $a=b=c=1$, which 
 corresponds to the classical Pythagorean equation  
$x^2 + y^2 = z^2,$
in which case  
\eqref{quadratic_forms_a=c} gives 
$$
P_1(m,n)=2mn,   \quad P_2(m,n) = m^2-n^2, \quad  
P_3(m,n)=m^2+n^2.
$$

Since $P_2$ is reducible, we have that $ab$ is a perfect square, i.e., 
$ab=\gamma ^2$ for some $\gamma \in \N$, and we have that 
$P_2(m,n)=(m - \gamma n)(m + \gamma n)$. 
Let $\kappa \in \N$ so that $1-\kappa \gamma \neq 0$.

To prove the result, given $Q_1 \in \Phi_r$, $Q_2 \in \Phi_{r,K}$, and 
$Q_3 \in \Phi_{K,L, P_3}$,
we will choose a suitable integer $v$, depending on $r,K,L,Q_1,Q_2, Q_3$, and 
we will average over the grid 
$(Q_{L} m+1, Q_{L} n +v)$ in \eqref{T: main a=c eq}, where $Q_{L}$
is as in \eqref{Q_l_def}.

From now on, we assume that $r$ is large enough so that for $p>r$ we have 
$\omega_{P_3}(p)\in \{0,2\}$ and 
\begin{equation}\label{E1: p not dividing}
p\nmid 2\kappa \gamma ab(1+ab)(1-\kappa \gamma)
(1+\kappa \gamma)(1+\kappa ^2 ab).    
\end{equation}

\begin{lemma}\label{lemma_choice_of_v_1}
Let $r,K,L\in \N$ with $r<K<L/2$, 
$P_1,P_2,P_3$ be as in \eqref{quadratic_forms_a=c}, and let $Q_1 \in \Phi_r$, 
$Q_2 \in \Phi_{r,K}$,  $Q_3 \in \Phi_{K,L,P_3}$. Then there is an integer
$v$ with $0\leq v \leq Q_{L}-1$ such that 
\begin{enumerate}[(i)]
    \item \label{E1: lemma_choice_of_v_1} $(Q_{L},v)=Q_1$, 
    $\frac{v}{Q_1}\equiv 1 \pmod {Q_1}$,
    \item \label{E2: lemma_choice_of_v_1} 
    $(Q_{L},1-\gamma v)= Q_2$, 
    $\frac{1-\gamma v}{Q_2} \equiv Q_2 ^{-1} 
    \pmod{{Q_1}}$,
    \item \label{E3: lemma_choice_of_v_1} $(Q_{L},1+\gamma v)=1$, 
    ${1+\gamma v}\equiv 1 \pmod {{Q_1}}$,
    \item \label{E4: lemma_choice_of_v_1} 
    $(Q_{ L},P_3(1,v))=Q_3$,
    $\frac{P_3(1,v)}{Q_3} \equiv Q_3^{-1} \pmod{{Q_1}}$.
\end{enumerate}
\end{lemma}

\begin{proof}
Let $p\in \P$ with $K<p \leq L$ and $\omega_{P_3}(p)=2$, and consider the 
polynomial $g(x):=P_3(1,x)-p^{\theta_p(Q_3)}$. 
Then $\textup{mod}\: p$ the polynomial
becomes $g(x)=P_3(1,x)$, and it has two simple roots
$z_1, z_2$ with $z_1 \not \equiv z_2 \pmod p$. Using  
Hensel's lemma, we may lift $z_1$ to a unique root 
$\zeta = \zeta (p, Q_3)$ of $g \mod {p^{\theta_p(Q_3)+1}}$.

Using the Chinese Remainder theorem we may then choose $v$ with 
$$0\leq v \leq \prod_{p\leq K} p^{4K} \prod_{K<p\leq L}p^{1+3L/2}-1\leq 
Q_{L}-1$$ 
such that 
\begin{itemize}
     \item $v\equiv Q_1 \pmod {p^{2 \theta_p(Q_1)}}$ for 
    $p\leq r$,
    \item $v\equiv \gamma ^{-1} - \gamma ^{-1}p^{\theta_p(Q_2)}$
    $\pmod {p^{\theta_p(Q_2)+1}}$ for $r< p\leq K$,
    \item $v\equiv \kappa \pmod {p}  $ for $K<p\leq L$ with $\omega_{P_3}(p)=0$,
    \item $v\equiv \zeta(p,Q_3) \pmod {p^{\theta_p(Q_3)+1}}$ for 
    $K<p\leq L$ with $\omega_{P_3}(p)=2$.
\end{itemize}  
For $p\leq r$, $v\equiv Q_1 \pmod{p^{2\theta_p(Q_1)}}$, so 
$p^{\theta_p(Q_1)}\parallel v$. For $r<p\leq K$, $v\equiv \gamma^{-1}
\pmod {p}$, so $(p,v)=1$. For $K<p \leq L$ with $\omega_{P_3}(p)=0$, 
$v\equiv \kappa \pmod {p}$, so in view of \eqref{E1: p not dividing} we 
get $(p,v)=1$. Finally, for 
$K<p \leq L$ with $\omega_{P_3}(p)=2$, we have $v\equiv \zeta(p,Q_3)
\pmod {p^{\theta_p(Q_3)+1}}$, so 
$v\equiv z_1 \pmod {p}$. If $p\mid v$, then $p\mid z_1$, which
implies that $P_3(1,z_1) \equiv 1 \pmod {p}$. On the other hand, 
$P_3(1,z_1)\equiv 0 \pmod {p}$, so we reach a contradiction. Therefore,
$(p,v)=1$. Combining all the above, we infer that 
$(Q_{L},v)=Q_1$. For each $p\leq r$, 
$\frac{v}{Q_1} \equiv 1 \pmod {p^{\theta_p(Q_1)}}$, so 
$\frac{v}{Q_1} \equiv 1 \pmod {Q_1}$. Therefore, $v$ satisfies
\eqref{E1: lemma_choice_of_v_1}.

For $p\leq r$, $1-\gamma v\equiv 1 \pmod{p}$, so 
$(p,1-\gamma v)=1$. For $r<p\leq K$, one sees that 
$1-\gamma v \equiv p^{\theta_p(Q_2)}
\pmod {p^{\theta_p(Q_2)+1}}$, so $p^{\theta_p(Q_2)} \parallel (1-\gamma v)$. 
For $K<p \leq L$ with $\omega_{P_3}(p)=0$, 
$1-\gamma v  \equiv 1-\kappa \gamma \pmod {p}$, which in view of 
\eqref{E1: p not dividing} implies $(p,1 -\gamma v)=1$.
Finally, for 
$K<p \leq L$ with $\omega_{P_3}(p)=2$, we have 
$1-\gamma v\equiv 1-\gamma \zeta(p,Q_3) \pmod{p^{\theta_p(Q_3)+1}}$, so 
$1-\gamma v\equiv1-\gamma  z_1 \pmod {p}$. 
If $p\mid 1-\gamma v$, then $p\mid 1-\gamma z_1$, which
implies that $P_2(1,z_1) \equiv 0 \pmod {p}$. Recall that   
$P_3(1,z_1)\equiv 0 \pmod {p}$, so adding we obtain that 
$2\equiv 0 \pmod p$, which contradicts 
\eqref{E1: p not dividing}. Therefore, $(p,1-\gamma v)=1$. 
Combining all the above, we infer that 
$(Q_{ L},1-\gamma v)=Q_2$. From the choice of $v$ we have  
$1-\gamma v\equiv 1 \pmod {Q_1}$, so  
$\frac{1-\gamma v}{Q_2} \equiv Q_2 ^{-1} \pmod {Q_1}$, and we 
conclude that $v$ satisfies
\eqref{E2: lemma_choice_of_v_1}.

For $p\leq r$, $1+\gamma v\equiv 1 \pmod{p}$, so 
$(p,1+\gamma v)=1$. For $r<p\leq K$, if $p\mid 1+\gamma v$, then, 
since $p\mid 1-\gamma v$, we get that $p\mid 2$, which contradicts 
\eqref{E1: p not dividing}. Therefore, $(p,1+\gamma v)=1$.
For $K<p \leq L$ with $\omega_{P_3}(p)=0$, 
$1+\gamma v  \equiv 1+\kappa \gamma \pmod {p}$, which in view of 
\eqref{E1: p not dividing} implies that $(p,1+\gamma v)=1$.
Finally, for 
$K<p \leq L$ with $\omega_{P_3}(p)=2$,
if $p\mid 1+ \gamma v$, then as before we have 
$P_2(1,z_1) \equiv 0 \pmod {p}$, and since $P_3(1,z_1)\equiv 0 \pmod {p}$,
we get that $2\equiv 0 \pmod p$, contradicting 
\eqref{E1: p not dividing}. Therefore, $(p,1+\gamma v)=1$. 
Combining all the above, we obtain that 
$(Q_{ L},1+\gamma v)=1$ and 
$1+\gamma v\equiv 1 \pmod {Q_1}$, so $v$ satisfies
\eqref{E3: lemma_choice_of_v_1}.

It remains to prove \eqref{E4: lemma_choice_of_v_1}.
For $p\leq r$, $P_3(1,v)\equiv 1 \pmod{p}$, so 
$(p,P_3(1,v))=1$. For $r<p\leq K$, if $p\mid P_3(1,v)$, then 
since $p\mid P_2(1,v)$, we get that $p\mid P_2(1,v)+P_3(1,v)=2$, 
which contradicts 
\eqref{E1: p not dividing}. Therefore, $(p, P_3(1,v))=1$.
For $K<p \leq L$ with $\omega_{P_3}(p)=0$, 
$P_3(1,v)  \equiv 1+\kappa^2 ab \pmod {p}$, which in view of 
\eqref{E1: p not dividing} implies that $(p, P_3(1,v))=1$.
Finally, for 
$K<p \leq L$ with $\omega_{P_3}(p)=2$, we have 
$v\equiv \zeta(p,Q_3) \pmod {p^{\theta_p(Q_3)+1}}$, which implies that
$P_3(1,v)-p^{\theta_p(Q_3)}\equiv 0 \pmod{p^{\theta_p(Q_3) +1}}$. 
We deduce that $p^{\theta_p(Q_3)}\parallel P_3(1,v)$.
Combining all the above,  we obtain that 
$(Q_{L},P_3(1,v))=Q_3$. In addition, since $v\equiv 0 \pmod {Q_1}$,
we have $P_3(1,v) \equiv 1 \pmod {Q_1}$, so 
$\frac{P_3(1,v)}{Q_3}\equiv Q_3 ^{-1} \pmod {Q_1}$, so $v$ satisfies
\eqref{E4: lemma_choice_of_v_1}. This concludes the proof 
of the lemma.
\end{proof}

For the remainder of Section \ref{Sec : a=c_P_2_reducible}, we suppress the  
dependence of $v$ on the parameters $r,K,L,Q_1,Q_2,Q_3$, and simply  
write $v$ in place of $v(r,K,L,Q_1,Q_2,Q_3)$. In several other cases  
we leave the dependence on the parameters explicit, at the cost of  
making some formulas less aesthetically pleasing, as omitting it would  
obscure the underlying arguments. 

Given $r, K, L, N\in \N$ with 
$r<K<L<N $, $Q_1 \in \Phi_r$, $Q_2\in \Phi_{r,K}$, 
$Q_3\in \Phi_{K,L,P_3}$, and $m,n\in \N$, let 
\begin{enumerate}
\item \label{R_1 def} $R_1 (r, K, L, N, Q_1, Q_2, Q_3,m,n):= 2a
({Q_{L}} m+1)
(\frac{Q_{L}}{Q_1}n+ \frac{v}{Q_1} )$,
\item \label{R_2 def} $R_2 ( r, K, L, N, Q_1, Q_2, Q_3,m,n):= 
(\frac{Q_{L}}{Q_2} (m -\gamma n)+ 
\frac{1-\gamma v}{Q_2})({Q_{L}}
(m+\gamma n)+ 1+\gamma v )$,
\item \label{R_3 def} $R_3 (r, K, L, N, Q_1, Q_2, Q_3,m,n):= 
\frac{P_3(Q_{L} m +1, Q_{L} n +v )}{Q_3}$,
\end{enumerate}
where $v$ is as in Lemma \ref{lemma_choice_of_v_1}.
Then  for $j\in \{1,2,3\}$ we have 
\begin{equation}\label{eq_extr_1}
    P_j(Q_{L} m+1, Q_{L} n +v)= Q_j \,
    R_j (r, K, L, N, Q_1, Q_2, Q_3,m,n).
\end{equation}
Sometimes in the notation we suppress the dependence of $R_j$ on 
the parameters 
$r, K$, $L$, $N,$ $ Q_1,$  $Q_2, Q_3$, and we only write $R_j(m,n)$.

Observe that if $m>2\sqrt{ab}n=2 \gamma n$, then since 
$0\leq v \leq Q_{L}-1$, we have that $Q_L m+1 > \sqrt{ab} (Q_Ln+v)$, 
so $P_2(Q_L m+1, Q_L n+v)>0$. For $(m,n) \in S_{1/q}$ we have $m>2\sqrt{ab}n$, so 
$P_2(Q_L m+1, Q_L n+v)>0$, while for $j=1,3,$ we have
$P_j(Q_Lm+1, Q_Ln+v)>0$ for all $m,n \in \N$. 

\begin{comment}
{The set that will do the job for us is the following: 
For each $q \in \N$ (sufficiently large depending on $a,b$) consider 
the set 
\begin{equation}\label{S_delta_a=c_P_2_reduc_fixed}
    S_{q}:=\{(m,n)\in \N^2 : m> 2 \sqrt{ab}n, \: 
    |\log\frac{P_2 (m,n)}{P_3 (m,n)} | \leq \frac{1}{q}, \:
    |\log\frac{P_1 (m,n)}{P_3 (m,n)} - \log \frac{1}{q} | \leq \frac{1}{q} \}.
\end{equation}
Now $q$ will play the role of $\delta$, and we will average over it multiplicatively, i.e. we will have 
$\cesE_{q \in \Phi_{\ell}}$, so 
the last limit taken will be $\ell \to \infty$ instead of 
$\delta \to 0$. 

We could also work with the sets 
\begin{equation}\label{S_delta_a=c_P_2_reduc_fixed_1}
    S_{q,\delta}:=\{(m,n)\in \N^2 : m> 2 \sqrt{ab}n, \: 
    |\log\frac{P_2 (m,n)}{P_3 (m,n)} | \leq \delta, \:
    |\log\frac{P_1 (m,n)}{P_3 (m,n)} - \log \frac{1}{q} | \leq 
    \delta \}.
\end{equation}
and and take first $\ell \to \infty$ and then $\delta \to 0$, 
but I do not see the reason to insert the $\delta$ if it is not necessary.}
\end{comment}

When $P_2$ is reducible, Theorem \ref{T: main a=c} follows from the next 
result.

\begin{proposition}\label{P: main a=c P_2 reducible}
Let $(X,\CX,\mu,T_n)$  be a pretentious multiplicative action, let 
$F\in L^{\infty}(\mu)$ with $F\geq 0$, 
let $P_1,P_2,P_3$ be as in \eqref{quadratic_forms_a=c}, and suppose  that $P_2$ is reducible. 
Also, let $S_{1/q}$ be as in \eqref{S_delta_a=c_P_2_reduc_fixed}  and $v=v(r,K,L,Q_1,Q_2,Q_3)$ be  as in Lemma
\ref{lemma_choice_of_v_1}.
Then 
\begin{multline}\label{main_term_1}
    \inflim_{\, \ell, r,K,L, N}\,
    \cesE_{q\in \Phi_{\ell}} \,
    \cesE_{Q_{1}\in \Phi_{r}} \, \cesE_{Q_{2}\in \Phi_{r,K}} \,
    \cesE_{Q_{3}\in \Phi_{K,L,P_3}}\, 
    \cesE_{(m,n) \in S_{1/q} \cap [N]^2}  \\
    \int_{X} \prod_{j=1} ^3 
    T_{P_j(Q_{L} m+1, Q_{L} n +v)} F 
    \, d \mu \geq \Big(\int_X F \, d \mu\Big)^3.
    \end{multline}
\end{proposition}

\begin{proof}[Proof that Proposition \ref{P: main a=c P_2 reducible} implies 
Theorem \ref{T: main a=c} when $P_2$ is reducible]
Let $\e>0$. Using \eqref{main_term_1} we infer that there are 
$\ell,r,K,L \in \N$ and $q \in \Phi_{\ell}$, $Q_1 \in \Phi_r$, 
$Q_2 \in \Phi_{r,K},$  $Q_3 \in \Phi_{K,L,P_3}$ such that 
\begin{equation}\label{E: prop implies thm 1}
    \limsup_{N\to \infty} \cesE_{(m,n) \in S_{1/q} \cap [N]^2}
    \int_{X} \prod_{j=1} ^3
    T_{P_j(Q_{L} m+1, Q_{L} n +v)} F 
    \, d \mu  \geq \Big(\int_X F \, d \mu\Big)^3 -\frac{\e}{2}.
\end{equation}
From \eqref{E: prop implies thm 1} and the fact that $S_{1/q}$
has positive density, we get that the set 
\begin{equation}\label{E: prop implies thm 2}
    \Big\{ (m,n)\in S_{1/q} \colon  \int_{X} \prod_{j=1} ^3
    T_{P_j(Q_{L} m+1, Q_{L} n +v)} F 
    \, d \mu \geq \Big(\int_X F \, d \mu\Big)^3 -\e\Big\}
\end{equation}
has positive upper density. Now, if a set $A\subset \N^2$ has positive 
upper density, and $Q,w_1, w_2 \in \N$, then $B=\{(Qm+w_1, Qn + w_2)\colon
(m,n)\in A\}$ also has positive upper density. In addition, as we 
commented earlier, if $(m,n)\in S_{1/q}$, then $m>2\sqrt{ab}n$, which
implies $Q_L m+1 > \sqrt{ab}(Q_L n+v)$. Combining the above with 
\eqref{E: prop implies thm 2} we deduce that the set
\begin{equation*}
    \Big\{ (m,n)\in \N^2\colon  m>\sqrt{ab}n, 
    \int_X \prod_{j=1} ^3 T_{P_j(m,n)} F \, d\mu
    \geq  \Big(\int_X F \, d \mu \Big)^3 - \e\Big\}
\end{equation*}
has positive upper density. This concludes the proof.
\end{proof}

It remains to prove Proposition \ref{P: main a=c P_2 reducible}. For that we need the 
following lemma.

\begin{lemma}\label{L: main a=c P_2 reducible}
Let $(X, \X, \mu, T_n)$, $F$, and 
$R_1, R_2, R_3$ be as in \eqref{R_1 def}-\eqref{R_3 def}.
Then for every $\ell$ we have 
\begin{align}
& \lim_{r, K, L, N} \,\, \max_{q \in \Phi_{\ell}} \,\,
\max_{Q_1 \in \Phi_r, \; Q_2, Q_{2}' \in \Phi_{r,K}, \:Q_3, Q_{3}' \in \Phi_{K,L,P_3} } 
\cesE_{(m,n) \in S_{1/q} \cap [N]^2} 
    \label{L: main a=c P_2 reducible eq 1}\\ 
& \hspace{10mm} \lVert 
 T_{R_1(r, K, L, N, Q_1, Q_2, Q_3,m,n)} F - 
 T_{R_1( r, K, L, N, Q_1, Q_2 ', Q_3 ',m,n)} F
\rVert _{L^2(\mu)}=0, \nonumber \\
& \lim_{r, K, L, N}  \, \, \max_{q \in \Phi_{\ell}} \,\,
\max_{Q_1, Q_1 ' \in \Phi_r, \; Q_2 \in \Phi_{r,K}, \:Q_3, Q_{3}' \in \Phi_{K,L,P_3} }
 \cesE_{(m,n) \in S_{1/q} \cap [N]^2} 
     \label{L: main a=c P_2 reducible eq 2} \\
& \hspace{10mm}  \lVert 
T_{R_2(r, K, L, N, Q_1, Q_2, Q_3,m,n)} F - 
T_{R_2( r, K, L, N, Q_1 ', Q_2, Q_3 ',m,n)} F
\rVert _{L^2(\mu)}=0, \nonumber  \\
& \lim_{r, K, L, N}\, \, \max_{q \in \Phi_{\ell}} \,\,
\max_{Q_1, Q_1 ' \in \Phi_r, \; Q_2, Q_2 ' \in \Phi_{r,K}, \:Q_3 
\in \Phi_{K,L, P_3} } 
 \cesE_{(m,n) \in S_{1/q} \cap [N]^2} 
     \label{L: main a=c P_2 reducible eq 3} \\
& \hspace{10mm} \lVert 
T_{R_3( r, K, L, N, Q_1, Q_2, Q_3,m,n)} F - 
T_{R_3(r, K, L, N, Q_1 ', Q_2 ', Q_3,m,n)} F
\rVert _{L^2(\mu)}=0. \nonumber 
\end{align}
\end{lemma}

\begin{remark}
Lemma \ref{L: main a=c P_2 reducible} essentially says that 
in the integral in \eqref{main_term_1},
the first term is asymptotically independent of $Q_2, Q_3$, the second term 
is asymptotically independent of $Q_1, Q_3$ and the third term 
is asymptotically independent of $Q_1, Q_2$.
\end{remark}

\begin{proof}
Let $\ell$ be fixed. We can assume that 
$\lVert F \rVert_{L^{\infty}(\mu)}\leq 1$.
%In the proof, it is always implicit that $r<K<L<N$, and $L$ is 
%sufficiently large depending on $r,K$.
We start by proving \eqref{L: main a=c P_2 reducible eq 1}.
Note by \eqref{spec_eq_2} we have
\begin{multline*}
\lVert 
T_{R_1( r, K, L, N, Q_1, Q_2, Q_3,m,n)} F - 
T_{R_1( r, K, L, N, Q_1, Q_2 ', Q_3 ',m,n)} F
\rVert _{L^2(\mu)}= \\ 
\lVert 
f(R_1( r, K, L, N, Q_1, Q_2, Q_3,m,n)) - 
f(R_1( r, K, L, N, Q_1, Q_2 ', Q_3 ',m,n)) 
\rVert _{L^2(\sigma_{F}(f))},
\end{multline*}
so after applying the Cauchy-Schwarz inequality and Fatou's lemma 
multiple times, one 
sees that in order to establish \eqref{L: main a=c P_2 reducible eq 1}, it
suffices to prove that for each $f\in \M_p $ we have 
\begin{multline}\label{eq_conc_1}
\lim_{r,K, L, N}\, \, \max_{q \in \Phi_{\ell}} \,\,
\max_{Q_1 \in \Phi_r, \; Q_2, Q_{2}' \in \Phi_{r,K}, \:Q_3, Q_{3}' \in \Phi_{K,L,P_3}} \cesE_{(m,n) \in S_{1/q} \cap [N]^2} 
  \\
|f(R_1( r, K, L, N, Q_1, Q_2, Q_3,m,n))
\overline{f(R_1( r, K, L, N, Q_1, Q_2 ', Q_3 ',m,n))}-1 |
=0.
\end{multline}
Let $f\in \M_p$ and let $\chi,t$ so that $f\sim \chi\cdot n^{it}$. Let  
$q_{\chi}$ be the conductor of $\chi$, and suppose that
$r,L$ are large enough so that $q_{\chi}$ divides 
$ \frac{Q_{ L}}{Q_1}, Q_1$ for every $Q_1 \in \Phi_r$. Using 
Lemma \ref{lemma_choice_of_v_1} \eqref{E1: lemma_choice_of_v_1} and
applying Proposition \ref{linear_conc} and Remark \ref{uniform_in_Q} we get that 
\begin{align*}
&\limsup_{N\to \infty} \,\, \max_{q \in \Phi_{\ell}} \,\, \max_{Q_1 \in \Phi_r, \; Q_2, Q_{2}' \in \Phi_{r,K}, \:Q_3, Q_{3}' \in \Phi_{K,L,P_3}}
\cesE_{(m,n) \in  S_{1/q} \cap [N]^2}  \\
&|f(R_1( r, K, L, N, Q_1, Q_2, Q_3,m,n)) \cdot 
 \overline{f(R_1( r, K, L, N, Q_1, Q_2 ', Q_3 ',m,n))}-1 | \\
&\ll_{\ell} \limsup_{N\to \infty} 
|\exp(2\Re (F_N(f,L))) -1| + 
\D(f, \chi \cdot n^{it}; L, \infty)
+L^{-1/2}.
\end{align*}
Since $f\sim \chi \cdot n^{it}$, we have that $2\Re(F_N(f,L))$ converges to
$2\sum_{p>L} \frac{1}{p}(\Re(f(p) \overline{\chi(p)}p^{-it}) -1)$ as $N\to\infty$. Then
\eqref{eq_conc_1} follows by
taking $L\to \infty$ in the last relation. This concludes the proof 
of \eqref{L: main a=c P_2 reducible eq 1}. 

For the proof of \eqref{L: main a=c P_2 reducible eq 2}, arguing as 
before, it suffices to prove that for every 
$f\in \M_p $
\begin{multline}\label{eq_conc_2}
\lim_{r, K, L, N} \,\, \max_{q \in \Phi_{\ell}} \,\,
\max_{Q_1, Q_1 ' \in \Phi_r, \; Q_2 \in \Phi_{r,K}, \:Q_3, Q_{3}' \in \Phi_{K,L,P_3} } 
\cesE_{(m,n) \in S_{1/q} \cap [N]^2} 
   \\
| f({R_2( r, K, L, N, Q_1, Q_2, Q_3,m,n)}) \cdot
\overline{f(R_2( r, K, L, N, Q_1 ', Q_2, Q_3 ',m,n))}
-1| =0.
\end{multline}
Again, let $f\in \M_p$, let $\chi,t$ so that $f\sim \chi\cdot n^{it}$,
and let $q_{\chi}$ be the conductor of $\chi$. 
Suppose that $r,L$ are sufficiently large so that $q_{\chi}$ divides 
$\frac{Q_{L}}{Q_{2}}, {Q_1}$ for every $Q_1 \in \Phi_r$ and 
$Q_{2} \in \Phi_{r,K}$.
Using Lemma \ref{lemma_choice_of_v_1} \eqref{E2: lemma_choice_of_v_1},
\eqref{E3: lemma_choice_of_v_1} and
applying Proposition \ref{linear_conc}, Remark \ref{uniform_in_Q} and Lemma
\ref{bilinear_lemma} we get that 
\begin{align*}
&\limsup_{N\to \infty} \,\, \max_{q \in \Phi_{\ell}} \,\,
\max_{Q_1, Q_1 ' \in \Phi_r, \; Q_2 \in \Phi_{r,K}, \:Q_3, Q_{3}' \in \Phi_{K,L,P_3} }  
\cesE_{(m,n) \in S_{1/q} \cap [N]^2}  \\ 
&| f({R_2( r, K, L, N, Q_1, Q_2, Q_3,m,n)})\cdot  
\overline{f(R_2( r, K, L, N, Q_1 ', Q_2, Q_3 ',m,n))}
-1| \\ & \ll_{a,b, \ell}
\limsup_{N\to \infty} |\exp(4\Re (F_N(f,L))) -1| + 
\D(f, \chi \cdot n^{it}; L, \infty) +L^{-1/2},
\end{align*}
so again \eqref{eq_conc_2} follows by taking $L\to \infty$ in the last 
relation and using that $f\sim \chi \cdot n^{it}$.

Finally, to prove \eqref{L: main a=c P_2 reducible eq 3},
it suffices to prove that for every 
$f\in \M_p $ we have 
\begin{multline}\label{eq_conc_3}
\lim _{ r, K, L, N}  \,\, \max_{q \in \Phi_{\ell}} \,\, 
\max_{Q_1, Q_1 ' \in \Phi_r, \; Q_2, Q_2' \in \Phi_{r,K}, \:Q_3 \in \Phi_{K,L,P_3} } 
\cesE_{(m,n) \in S_{1/q} \cap [N]^2} 
    \\
| f({R_3( r, K, L, N, Q_1, Q_2, Q_3,m,n)}) \cdot 
\overline{f(R_3( r, K, L, N, Q_1 ', Q_2 ', Q_3,m,n))}
-1| =0.
\end{multline}
Let $f\in \M_p$, let $\chi,t$ so that $f\sim \chi\cdot n^{it}$, and let 
$q_{\chi}$ be the conductor of $\chi$. 
Suppose that $r,L$ are sufficiently large so that $q_{\chi}$ divides 
$\frac{Q_{ L}}{{Q}_{3}}, {Q_1}$ for every 
$Q_1 \in \Phi_r$ and $Q_{3} \in \Phi_{K,L,P_3}$.
Then, using Lemma \ref{lemma_choice_of_v_1} \eqref{E4: lemma_choice_of_v_1},
and applying Proposition \ref{quad_conc} and Remark \ref{uniform_in_Q} we get that 
\begin{align*}
&\limsup_{N\to \infty}  \,\, \max_{q \in \Phi_{\ell}} \,\, \max_{Q_1, Q_1 ' \in \Phi_r, \; Q_2, Q_2' \in \Phi_{r,K}, \:Q_3 \in \Phi_{K,L,P_3} }  
\cesE_{(m,n)\in S_{1/q} \cap [N]^2} \\
&| f({R_3( r, K, L, N, Q_1, Q_2, Q_3,m,n)})\cdot 
\overline{f(R_3( r, K, L, N, Q_1 ', Q_2 ', Q_3,m,n))}
-1|\\ & \ll_{P_3, \ell}
\limsup_{N\to \infty} |\exp(2\Re(G_{P_3,N}(f,L)))  -1| +
\D(f, \chi \cdot n^{it}; L, \infty)+L^{-1/2}. 
\end{align*}
Again, using $f\sim \chi \cdot n^{it}$, we have that 
$2\Re(G_{P_3,N}(f,L))$ converges  as $N\to\infty$ to
$$2\sum_{p>L} \frac{\omega_{P_3}(p)}{p}\big(\Re(f(p) 
\overline{\chi(p)}p^{-it}) -1\big),$$ and
\eqref{eq_conc_3} follows by
taking $L\to \infty$ in the last relation. This concludes the proof 
of the lemma. 
\end{proof}

We are now ready to prove Proposition 
\ref{P: main a=c P_2 reducible}. The argument is somewhat lengthy, so we divide it into several steps.
{ Recall the definition of the  factors $\X_{\A}$ and $\X_{p,\textup{fin.supp.}}$
introduced in Section~\ref{SS:factors}, whose corresponding
$L^2(\mu)$ subspaces are defined in \eqref{def_X_A}
and \eqref{def_X_p_fin_supp}, respectively.}

{
\begin{proof}[Proof of Proposition \ref{P: main a=c P_2 reducible}]
By homogeneity, we may assume that
	$\lVert F\rVert_{L^{\infty}(\mu)}\leq 1$. In the proof, it is always 
implicit that $r<K<L<N$, and that each of the previous variables is 
sufficiently large depending on the smaller ones.

In view of \eqref{eq_extr_1}, the quantity in 
$\inflim_{\,\ell,r,K,L,N}$ on the left-hand side in 
\eqref{main_term_1} is equal to 
\begin{multline*}
\cesE_{q\in \Phi_{\ell}}\, \cesE_{Q_{1}\in \Phi_{r}}\, \cesE_{Q_{2}\in \Phi_{r,K}}\,
    \cesE_{Q_{3}\in \Phi_{K,L,P_3}} \, \cesE_{(m,n) \in S_{1/q} \cap [N]^2} \\
     \int_{X}  T_{Q_1 R_1(m,n)}  F \cdot
    T_{Q_2 R_2(m,n)} F \cdot 
    T_{{Q}_3 R_3(m,n)} F
  \,  d \mu.
\end{multline*}
For each $r,K,L\in \N$, choose any $Q_{1,r}\in \Phi_r$,  
$Q_{2, r, K} \in \Phi_{r,K}$ and $Q_{3,K,L}\in \Phi_{K,L,P_3}$. 
Using Lemma \ref{L: main a=c P_2 reducible}, one sees that in order to prove
\eqref{main_term_1}, it suffices to prove that 
\begin{multline}\label{main_term_3}
\inflim_{\, \ell,r, K, L,N}
    \int_{X} \cesE_{q\in \Phi_{\ell}} \,  \cesE_{(m,n) \in S_{1/q} \cap [N]^2} \\
    \Big( \cesE_{Q_{1}\in \Phi_{r}} T_{Q_1 R_1( r, K, L, N, Q_1, Q_{2,r,K} , Q_{3,K,L} ,m,n)} F \Big) \cdot 
    \Big( \cesE_{Q_{2}\in \Phi_{r,K}}  
     T_{Q_2 R_2( r, K, L, N, Q_{1,r}, 
    Q_2 , Q_{3,K,L},m,n)} F \Big) \cdot \\
    \Big( \cesE_{Q_{3}\in \Phi_{K,L,P_3}} 
    T_{{Q}_3 R_3( r, K, L, N, Q_{1,r}, Q_{2,r,K} , Q_3 ,m,n)} 
    F \Big) \,
    d \mu \geq \Big(\int_X F \, d \mu\Big)^3 .
\end{multline}

{ Our strategy is to replace some occurrences of $F$ by its conditional expectations onto suitable factors of the action, for which the corresponding averages satisfy concentration estimates. }

\smallskip 

{\bf Step 1 (Reduction to $\X_\A$ in the average over $Q_1$).} 
Write $F=F_1 + F_3$, where $F_1:=\E(F \mid {\X}_{\A})$ and 
$F_3 \perp F_1$. Then the integral on the left-hand side of 
\eqref{main_term_3} is equal to 
\begin{multline}\label{main_term_4}
    \sum_{j\in \{1,3\}}\int_{X}  \cesE_{q\in \Phi_{\ell}} \, 
    \cesE_{(m,n) \in S_{1/q} \cap [N]^2} \\
    \Big( \cesE_{Q_{1}\in \Phi_{r}} T_{Q_1 R_1( r, K, L, N, Q_1, Q_{2,r,K} , Q_{3,K,L} ,m,n)} F_j \Big) \cdot 
    \Big( \cesE_{Q_{2}\in \Phi_{r,K}}  
     T_{Q_2 R_2( r, K, L, N, Q_{1,r}, 
    Q_2 , Q_{3,K,L},m,n)} F \Big) \cdot \\
    \Big( \cesE_{Q_{3}\in \Phi_{K,L,P_3}} 
    T_{{Q}_3 R_3( r, K, L, N, Q_{1,r}, Q_{2,r,K} , Q_3 ,m,n)} 
    F \Big) \,
    d \mu. 
\end{multline}

The term corresponding to $j=3$ in \eqref{main_term_4} is bounded by 
\begin{equation}\label{aux_quant_1}
\max_{q \in \Phi_{\ell}} \,    \cesE_{(m,n) \in S_{1/q} \cap [N]^2}  \big \lVert 
\cesE_{Q_{1}\in \Phi_{r}} 
T_{Q_1 R_1( r, K, L, N, Q_1, Q_{2,r,K} , Q_{3,K,L} ,m,n)} F_3
\big \rVert_{L^2(\mu)}.
\end{equation}
We will show that the quantity in \eqref{aux_quant_1} goes to $0$ as 
$N\to \infty, L \to \infty, K \to \infty, r \to \infty$. 
Since $\sigma_{F_3}$ is supported on
$\M_p \setminus \A$, using \eqref{spec_eq_1} and applying Fatou's lemma 
multiple times, one sees 
that it suffices to prove that for every 
$f\in \M_p \setminus \A$, we have 
\begin{equation*}
\lim _{r,K,L,N} \,\max_{q \in \Phi_{\ell}} \,\,
\cesE_{(m,n) \in S_{1/q} \cap [N]^2} \big | 
\cesE_{Q_{1}\in \Phi_{r}} f({Q_1 
R_1( r, K, L, N, Q_1, Q_{2,r,K} , Q_{3,K,L} ,m,n)}) \big |=0.
\end{equation*}
Let $f\in \M_p \setminus \A$, let $\chi, t$ so that $f\sim \chi \cdot n^{it}$,
and suppose that $r,L$ are large enough so that the conductor $q_{\chi}$ of 
$\chi$ divides $\frac{Q_{ L}}{Q_1}$, $Q_1$ for all 
$Q_1 \in \Phi_r$. 
Using Lemma \ref{lemma_choice_of_v_1} \eqref{E1: lemma_choice_of_v_1} and
applying Proposition \ref{linear_conc} and Remark \ref{uniform_in_Q}, we deduce that 
\begin{align}\label{eq_useful_1}
&\limsup_{N\to \infty} \,\max_{q \in \Phi_{\ell}} \,\, \cesE_{(m,n) \in S_{1/q} \cap [N]^2}  \big | 
\cesE_{Q_{1}\in \Phi_{r}} f({Q_1 
R_1( r, K, L, N, Q_1, Q_{2,r,K} , Q_{3,K,L} ,m,n)})
\big | \ll_{\ell} \\
&\limsup_{N\to \infty} \,\max_{q \in \Phi_{\ell}} \,\,
\cesE_{(m,n) \in S_{1/q} \cap [N]^2}  \big | 
\cesE_{Q_{1}\in \Phi_{r}}  f(Q_1) \,  Q_1 ^{-it} \,  f(2a) 
\, Q_{L}^{2it} \, m^{it} \, n^{it}
\exp(2F_N(f,L)) \big |   \nonumber \\
& \hspace{25mm} + \D(f, \chi \cdot n^{it}; L, \infty)
+L^{-1/2}
= \nonumber \\
& \limsup_{N\to \infty} |\exp(2F_N(f,L))|\cdot 
\big| \cesE_{Q_1 \in \Phi_r} f(Q_1)\, Q_1 ^{-it}\, 
\big|+ \D(f, \chi \cdot n^{it}; L, \infty)
+L^{-1/2}. \nonumber
\end{align}
Note that $|\exp(2F_N(f,L))|=\exp(2\Re(F_N(f,L))) \to |\exp(-2\D^2(f, \chi \cdot n^{it};L, \infty))|$ as $N\to \infty$.
Since $f \in \M_p \setminus \A$, we have ${f(n)} \, n^{-it}\neq 1$,  and one easily deduces  that 
\begin{equation*}
   \lim_{r\to\infty} \cesE_{Q_1 \in \Phi_r} {f(Q_1)} \, Q_1^{-it} = 0. 
\end{equation*}
Combining this with \eqref{eq_useful_1} we infer that  for every $\ell>0$ 
\begin{equation}\label{aux_eq_3}
    \lim_{r,K,L,N}\, \max_{q \in \Phi_{\ell}} \,\,
    \cesE_{(m,n) \in S_{1/q}\cap [N]^2}  \big \lVert 
\cesE_{Q_{1}\in \Phi_{r}} 
T_{Q_1 R_1( r, K, L, N, Q_1, Q_{2,r,K} , Q_{3,K,L} ,m,n)} F_3
\big \rVert_{L^2(\mu)}=0.
\end{equation}

Using \eqref{main_term_4} and \eqref{aux_eq_3} we see 
that in order to prove \eqref{main_term_3}, it suffices to show that
\begin{multline}\label{main_term_fix_1}
\inflim_{\, \ell,r, K, L,N}
    \int_{X}  \cesE_{q\in \Phi_{\ell}}  \cesE_{(m,n) \in S_{1/q} \cap [N]^2} \\
    \Big( \cesE_{Q_{1}\in \Phi_{r}} T_{Q_1 R_1( r, K, L, N, Q_1, Q_{2,r,K} , Q_{3,K,L} ,m,n)} F_1 \Big) \cdot 
    \Big( \cesE_{Q_{2}\in \Phi_{r,K}}  
     T_{Q_2 R_2( r, K, L, N, Q_{1,r}, 
    Q_2 , Q_{3,K,L},m,n)} F \Big) \cdot \\
    \Big( \cesE_{Q_{3}\in \Phi_{K,L,P_3}} 
    T_{{Q}_3 R_3( r, K, L, N, Q_{1,r}, Q_{2,r,K} , Q_3 ,m,n)} 
    F \Big) \,
    d \mu \geq \Big(\int_X F \, d \mu\Big)^3.
\end{multline}

Using the $T$-invariance of $\mu$, we may rewrite \eqref{main_term_fix_1} 
as
\begin{multline}\label{main_term_fix_8}
\inflim_{\, \ell,r, K, L,N}
    \int_{X} F \cdot \cesE_{q\in \Phi_{\ell}}\,  \cesE_{(m,n) \in S_{1/q} \cap [N]^2}
    \, \cesE_{Q_{3}\in \Phi_{K,L,P_3}} \\
    \Big( \cesE_{Q_{1}\in \Phi_{r}} T_{\frac{Q_1 R_1( r, K, L, N, Q_1, Q_{2,r,K} , Q_{3,K,L} ,m,n)}{{Q}_3 R_3( r, K, L, N, Q_{1,r}, Q_{2,r,K} , Q_3 ,m,n)}} F_1 \Big) \cdot 
    \Big( \cesE_{Q_{2}\in \Phi_{r,K}}  
     T_{\frac{Q_2 R_2( r, K, L, N, Q_{1,r}, 
    Q_2 , Q_{3,K,L},m,n)}{{Q}_3 R_3( r, K, L, N, Q_{1,r}, Q_{2,r,K} , Q_3 ,m,n)}} F \Big) \,
    d \mu  \\ \geq \Big(\int_X F \, d \mu\Big)^3.
\end{multline}

{\bf Step 2 (Concentration for $F_1$).}  
We will now show that we can replace the expression involving 
$F_1$ by $T_{1/q} F_1$. More precisely, we will show that 
\begin{multline}\label{eq_useful_fix_5}
    \lim_{\ell, r, K, L, N}\, \max_{q \in \Phi_{\ell}} \,\,
    \max_{Q_1 \in \Phi_{r}, Q_3 \in \Phi_{K,L,P_3}} 
    \cesE_{(m,n) \in S_{1/q} \cap [N]^2} \\ 
    \Big \lVert 
    T_{\frac{Q_1 R_1( r, K, L, N, Q_1, 
    Q_{2,r,K} , Q_{3,K,L},m,n)}{{Q}_3 R_3( r, K, L, N, Q_{1,r}, Q_{2,r,K} , Q_3 ,m,n)}} F_1
    -T_{1/q} F_1 \Big \rVert_{L^2(\mu)} =0.
\end{multline}
For $f \in \A$ with $f=n^{it}$, using Taylor expansion we have 
\begin{align}\label{eq_useful_fix_6} 
    &\max_{q \in \Phi_{\ell}} \,\,
    \max_{Q_1 \in \Phi_{r}, Q_3 \in \Phi_{K,L,P_3}} 
    \cesE_{(m,n) \in S_{1/q} \cap [N]^2} \\ 
    & \hspace{2cm} \big | f(Q_1 R_1( r, K, L, N, Q_1, 
    Q_{2,r,K} , Q_{3,K,L},m,n)) \cdot  \nonumber \\
    & \hspace{2cm}
    \overline{f(Q_3 R_3( r, K, L, N, Q_{1,r}, Q_{2,r,K} , Q_3 ,m,n))} - f(1/q) \big| =  \nonumber\\
    &\max_{q \in \Phi_{\ell}} \,\,
    \max_{Q_1 \in \Phi_{r}, Q_3 \in \Phi_{K,L,P_3}} 
    \Big( \cesE_{(m,n) \in S_{1/q} \cap [N]^2} \big | Q_L ^{2it} P_1(m,n) ^{it}
    Q_L ^{-2it} P_3(m,n) ^{-it} - q^{-it}\big | \nonumber \\ 
    &\hspace{46mm}  + \oh_{t, Q_1,
     Q_3,  Q_{1,r}, Q_{2,r,K}, Q_{3,K,L}, N\to \infty} (1) \Big) \nonumber = \\
     &\max_{q \in \Phi_{\ell}} \,\,
    \max_{Q_1 \in \Phi_{r}, Q_3 \in \Phi_{K,L,P_3}} 
    \Big( \cesE_{(m,n) \in S_{1/q} \cap [N]^2} \big | P_1(m,n) ^{it}
    P_3(m,n) ^{-it} - q^{-it}\big | \nonumber \\ 
    &\hspace{46mm}  + \oh_{t, Q_1,
     Q_3,  Q_{1,r}, Q_{2,r,K}, Q_{3,K,L}, N\to \infty} (1) \Big) \nonumber
\end{align}
From \eqref{arch conc S_d a=c} we have that for $(m,n) \in S_{1/q}$,
$\big| P_1(m,n)^{it} P_3(m,n) ^{-it} -q ^{-it} \big| \leq |t| /q$,
so the last expression in 
\eqref{eq_useful_fix_6} is 
\begin{multline*}
    \leq \max_{q \in \Phi_{\ell}} \,\,
    \max_{Q_1 \in \Phi_{r}, Q_3 \in \Phi_{K,L,P_3}} \big( {|t|}/{q} + 
    \oh_{t, Q_1, Q_3,  Q_{1,r}, Q_{2,r,K}, Q_{3,K,L}, N\to \infty} (1) \Big) \\
    =\oh_{r,K,L, N\to \infty} (1)+ \max_{q\in \Phi_{\ell}} \frac{|t|}{q}= \oh_{r,K,L, N\to \infty} (1) + \frac{|t|}{\min_{q\in \Phi_{\ell}}q}=
    \oh_{r,K,L, N\to \infty} (1) + \oh_{t, \ell \to \infty}(1).
\end{multline*}
We infer that for $f \in \A$, the quantity in \eqref{eq_useful_fix_6} goes to $0$ 
as $N\to \infty, L \to \infty, K \to \infty, r \to \infty, \ell \to \infty$, and 
since the spectral measure of $F_1$ is supported
on $\A$, \eqref{eq_useful_fix_5} follows.
 
In view of \eqref{eq_useful_fix_5}, to prove \eqref{main_term_fix_8},
it suffices to show that 
\begin{multline}\label{main_term_fix_6}
\inflim_{\, \ell,r, K, L,N}
    \int_{X} F \cdot \cesE_{q\in \Phi_{\ell}}\, T_{1/q} F_1\cdot  \cesE_{(m,n) \in S_{1/q} \cap [N]^2}
    \, \cesE_{Q_{3}\in \Phi_{K,L,P_3}} \\
    \cesE_{Q_{2}\in \Phi_{r,K}}  
     T_{\frac{Q_2 R_2( r, K, L, N, Q_{1,r}, 
    Q_2 , Q_{3,K,L},m,n)}{{Q}_3 R_3( r, K, L, N, Q_{1,r}, Q_{2,r,K} , Q_3 ,m,n)}} F  \,
    d \mu \geq \Big(\int_X F \, d \mu\Big)^3.
\end{multline}

\smallskip

{\bf Step 3 (Reduction to $\X_{p, \textnormal{\textup{fin.supp.}}}$ in 
the average over $Q_2$).}  Writing $F=F_2 + F_4$ where $F_2:=\E(F \mid \X_{p, \textup{fin.supp.}})$ 
and $F_4 \perp F_2$, one sees that the main term
\eqref{main_term_fix_6} is equal to 
\begin{multline}\label{main_term_6}
    \sum_{j\in \{2,4\}}\int_{X} F \cdot \cesE_{q\in \Phi_{\ell}}\, T_{1/q} F_1\cdot  \cesE_{(m,n) \in S_{1/q} \cap [N]^2}
    \, \cesE_{Q_{3}\in \Phi_{K,L,P_3}} \\
    \cesE_{Q_{2}\in \Phi_{r,K}}  
     T_{\frac{Q_2 R_2( r, K, L, N, Q_{1,r}, 
    Q_2 , Q_{3,K,L},m,n)}{{Q}_3 R_3( r, K, L, N, Q_{1,r}, Q_{2,r,K} , Q_3 ,m,n)}} F_j  \,
    d \mu.
\end{multline}
The term corresponding to $j=4$ in \eqref{main_term_6} is bounded by 
\begin{align}\label{aux_quant_4}
&\cesE_{q\in \Phi_{\ell}} \cesE_{(m,n) \in S_{1/q} \cap [N]^2}
    \, \cesE_{Q_{3}\in \Phi_{K,L,P_3}} \big \lVert 
\cesE_{Q_{2}\in \Phi_{r,K}} 
T_{\frac{Q_2 R_2( r, K, L, N, Q_{1,r},  Q_2 , Q_{3,K,L},m,n)}{{Q}_3 R_3( r, K, L, N, Q_{1,r}, Q_{2,r,K} , Q_3 ,m,n)}} F_4
\big \rVert_{L^2(\mu)} \\
= & \cesE_{q\in \Phi_{\ell}} \cesE_{(m,n) \in S_{1/q} \cap [N]^2}
\big \lVert \cesE_{Q_{2}\in \Phi_{r,K}} 
T_{Q_2 R_2( r, K, L, N, Q_{1,r},  Q_2 , Q_{3,K,L},m,n)} F_4
\big \rVert_{L^2(\mu)} \nonumber \\
\leq & \max_{q \in \Phi_{\ell}} \,    \cesE_{(m,n) \in S_{1/q} \cap [N]^2}  \big \lVert 
\cesE_{Q_{2}\in \Phi_{r,K}} 
T_{Q_2 R_2( r, K, L, N, Q_{1,r},  Q_2 , Q_{3,K,L},m,n)} F_4
\big \rVert_{L^2(\mu)}. \nonumber
\end{align}
We will show that the last quantity in \eqref{aux_quant_4} goes to $0$ as 
$N\to \infty, L \to \infty, K \to \infty, r \to \infty$. 
Since $\sigma_{F_4}$ is supported on
$\M_p \setminus \M_{p, \textup{fin.supp.}}$,  using \eqref{spec_eq_1}
and applying Fatou's lemma multiple times, we see 
that it suffices to show that if 
$f\in \M_p \setminus \M_{p, \textup{fin.supp.}}$, then
\begin{equation*}
\lim _{r,K,L,N} \,\max_{q \in \Phi_{\ell}} \,\,
\cesE_{(m,n) \in S_{1/q} \cap [N]^2} \big | 
\cesE_{Q_{2}\in \Phi_{r,K}} f(Q_2 
R_2( r, K, L, N, Q_{1,r},  Q_2 , Q_{3,K,L},m,n)) \big |=0.
\end{equation*}
Let $f\in \M_p \setminus \M_{p, \textup{fin.supp.}}$, 
let $\chi, t$ so that $f\sim \chi \cdot n^{it}$,
and suppose that $r,L$ are large enough so that the conductor $q_{\chi}$ of 
$\chi$ divides $Q_{1,r}$, $\frac{Q_{ L}}{Q_2}$ for all $Q_2 \in \Phi_{r,K}$. 
Using Lemma \ref{lemma_choice_of_v_1} \eqref{E2: lemma_choice_of_v_1},
\eqref{E3: lemma_choice_of_v_1}, and
applying Proposition \ref{linear_conc}, Remark \ref{uniform_in_Q},
and Lemma \ref{bilinear_lemma}, we deduce that 
\begin{align}\label{eq_useful_3}
&\limsup_{N\to \infty} \,\max_{q \in \Phi_{\ell}} \,\, \cesE_{(m,n) \in S_{1/q}  \cap [N]^2} \big | 
\cesE_{Q_{2}\in \Phi_{r,K}} f({Q_2
R_2( r, K, L, N, Q_{1,r}, Q_{2} , Q_{3,K,L} ,m,n)})
\big | \ll_{a,b, \ell} \\
&\limsup_{N\to \infty} \,\max_{q \in \Phi_{\ell}} \,\, 
\cesE_{(m,n) \in S_{1/q} \cap [N]^2}  \big | 
\cesE_{Q_{2}\in \Phi_{r,K}}  f({Q_2})\,  \overline{\chi(Q_2)} \, Q_2 ^{-it}
  \,    P_2(m,n)^{it} \, Q_{ L}^{2it}
     \exp(2F_N(f,L)) \big |   \nonumber \\
& \hspace{25mm} + \D(f, \chi \cdot n^{it}; L, \infty)
+L^{-1/2}
= \nonumber \\
& \limsup_{N\to \infty} |\exp(2F_N(f,L))|\cdot 
\big| \cesE_{Q_2 \in \Phi_{r,K}} f(Q_2)\, \overline{\chi(Q_2)} \, Q_2 ^{-it}
\big|+ \D(f, \chi \cdot n^{it}; L, \infty)
+L^{-1/2}. \nonumber
\end{align}
Using that $f\sim \chi \cdot n^{it}$, as before we infer that 
$\lim_{L\to \infty} \lim_{N\to \infty} |\exp(2F_N(f,L))|=1$.
Since $f \in \M_p \setminus \M_{p, \textup{fin.supp.}}$, for each $r$ there is 
$p>r$ so that $f(p) \neq \chi(p)\,  p^{it}$, and since $\Phi_{r,K}$ is asymptotically
invariant under dilation by $p$ as $K\to \infty$ we have 
\begin{multline*}
    {f(p)}\, \overline{\chi(p)}\, p^{-it} \,
    \cesE_{Q_{2} \in \Phi_{r,K}} f(Q_{2})\, \overline{\chi(Q_{2})}\, Q_{2} ^{-it}
    =
    \cesE_{Q_{2} \in p \Phi_{r,K}} f(Q_{2}) \, \overline{\chi(Q_{2})} \,Q_{2} ^{-it}
    \\ =
    \cesE_{Q_{2} \in  \Phi_{r,K}} f(Q_{2})\, \overline{\chi(Q_{2})} \, Q_{2} ^{-it}
   + \oh_{K\to \infty}(1).
\end{multline*}
Using $f(p) \neq \chi(p)\, p^{it}$, the above implies that 
\begin{equation*}
   \lim_{K\to \infty} \cesE_{Q_{2} \in  \Phi_{r,K}} \, f(Q_{2}) 
 \,   \overline{\chi(Q_{2})} \, Q_{2} ^{-it} = 0.
\end{equation*}
Combining this with \eqref{eq_useful_3} we infer that for every $\ell$ we have
\begin{equation}\label{aux_eq_4}
   \lim_{r,K,L,N}\, \max_{q \in \Phi_{\ell}} \,\, \cesE_{(m,n) \in S_{1/q} \cap [N]^2}
   \big \lVert 
\cesE_{Q_{2}\in \Phi_{r,K}} 
T_{Q_2 R_2( r, K, L, N, Q_{1,r}, Q_{2} , Q_{3,K,L} ,m,n)} F_4
\big \rVert_{L^2(\mu)}=0.
\end{equation}

Therefore, using \eqref{main_term_6} and \eqref{aux_eq_4}, we see 
that in order to prove \eqref{main_term_fix_6}, it suffices to show
that
\begin{multline}\label{main_term_fix_2}
\inflim_{\, \ell,r, K, L,N}
    \int_{X} F \cdot \cesE_{q\in \Phi_{\ell}}\, T_{1/q} F_1\cdot  \cesE_{(m,n) \in S_{1/q} \cap [N]^2}
    \, \cesE_{Q_{3}\in \Phi_{K,L,P_3}} \\
    \cesE_{Q_{2}\in \Phi_{r,K}}  
     T_{\frac{Q_2 R_2( r, K, L, N, Q_{1,r}, 
    Q_2 , Q_{3,K,L},m,n)}{{Q}_3 R_3( r, K, L, N, Q_{1,r}, Q_{2,r,K} , Q_3 ,m,n)}} F_2  \,
    d \mu \geq \Big(\int_X F \, d \mu\Big)^3.
\end{multline}

\smallskip

{\bf Step 4 (Concentration for $F_2$).}
We will now show that we can replace the expression involving 
$F_2$ by $F_2$ itself. To do it, we show that  
\begin{multline}\label{eq_useful_4}
    \lim_{\ell, r, K, L, N}\, \max_{q \in \Phi_{\ell}} \,\,
    \max_{Q_2 \in \Phi_{r,K}, Q_3 \in \Phi_{K,L, P_3}} 
    \cesE_{(m,n) \in S_{1/q} \cap [N]^2} \\ 
    \Big \lVert 
    T_{\frac{Q_2 R_2( r, K, L, N, Q_{1,r}, 
    Q_2 , Q_{3,K,L},m,n)}{{Q}_3 R_3( r, K, L, N, Q_{1,r}, Q_{2,r,K} , Q_3 ,m,n)}} F_2
    -F_2 \Big \rVert_{L^2(\mu)} =0.
\end{multline}

To see this, let $f\in \M_{p, \textup{fin.supp.}}$, and
let $\chi, t$ and $r_0 \in \N$ so that 
$f(p)=\chi(p)\, p^{it}$ for all primes $p>r_0$. Also,
suppose that $r,L$ are sufficiently
large so that $r>r_0$ and the conductor $q_{\chi}$ of $\chi$ 
divides $Q_{1,r}$, $\frac{Q_{ L}}{Q_2}$ and 
$\frac{Q_{ L}}{Q_3}$ for every
$Q_{2} \in \Phi_{r,K}, Q_{3} \in \Phi_{K, L, P_3}$.
Using Lemma \ref{lemma_choice_of_v_1} \eqref{E2: lemma_choice_of_v_1},
\eqref{E3: lemma_choice_of_v_1}, \eqref{E4: lemma_choice_of_v_1} and
applying Propositions \ref{linear_conc} and \ref{quad_conc}, Remark
\ref{uniform_in_Q}, and 
Lemma \ref{bilinear_lemma}, we get that
\begin{align}\label{eq_useful_6}
    &\limsup_{N\to \infty} \, \max_{q \in \Phi_{\ell}} \,\,
    \max_{Q_2 \in \Phi_{r,K}, Q_3 \in \Phi_{K,L,P_3}} \,
    \cesE_{(m,n) \in S_{1/q} \cap [N]^2} 
    \big|f(Q_2) \overline{f(Q_3)} \cdot \\
    & f(R_2( r, K, L, N, Q_{1,r}, Q_{2} , Q_{3,K,L} ,m,n))
     \overline{f(R_3( r, K, L, N, Q_{1,r}, Q_{2,r,K} , Q_3 ,m,n))}-1\big| \leq 
   \nonumber \\
   & \Oh_{a,b,\ell}\big(\D(f, \chi \cdot n^{it}, L, \infty) +  L^{-1/2}\big)+  \limsup_{N\to \infty} \, \max_{q \in \Phi_{\ell}} \,\,
    \max_{Q_2 \in \Phi_{r,K}, Q_3 \in \Phi_{K,L,P_3}} \, \nonumber \\
    &\cesE_{(m,n) \in S_{1/q} \cap [N]^2} 
    |f(Q_{2}) \overline{f(Q_3)} \, \overline{\chi(Q_{2})}  \, Q_{2}^{-it}  
    P_2(m,n)^{it} \, Q_{ L}^{2it}  \,
      \exp(2F_N(f,L)) \chi(Q_{3}) Q_3^{it} \cdot \nonumber \\  
      & \hspace{25mm}  P_3(m,n)^{-it}
      Q_{ L}^{-2it}\,
      \overline{\exp(G_{P_3,N}(f,L))}-1|.  \nonumber
\end{align}

For $p>r$, $f(p)=\chi(p)\, p^{it}$, which implies that 
$f(Q_{2})= \chi(Q_{2}) \, Q_{2}^{it}$, 
$f(Q_{3})= \chi(Q_{3}) \, Q_{3}^{it}$, and since $L>r_0$, we have
$F_N(f,L)=G_{P_3, N}(f,L)=0$. Therefore, the last quantity in \eqref{eq_useful_6} 
is equal to 
\begin{equation}\label{useful_eq_fix_1}
\oh_{\ell, L\to \infty}(1) + \limsup_{N\to \infty} \, \max_{q \in \Phi_{\ell}} \,\,
    \max_{Q_2 \in \Phi_{r,K}, Q_3 \in \Phi_{K,L,P_3}} \,
    \cesE_{(m,n) \in S_{1/q} \cap [N]^2} 
    |  P_2(m,n)^{it} P_3(m,n)^{-it}   
       -1|.
\end{equation}
From \eqref{arch conc S_d a=c} we have that for $(m,n) \in S_{1/q}$,
$\big| P_2(m,n)^{it} P_3(m,n) ^{-it} -1 \big| \leq |t| /q$,
so the expression in 
\eqref{useful_eq_fix_1} is 
\begin{equation*}
    \leq \oh_{\ell, L\to \infty}(1) + \max_{q\in \Phi_{\ell}} \frac{|t|}{q}= 
    \oh_{\ell, L\to \infty}(1) + \frac{|t|}{\min_{q\in \Phi_{\ell}}q}= 
    \oh_{\ell, L\to \infty}(1) + \oh_{t, \ell \to \infty}(1).
\end{equation*}
We infer that 
for $f\in \M_{p, \textup{fin.supp.}}$, we have 
\begin{multline*}
    \lim_{\ell, r, K, L, N}\, \max_{q \in \Phi_{\ell}} \,\,
    \max_{Q_2 \in \Phi_{r,K}, Q_3 \in \Phi_{K,L,P_3}} 
    \cesE_{(m,n) \in S_{1/q} \cap [N]^2}
    \big|f(Q_2) \overline{f(Q_3)} \cdot \\
     f(R_2( r, K, L, N, Q_{1,r}, Q_{2} , Q_{3,K,L} ,m,n))
     \overline{f(R_3( r, K, L, N, Q_{1,r}, Q_{2,r,K} , Q_3 ,m,n))}-1\big|=0,
\end{multline*}
and since the spectral measure of $F_2$ is supported on 
$\M_{p, \textup{fin.supp.}}$, \eqref{eq_useful_4} follows. 

Now, in view of \eqref{eq_useful_4} we see that to prove 
\eqref{main_term_fix_2}, it suffices to show that 
\begin{multline}\label{main_term_fix_5}
\inflim_{\, \ell,r, K, L,N}
    \int_{X} F \cdot \cesE_{q\in \Phi_{\ell}}\, T_{1/q} F_1\cdot   F_2  \,
    d \mu = \inflim_{\, \ell}
    \int_{X} F \cdot F_2 \cdot \cesE_{q\in \Phi_{\ell}}\, T_{1/q} F_1  \,
    d \mu \geq \Big(\int_X F \, d \mu\Big)^3.
\end{multline}

\smallskip

{\bf Step 5 (Finishing the argument).}  
From the mean ergodic theorem, we have that 
$\cesE_{q \in \Phi_{\ell}} T_{1/q} F_1 \to \E(F_1 \mid \mathcal{I})$ in $L^2(\mu)$, where $\mathcal{I}:=\{B\in \CX\colon T_n^{-1}B=B, \text{ for all } n\in\N\}$ is the invariant factor. 
Note that $\mathcal{I} \subset \X_{\A}$, and since 
$F_1:=\E(F \mid {\X}_{\A})$, we have that 
$\E(F_1 \mid \mathcal{I})=\E(F \mid \mathcal{I})$. Therefore we obtain 
\begin{multline*}
   \lim_{\ell \to \infty}
    \int_{X} F \cdot F_2 \cdot \cesE_{q \in \Phi_{\ell}}
    T_{1/q} F_1 \, d \mu = \int_{X} F \cdot F_2 \cdot 
    \E(F \mid  \mathcal{I}) \,d \mu = \\
    \int_{X} F \cdot  \E(F \mid \X_{p, \textup{fin.supp.}}) \cdot 
    \E(F \mid \mathcal{I}) \,  d \mu
    \geq 
    \Big(\int_X F \, d \mu\Big)^3,
\end{multline*}
where for the last inequality we used Proposition
\ref{chu-lemma}. This concludes the proof. 
\end{proof}}

\subsection{$P_2$ irreducible}\label{Sec : a=c_P_2_irreducible}
Here we handle  the case where $P_2$ is irreducible.  
A representative example is the equation $x^2+2y^2=z^2$, in which case  
\eqref{quadratic_forms_a=c} gives 
$$
P_1(m,n)=2mn,   \quad P_2(m,n) = m^2-2n^2, \quad  
P_3(m,n)=m^2+2n^2.
$$
The proof of this case is quite similar to the case when $P_2$ is reducible.

Let $\kappa \in \N$ so that $P_2(1, \kappa)= 1- ab \kappa^2  \neq 0$, and 
from now on, suppose that $r$ is large enough so that
for $p>r$ we have $\omega_{P_2}(p), \omega_{P_3}(p)\in \{0,2\}$ and 
\begin{equation}\label{E2: p not dividing}
p\nmid 2\kappa ab(1+ab)(1- ab \kappa ^2)(1+ ab \kappa ^2).    
\end{equation}

To prove the result, given $Q_1 \in \Phi_r$, $Q_2 \in \Phi_{r,K,P_2}$, 
$Q_3 \in \Phi_{K,L, P_3}$,
we will choose an integer $v$ depending on $r,K,L,Q_1,Q_2, Q_3$, and 
we will average over the grid 
$(Q_{L} m+1, Q_{L} n +v)$ in \eqref{T: main a=c eq}, where $Q_{L}$
is as in \eqref{Q_l_def}.

\begin{lemma}\label{lemma_choice_of_v_1_a=c_P_2_irred}
Let $r,K,L\in \N$ with $r<K<L/2$, $P_1,P_2,P_3$ be as in \eqref{quadratic_forms_a=c}, and let $Q_1 \in \Phi_r$, 
$Q_2 \in \Phi_{r,K,P_2}$, $Q_3 \in \Phi_{K,L,P_3}$. Then there is an integer
$v$ with $0\leq v \leq Q_{L}-1$ such that 
\begin{enumerate}[(i)]
    \item \label{E1: lemma_choice_of_v_1_a=c_P_2_irred} $(Q_{L},v)=Q_1$, 
    $\frac{v}{Q_1}\equiv 1 \pmod {Q_1}$,
    \item \label{E2: lemma_choice_of_v_1_a=c_P_2_irred} 
    $(Q_{L},P_2(1,v))= Q_2$, 
    $\frac{P_2(1,v)}{Q_2} \equiv Q_2 ^{-1} 
    \pmod{{Q_1}}$,
    \item \label{E3: lemma_choice_of_v_1_a=c_P_2_irred} 
    $(Q_{L},P_3(1,v))=Q_3$, 
    $\frac{P_3(1,v)}{Q_3} \equiv Q_3^{-1} \pmod{{Q_1}}$.
\end{enumerate}
\end{lemma}

\begin{proof}
Let $p\in \P$ with $r<p \leq K$ and $\omega_{P_2}(p)=2$, and consider the 
polynomial $g_2(x):=P_2(1,x)-p^{\theta_p(Q_2)}$. 
Then $\textup{mod}\: p$ the polynomial
becomes $g_2(x)=P_2(1,x)$, and it has two simple roots
$z_2, z_2'$ with $z_2 \not \equiv z_2' \pmod p$. Using  
Hensel's lemma, we may lift $z_2$ to a unique root 
$\zeta_2 = \zeta_2 (p, Q_2)$ of $g_2 \mod {p^{\theta_p(Q_2)+1}}$.

Analogously, for $p\in \P$ with $K<p \leq L$ and $\omega_{P_3}(p)=2$, 
we may find $\zeta_3 = \zeta_3 (p, Q_3)$ so that if 
$g_3(x):=P_3(1,x)-p^{\theta_p(Q_3)}$, then $g_3(\zeta_3)\equiv 0 
\pmod {p^{\theta_p(Q_3)+1}}$.

Using the Chinese Remainder theorem we may then choose $v$ with
$0\leq v \leq Q_{L}-1$
such that 
\begin{itemize}
     \item $v\equiv Q_1 \pmod {p^{2 \theta_p(Q_1)}}$ for 
    $p\leq r$,
    \item $v\equiv \kappa$
    $\pmod {p}$ for $r< p\leq K$ with $\omega_{P_2}(p)=0$,
    \item $v\equiv \zeta_2(p, Q_2)$
    $\pmod {p^{\theta_p(Q_2)+1}}$ for $r< p\leq K$ with 
    $\omega_{P_2}(p)=2$,
    \item $v\equiv \kappa \pmod {p}  $ for $K<p\leq L$ with $\omega_{P_3}(p)=0$,
    \item $v\equiv \zeta_3(p,Q_3) \pmod {p^{\theta_p(Q_3)+1}}$ for 
    $K<p\leq L$ with $\omega_{P_3}(p)=2$.
\end{itemize}
The rest of the proof is very similar to the proof of 
Lemma \ref{lemma_choice_of_v_1}, so it is omitted.
\end{proof}

{ 
As in Section \ref{Sec : a=c_P_2_reducible}, 
given $r, K, L, N\in \N$ with 
$r<K<L<N $, $Q_1 \in \Phi_r$, $Q_2\in \Phi_{r,K,P_2}$, 
$Q_3\in \Phi_{K,L,P_3}$, and $m,n\in \N$, we let 
\begin{enumerate}
\item \label{R_1 def irred} $R_1 (r, K, L, N, Q_1, Q_2, Q_3,m,n):= 2a
({Q_{L}} m+1)
(\frac{Q_{L}}{Q_1}n+ \frac{v}{Q_1} )$,
\item \label{R_2 def irred} 
$R_2 ( r, K, L, N, Q_1, Q_2, Q_3,m,n):= 
\frac{P_2(Q_{L} m +1, Q_{L} n +v )}{Q_2}$,
\item \label{R_3 def irred} 
$R_3 ( r, K, L, N, Q_1, Q_2, Q_3,m,n):= 
\frac{P_3(Q_{L} m +1, Q_{L} n +v )}{Q_3}$,
\end{enumerate}
where $v$ is as in Lemma \ref{lemma_choice_of_v_1_a=c_P_2_irred}.
Then  for $j\in \{1,2,3\}$ we have 
\begin{equation}\label{eq_extr_1_irred}
    P_j(Q_{L} m+1, Q_{L} n +v)= Q_j \,
    R_j (r, K, L, N, Q_1, Q_2, Q_3,m,n).
\end{equation}
Recall that for $(m,n) \in S_{1/q}$, $m>2\sqrt{ab}n$, 
so $Q_L m+1 > \sqrt{ab} (Q_Ln+v)$, which implies that 
$P_2(Q_L m+1, Q_L n+v)>0$, while for $j=1,3,$ we have
$P_j(Q_Lm+1, Q_Ln+v)>0$ for all $m,n \in \N$. 

When $P_2$ is irreducible, Theorem \ref{T: main a=c} 
follows from the next result.}

\begin{proposition}\label{P: main a=c P_2 irreducible}
Let $(X,\CX,\mu,T_n)$  be a pretentious multiplicative action, let 
$F\in L^{\infty}(\mu)$ with $F\geq 0$, 
let $P_1,P_2,P_3$ be as in \eqref{quadratic_forms_a=c},  and suppose that $P_2$ is irreducible. 
Also, let $S_{1/q}$ be as in \eqref{S_delta_a=c_P_2_reduc_fixed} and 
$v=v(r,K,L,Q_1,Q_2,Q_3)$ be  as in 
Lemma \ref{lemma_choice_of_v_1_a=c_P_2_irred}.
Then 
    \begin{multline}\label{E: P: main a=c P_2 irreducible}
    \inflim_{\, \ell, r,K,L,N}\: \cesE_{q\in \Phi_{\ell}}
    \cesE_{Q_{1}\in \Phi_{r}} \cesE_{Q_{2}\in \Phi_{r,K,P_2}}
    \cesE_{Q_{3}\in \Phi_{K,L,P_3}}
   \cesE_{(m,n) \in S_{1/q} \cap [N]^2}   \\
    \int_{X} \prod_{j=1} ^3
    T_{P_j(Q_{L} m+1, Q_{L} n +v)} F 
   \,  d \mu \geq  \Big(\int_X F \, d \mu\Big)^3.
    \end{multline}
\end{proposition}

The proof that Proposition \ref{P: main a=c P_2 irreducible} implies Theorem
\ref{T: main a=c} when $P_2$ is irreducible is identical to the case
when $P_2$ is reducible, so it is omitted.
To prove Proposition \ref{P: main a=c P_2 irreducible}, one first has to prove an 
analog of Lemma \ref{L: main a=c P_2 reducible}, which asserts that in the 
integral in \eqref{E: P: main a=c P_2 irreducible},
the first term is asymptotically independent of $Q_2, Q_3$, 
the second term 
is asymptotically independent of $Q_1, Q_3$ and the third term 
is asymptotically independent of $Q_1, Q_2$, and then using that 
prove Proposition
\ref{P: main a=c P_2 irreducible}. Both the statement and the proof
of this lemma are identical to those of Lemma \ref{L: main a=c P_2 reducible},
so they are omitted. It remains to prove Proposition 
\ref{P: main a=c P_2 irreducible}. Since the proof is very similar to 
the proof of Proposition \ref{P: main a=c P_2 reducible}, we only give
a sketch. Recall from Section~\ref{SS:factors} the definition of the factors
$\X_{\A}$ and $\X_{p,\textup{fin.supp.},P_j}$, $j=2,3$,
whose corresponding $L^2(\mu)$ subspaces are defined in
\eqref{def_X_A} and \eqref{def_X_p_fin_supp_P}, respectively.
We will also use the intersection of the latter two factors.

\begin{proof}[Sketch of proof of Proposition \ref{P: main a=c P_2 irreducible}]
By homogeneity, we may assume that
	$\lVert F\rVert_{L^{\infty}(\mu)}\leq 1$, 
and it is always implicit that $r<K<L<N$, and that each of the 
previous variables is 
sufficiently large depending on the smaller ones.

For each $r,K,L\in \N$, choose any $Q_{1,r}\in \Phi_r$,  
$Q_{2, r, K} \in \Phi_{r,K, P_2}$ and $Q_{3,K,L}\in \Phi_{K,L,P_3}$. 
Using the aforementioned analog of Lemma 
\ref{L: main a=c P_2 reducible} and repeating Steps $1$ and $2$
in the proof
of Proposition \ref{P: main a=c P_2 reducible}, we see that in order
to prove \eqref{E: P: main a=c P_2 irreducible}, it suffices to show 
that 
\begin{multline}\label{main_term_fix_3_irred}
\inflim_{\, \ell,r, K, L,N}
    \int_{X} F \cdot \cesE_{q\in \Phi_{\ell}}\, T_{1/q} F_1\cdot  \cesE_{(m,n) \in S_{1/q} \cap [N]^2}
    \, \cesE_{Q_{3}\in \Phi_{K,L,P_3}} \\
    \cesE_{Q_{2}\in \Phi_{r,K, P_2}}  
     T_{\frac{Q_2 R_2( r, K, L, N, Q_{1,r}, 
    Q_2 , Q_{3,K,L},m,n)}{{Q}_3 R_3( r, K, L, N, Q_{1,r}, Q_{2,r,K} , Q_3 ,m,n)}} F  \,
    d \mu \geq \Big(\int_X F \, d \mu\Big)^3,
\end{multline}
where in the previous, $F_1= \E(F \mid \X_\A)$.

\smallskip

{\bf Step 3 (Reduction to 
$\X_{p, \textnormal{\textup{fin.supp.}},P_2}\cap
\X_{p, \textnormal{\textup{fin.supp.}},P_3}$ in 
the average over $Q_2, Q_3$).}
Since intersections of factors are also factors, $\mathcal{Y}:=\X_{p, \textnormal{\textup{fin.supp.}},P_2}\cap
\X_{p, \textnormal{\textup{fin.supp.}},P_3}$ is also a factor of 
$(X, \X, \mu, T_n)$. Write $F=F_2 + F_4$, where $F_2=\E( F \mid \mathcal{Y})$ and $F_4 \perp F_2$. The spectral measure 
$\sigma_{F_2}$ of $F_2$ is supported on 
$\M_{p, \textnormal{\textup{fin.supp.}},P_2}\cap 
\M_{p, \textnormal{\textup{fin.supp.}},P_3}$, while the spectral 
measure $\sigma_{F_4}$ of $F_4$ is supported on 
$(\M_p \setminus \M_{p, \textnormal{\textup{fin.supp.}},P_2}) \cup 
(\M _p \setminus \M_{p, \textnormal{\textup{fin.supp.}},P_3})$.

The main term in 
\eqref{main_term_fix_3_irred} is equal to 
\begin{multline}\label{main_term_fix_8_irred}
\sum_{j \in \{2,4\}}
    \int_{X} F \cdot \cesE_{q\in \Phi_{\ell}}\, T_{1/q} F_1\cdot  \cesE_{(m,n) \in S_{1/q} \cap [N]^2}
    \, \cesE_{Q_{3}\in \Phi_{K,L,P_3}} \\
    \cesE_{Q_{2}\in \Phi_{r,K, P_2}}  
     T_{\frac{Q_2 R_2( r, K, L, N, Q_{1,r}, 
    Q_2 , Q_{3,K,L},m,n)}{{Q}_3 R_3( r, K, L, N, Q_{1,r}, Q_{2,r,K} , Q_3 ,m,n)}} F_j  \,
    d \mu.
\end{multline}
The term corresponding to $j=4$ in \eqref{main_term_fix_8_irred} 
is bounded by 
\begin{equation}\label{aux_quant_4_irred}
 \max_{q \in \Phi_{\ell}} \, \cesE_{(m,n) \in S_{1/q} \cap [N]^2} 
 \big \lVert \cesE_{Q_{3}\in \Phi_{K,L,P_3}}\, 
 \cesE_{Q_{2}\in \Phi_{r,K, P_2}} T_{\frac{Q_2 R_2( r, K, L, N, Q_{1,r},  Q_2 , Q_{3,K,L},m,n)}{{Q}_3 R_3( r, K, L, N, Q_{1,r}, Q_{2,r,K} , Q_3 ,m,n)}} F_4
\big \rVert_{L^2(\mu)}.
\end{equation}
We will show that the last quantity in \eqref{aux_quant_4_irred} 
goes to $0$ as 
$N\to \infty, L \to \infty, K \to \infty, r \to \infty$. 
Since $\sigma_{F_4}$ is supported on
$(\M_p \setminus \M_{p, \textnormal{\textup{fin.supp.}},P_2}) \cup 
(\M _p \setminus \M_{p, \textnormal{\textup{fin.supp.}},P_3})$, using \eqref{spec_eq_1}
and applying Fatou's lemma multiple times, we see 
that it suffices to show that if 
$f\in (\M_p \setminus \M_{p, \textnormal{\textup{fin.supp.}},P_2}) \cup 
(\M _p \setminus \M_{p, \textnormal{\textup{fin.supp.}},P_3})$, then
\begin{multline*}
\lim _{r,K,L,N} \,\max_{q \in \Phi_{\ell}} \,\,
\cesE_{(m,n) \in S_{1/q} \cap [N]^2} \big | \cesE_{Q_{3}\in \Phi_{K,L,P_3}}\, 
 \cesE_{Q_{2}\in \Phi_{r,K, P_2}} f(Q_2)\overline{f(Q_3)} \\ 
 f(R_2( r, K, L, N, Q_{1,r},  Q_2 , Q_{3,K,L},m,n)) \cdot 
 \overline{f( R_3( r, K, L, N, Q_{1,r}, Q_{2,r,K} , Q_3 ,m,n))}
 \big |=0.
\end{multline*}
Let $f\in (\M_p \setminus \M_{p, \textnormal{\textup{fin.supp.}},P_2}) \cup 
(\M _p \setminus \M_{p, \textnormal{\textup{fin.supp.}},P_3})$, 
let $\chi, t$ so that $f\sim \chi \cdot n^{it}$,
and suppose that $r,L$ are large enough so that the conductor $q_{\chi}$ of 
$\chi$ divides $Q_{1,r}$, $\frac{Q_{L}}{Q_2}$ and 
$\frac{Q_{ L}}{Q_3}$ for every
$Q_{2} \in \Phi_{r,K,P_2}, Q_{3} \in \Phi_{K, L, P_3}$.
Using Lemma \ref{lemma_choice_of_v_1_a=c_P_2_irred} 
\eqref{E2: lemma_choice_of_v_1_a=c_P_2_irred},
\eqref{E3: lemma_choice_of_v_1_a=c_P_2_irred}, 
and applying Proposition \ref{quad_conc} and Remark
\ref{uniform_in_Q} we get that
\begin{align}\label{eq_useful_3_irred}
&\limsup_{N\to \infty} \,\max_{q \in \Phi_{\ell}} \,\, \cesE_{(m,n) \in S_{1/q} \cap [N]^2} \big | 
\cesE_{Q_{3}\in \Phi_{K,L,P_3}}\, 
 \cesE_{Q_{2}\in \Phi_{r,K, P_2}} f(Q_2)\overline{f(Q_3)} \\
 & f(R_2( r, K, L, N, Q_{1,r},  Q_2 , Q_{3,K,L},m,n)) \cdot 
 \overline{f( R_3( r, K, L, N, Q_{1,r}, Q_{2,r,K} , Q_3 ,m,n))}\big | \ll_{P_2, P_3, \ell} \nonumber
 \\ &  \nonumber \limsup_{N\to \infty} \,\max_{q \in \Phi_{\ell}} \,\, \cesE_{(m,n) \in S_{1/q} \cap [N]^2} \big | 
\cesE_{Q_{3}\in \Phi_{K,L,P_3}}\, 
 \cesE_{Q_{2}\in \Phi_{r,K, P_2}} f(Q_2)\overline{f(Q_3)} 
 \, \overline{\chi(Q_{2})}  \, Q_{2}^{-it} \, P_2(m,n)^{it}
 Q_{ L}^{2it} \cdot  \\
& \nonumber \exp(G_{P_2,N}(f,L)) \chi(Q_{3}) Q_3^{it}
P_3(m,n)^{-it} Q_{ L}^{-2it}\,
      \overline{\exp(G_{P_3,N}(f,L))}  \big |  + 
      \D(f, \chi \cdot n^{it}, L, \infty) +  L^{-1/2} \\
& \nonumber = 
\limsup_{N\to \infty} |\exp(G_{P_2,N}(f,L)) \exp(G_{P_3,N}(f,L))|\cdot 
\big| \cesE_{Q_2 \in \Phi_{r,K, P_2}} f(Q_2)\, \overline{\chi(Q_2)} \, Q_2 ^{-it}
\big|\cdot \\
& \nonumber \big| \cesE_{Q_3 \in \Phi_{K,L,P_3}} f(Q_3)\, \overline{\chi(Q_3)} \, Q_3 ^{-it}
\big|+ \D(f, \chi \cdot n^{it}; L, \infty)
+L^{-1/2}. 
\end{align}
Using that $f\sim \chi \cdot n^{it}$, as before we infer that 
$\lim_{L\to \infty} \lim_{N\to \infty} |\exp(G_{P_2, N}(f,L))|= 
\lim_{L\to \infty} \lim_{N\to \infty} |\exp(G_{P_3, N}(f,L))|=1$.
If $f \in \M_p \setminus \M_{p, \textup{fin.supp.},P_2}$, then
for each $r$ there is 
$p>r$ so that $\omega_{P_2}(p)>0$ and $f(p) \neq \chi(p)\,  p^{it}$, and since $\Phi_{r,K, P_2}$ is asymptotically
invariant under dilation by $p$ as $K\to \infty$ we have 
\begin{multline*}
    {f(p)}\, \overline{\chi(p)}\, p^{-it} \,
    \cesE_{Q_{2} \in \Phi_{r,K,P_2}} f(Q_{2})\, \overline{\chi(Q_{2})}\, Q_{2} ^{-it}
    =
    \cesE_{Q_{2} \in p \Phi_{r,K, P_2}} f(Q_{2}) \, \overline{\chi(Q_{2})} \,Q_{2} ^{-it}
    \\ =
    \cesE_{Q_{2} \in  \Phi_{r,K, P_2}} f(Q_{2})\, \overline{\chi(Q_{2})} \, Q_{2} ^{-it}
   + \oh_{K\to \infty}(1).
\end{multline*}
Using $f(p) \neq \chi(p)\, p^{it}$, the above implies that 
\begin{equation*}
   \lim_{K\to \infty} \cesE_{Q_{2} \in  \Phi_{r,K, P_2}} \, f(Q_{2}) 
 \,   \overline{\chi(Q_{2})} \, Q_{2} ^{-it} = 0.
\end{equation*}
In case $f \in \M_p \setminus \M_{p, \textup{fin.supp.},P_3}$, then
with a similar argument one sees that 
\begin{equation*}
   \lim_{L\to \infty} \cesE_{Q_{3} \in  \Phi_{K, L, P_3}} \, f(Q_{3}) 
 \,   \overline{\chi(Q_{3})} \, Q_{3} ^{-it} = 0.
\end{equation*}
Combining the previous with \eqref{eq_useful_3_irred} we infer that 
for every $\ell$ we have
\begin{equation}\label{aux_eq_4_irred}
   \lim_{r,K,L,N}\, \max_{q \in \Phi_{\ell}} \,\, \cesE_{(m,n) \in S_{1/q} \cap [N]^2}
    \big \lVert \cesE_{Q_{3}\in \Phi_{K,L,P_3}}\, 
 \cesE_{Q_{2}\in \Phi_{r,K, P_2}} T_{\frac{Q_2 R_2( r, K, L, N, Q_{1,r},  Q_2 , Q_{3,K,L},m,n)}{{Q}_3 R_3( r, K, L, N, Q_{1,r}, Q_{2,r,K} , Q_3 ,m,n)}} F_4
\big \rVert_{L^2(\mu)}=0.
\end{equation}

Therefore, using \eqref{main_term_fix_8_irred} and \eqref{aux_eq_4_irred}, we see 
that in order to prove \eqref{main_term_fix_3_irred}, it suffices to show
that
\begin{multline}\label{main_term_fix_9_irred}
\inflim_{\, \ell,r, K, L,N}
    \int_{X} F \cdot \cesE_{q\in \Phi_{\ell}}\, T_{1/q} F_1\cdot  \cesE_{(m,n) \in S_{1/q} \cap [N]^2}
    \, \cesE_{Q_{3}\in \Phi_{K,L,P_3}} \\
    \cesE_{Q_{2}\in \Phi_{r,K, P_2}}  
     T_{\frac{Q_2 R_2( r, K, L, N, Q_{1,r}, 
    Q_2 , Q_{3,K,L},m,n)}{{Q}_3 R_3( r, K, L, N, Q_{1,r}, Q_{2,r,K} , Q_3 ,m,n)}} F_2  \,
    d \mu \geq \Big(\int_X F \, d \mu\Big)^3.
\end{multline}

\smallskip

{\bf Step 4 (Concentration for $F_2$).}
We will now show that we can replace the expression involving 
$F_2$ by $F_2$ itself. To do it, we show that  
\begin{multline}\label{eq_useful_10_irred}
    \lim_{\ell, r, K, L, N}\, \max_{q \in \Phi_{\ell}} \,\,
    \max_{Q_2 \in \Phi_{r,K, P_2}, Q_3 \in \Phi_{K,L, P_3}} 
    \cesE_{(m,n) \in S_{1/q} \cap [N]^2} \\ 
    \Big \lVert 
    T_{\frac{Q_2 R_2( r, K, L, N, Q_{1,r}, 
    Q_2 , Q_{3,K,L},m,n)}{{Q}_3 R_3( r, K, L, N, Q_{1,r}, Q_{2,r,K} , Q_3 ,m,n)}} F_2
    -F_2 \Big \rVert_{L^2(\mu)} =0.
\end{multline}
To see this, let $f\in \M_{p, \textup{fin.supp.}, P_2} \cap 
\M_{p, \textup{fin.supp.}, P_3}$, and
let $\chi, t$ and $r_0 \in \N$ so that for every prime $p>r_0$, if 
either $\omega_{P_2}(p) >0$ or $\omega_{P_3}(p) >0$, then
$f(p)=\chi(p)\, p^{it}$. Also,
suppose that $r,L$ are sufficiently
large so that $r>r_0$ and the conductor $q_{\chi}$ of $\chi$ 
divides $Q_{1,r}$, $\frac{Q_{ L}}{Q_2}$ and 
$\frac{Q_{ L}}{Q_3}$ for every
$Q_{2} \in \Phi_{r,K, P_2}, Q_{3} \in \Phi_{K, L, P_3}$.
Using Lemma \ref{lemma_choice_of_v_1_a=c_P_2_irred} 
\eqref{E2: lemma_choice_of_v_1_a=c_P_2_irred},
\eqref{E3: lemma_choice_of_v_1_a=c_P_2_irred}, 
and applying Proposition \ref{quad_conc} and Remark
\ref{uniform_in_Q} we get that
\begin{align}\label{eq_useful_6_irred}
    &\limsup_{N\to \infty} \, \max_{q \in \Phi_{\ell}} \,\,
    \max_{Q_2 \in \Phi_{r,K, P_2}, Q_3 \in \Phi_{K,L,P_3}} \,
    \cesE_{(m,n) \in S_{1/q} \cap [N]^2} 
    \big|f(Q_2) \overline{f(Q_3)} \cdot \\
    & f(R_2( r, K, L, N, Q_{1,r}, Q_{2} , Q_{3,K,L} ,m,n))
     \overline{f(R_3( r, K, L, N, Q_{1,r}, Q_{2,r,K} , Q_3 ,m,n))}-1\big| \leq 
   \nonumber \\
   & \Oh_{P_2, P_3,\ell}\big(\D(f, \chi \cdot n^{it}, L, \infty) +  L^{-1/2}\big)+  \limsup_{N\to \infty} \, \max_{q \in \Phi_{\ell}} \,\,
    \max_{Q_2 \in \Phi_{r,K, P_2}, Q_3 \in \Phi_{K,L,P_3}} \, \nonumber \\
    &\cesE_{(m,n) \in S_{1/q} \cap [N]^2} 
    |f(Q_{2}) \overline{f(Q_3)} \, \overline{\chi(Q_{2})}  \, Q_{2}^{-it}  
    P_2(m,n)^{it} \, Q_{ L}^{2it}  \,
      \exp(G_{P_2,N}(f,L)) \chi(Q_{3}) Q_3^{it} \cdot \nonumber \\  
      & \hspace{25mm}  P_3(m,n)^{-it}
      Q_{ L}^{-2it}\,
      \overline{\exp(G_{P_3,N}(f,L))}-1|.  \nonumber
\end{align}
Since for $p>r$ with $\omega_{P_2}(p) >0$ or $\omega_{P_3}(p) >0$ we
have $f(p)=\chi(p)\, p^{it}$, we get that  
$f(Q_{2})= \chi(Q_{2}) \, Q_{2}^{it}$, 
$f(Q_{3})= \chi(Q_{3}) \, Q_{3}^{it}$, and since $L>r_0$, we have
$G_{P_2, N}(f,L)=G_{P_3, N}(f,L)=0$. Therefore, the last quantity in \eqref{eq_useful_6_irred} is equal to 
\begin{equation}\label{useful_eq_fix_11_irred}
\oh_{\ell, L\to \infty}(1) + \limsup_{N\to \infty} \, \max_{q \in \Phi_{\ell}} \,\,
    \max_{Q_2 \in \Phi_{r,K,P_2}, Q_3 \in \Phi_{K,L,P_3}} \,
    \cesE_{(m,n) \in S_{1/q} \cap [N]^2} 
    |  P_2(m,n)^{it} P_3(m,n)^{-it}   
       -1|.
\end{equation}
Using \eqref{arch conc S_d a=c} we see that the quantity in 
\eqref{useful_eq_fix_11_irred} is 
\begin{equation*}
    \leq \oh_{\ell, L\to \infty}(1) + \max_{q\in \Phi_{\ell}} \frac{|t|}{q}= 
    \oh_{\ell, L\to \infty}(1) + \frac{|t|}{\min_{q\in \Phi_{\ell}}q}= 
    \oh_{\ell, L\to \infty}(1) + \oh_{t, \ell \to \infty}(1).
\end{equation*}
We infer that 
for $f\in \M_{p, \textup{fin.supp.}, P_2}\cap 
\M_{p, \textup{fin.supp.}, P_3}$, we have 
\begin{multline*}
    \lim_{\ell, r, K, L, N}\, \max_{q \in \Phi_{\ell}} \,\,
    \max_{Q_2 \in \Phi_{r,K,P_2}, Q_3 \in \Phi_{K,L,P_3}} 
    \cesE_{(m,n) \in S_{1/q} \cap [N]^2}
    \big|f(Q_2) \overline{f(Q_3)} \cdot \\
     f(R_2( r, K, L, N, Q_{1,r}, Q_{2} , Q_{3,K,L} ,m,n))
     \overline{f(R_3( r, K, L, N, Q_{1,r}, Q_{2,r,K} , Q_3 ,m,n))}-1\big|=0,
\end{multline*}
and since the spectral measure of $F_2$ is supported on 
$\M_{p, \textup{fin.supp.}, P_2}\cap 
\M_{p, \textup{fin.supp.}, P_3}$, \eqref{eq_useful_10_irred} follows. 

Now, in view of \eqref{eq_useful_10_irred} we see that to prove 
\eqref{main_term_fix_9_irred}, it suffices to show that 
\begin{multline}\label{main_term_fix_12_irred}
\inflim_{\, \ell,r, K, L,N}
    \int_{X} F \cdot \cesE_{q\in \Phi_{\ell}}\, T_{1/q} F_1\cdot   F_2  \,
    d \mu = \inflim_{\, \ell}
    \int_{X} F \cdot F_2 \cdot \cesE_{q\in \Phi_{\ell}}\, T_{1/q} F_1  \,
    d \mu \geq \Big(\int_X F \, d \mu\Big)^3.
\end{multline}

\smallskip

{\bf Step 5 (Finishing the argument).}  
As in Step 5 in the proof of Proposition \ref{P: main a=c P_2 reducible}, using the mean ergodic theorem, we get that  
\begin{equation*}
   \lim_{\ell \to \infty}
    \int_{X} F \cdot F_2 \cdot \cesE_{q \in \Phi_{\ell}}
    T_{1/q} F_1 \, d \mu = 
    \int_{X} F \cdot  \E(F \mid \mathcal{Y}) \cdot 
    \E(F \mid \mathcal{I}) \,  d \mu
    \geq 
    \Big(\int_X F \, d \mu\Big)^3,
\end{equation*}
where for the last inequality we used Proposition
\ref{chu-lemma}. This concludes the proof. 
\end{proof}

\section{Proof of the main theorem when  $a+b=c$}\label{Sec a+b=c}
In this section we prove Theorem \ref{T:Main} when  $a+b=c$. 
One easily verifies  that 
\begin{equation}\label{eq_param_a+b=c_2}
    x=k (m^2 -2bmn - ab n^2),\:\: y=k (m^2 +2amn - ab n^2), \:\:
    z=k ( m^2 + ab n^2),\:\:\: k, m,n\in \N,
\end{equation}
are solutions of $a x^2 + b y^2=c z^2$ when $a+b=c$. Let
\begin{equation}\label{quadratic_forms_a+b=c}
P_1(m,n):=m^2 + ab n^2, \:\:\:
P_2(m,n):=m^2 -2bmn - ab n^2, \:\:\:
P_3(m,n):= m^2 +2amn - ab n^2.
\end{equation} 
It is easy to see that, since $a, b > 0$,
$P_1$ is irreducible, and
the polynomials $P_1, P_2, P_3$ are pairwise independent.  
This property will  be implicitly used in order to  
solve certain simultaneous congruences. Let 
\begin{equation}\label{gamma_defn}
    \gamma:=\max\{4b, \sqrt{2ab}\}>0.
\end{equation}
For all $(m,n)\in \N^2$, $P_1(m,n)>0$, and if 
$m> \gamma n$, then we have that also $P_2(m,n)$, $P_3(m,n)>0$.

The case $a+b=c$ of Theorem \ref{T:Main} follows from the next result.
\begin{theorem}\label{T: main a+b=c}
Let $(X,\CX,\mu,T_n)$  be a pretentious multiplicative action, let  
$F\in L^{\infty}(\mu)$ with $F\geq 0$ and $F\neq 0$, 
and let $P_1,P_2,P_3$, $\gamma$ be as in 
\eqref{quadratic_forms_a+b=c}, \eqref{gamma_defn}. 
Then for every $\e>0$ we have
\begin{equation}\label{T: main a+b=c eq}
	\int_X T_{P_1(m,n)} F \cdot T_{P_2(m,n)} F \cdot T_{P_3(m,n)} F \, d\mu
	>  \lVert F \rVert _{L^{\infty}(\mu)}^{-1} \Big(\int_X F \, d \mu \Big)^4 -\e
\end{equation}
for a set of $(m,n)\in \N^2$, $m>\gamma n$, with positive upper 
density. { Furthermore, if one of the following holds:
\begin{enumerate}[(i)]
    \item \label{cond_1} one of  $P_2, P_3$ is reducible
    \item \label{cond_2} there are $i, j \in \{1,2,3\}$ with 
    $i\neq j$ such that $\omega_{P_i}(p) = \omega_{P_j}(p)$ for all 
    but finitely many primes $p$
\end{enumerate}
then we get the following stronger conclusion: for every $\e>0$ we have
\begin{equation}\label{T: main a+b=c eq 2}
    \int_X T_{P_1(m,n)} F \cdot T_{P_2(m,n)} F \cdot T_{P_3(m,n)} F \, d\mu
    > \Big(\int_X F \, d \mu \Big)^3 -\e
\end{equation}
for a set of $(m,n)\in \N^2$, $m>\gamma n$, with positive upper 
density. }
\end{theorem}

\begin{rmk}
    An example of an equation where all of the $P_i$'s are irreducible 
    and condition \eqref{cond_2} in Theorem \ref{T: main a+b=c} holds
    is $x^2 + y^2 = 2z^2$, because in that case two of the $P_i$'s  
    have the same discriminant. 
\end{rmk}

We split the proof of Theorem \ref{T: main a+b=c}
into cases according to whether 
$P_2$, $P_3$ are irreducible or not. { When at least one of $P_2, P_3$ is 
reducible, the proof is similar to the proof of the results in Section \ref{Sec a=c or b=c}. Also, when both $P_2$ and  $P_3$ are irreducible, 
but condition \eqref{cond_2} in Theorem \ref{T: main a+b=c} holds, then again the proof is similar. If on the other hand 
both $P_2$ and  $P_3$ are irreducible and \eqref{cond_2} doesn't 
hold, then the previous method does not go through. The different
proof we adopt in this case is responsible for the weaker lower bound
we get in \eqref{T: main a+b=c eq}.}

Throughout Section \ref{Sec a+b=c}, $\Phi_{r,K}$ and $\Phi_{r,K,P}$ are  
as in \eqref{eq_foln_def_1} and \eqref{eq_foln_def_2} (in some instances, $s,r$  
take the place of $r,K$),  
while $Q_s$ and $Q_{L}$ are as in \eqref{Q_l_def}.
Furthermore, 
$\omega_P$ is as in \eqref{omega_def}, and $\D$, $F_N(f,L)$, $G_{P,N}(f,L)$
are as in \eqref{E: definition of pretentious distance}, 
\eqref{def_F_N}, \eqref{def_G_P,N}. 
Finally, $X_{p,\textup{fin.supp.}}$ and $X_{p,\textup{fin.supp.},P}$
are the factors defined respectively in 
\eqref{def_X_p_fin_supp} and \eqref{def_X_p_fin_supp_P}.

{
For each $\delta \in (0, 1/2)$ consider the set $S_{\delta}$ as in \eqref{def S_d a+b=c}, i.e. 
\begin{equation}\label{def S_d a+b=c 1}
    S_{\delta}:= \Big\{ (m,n) \in \N ^2 \colon \frac{n}{m} < \frac{1- e^{-\delta}}{4ab} \Big\}.
\end{equation}
As we saw in Lemma \ref{L: S_d a+b=c}, $S_{\delta}$ has positive density, and for $(m,n) \in S_{\delta}$, $P_j(m,n)>0$ for 
$j=1,2,3$, and  
\begin{equation}\label{prop S_d a+b=c 1}
 \big| (P_l(m,n)/P_j(m,n)) ^{it} - 1 \big| \leq |t| \delta \quad \text{for}\,\,\,  l, j\in \{1,2,3\}.
\end{equation}
We will always assume that $\delta$ is small enough depending on $a,b$,
so that $2 \gamma < \frac{4ab}{1- e^{-\delta}}$. Then for 
$(m,n) \in S_{\delta}$, we have $m>2 \gamma n$, which will give us that 
$Qm+1 > \gamma (Q n + v)$, and therefore $P_j(Qm+1,Q n+v)>0$ for 
$j=1,2,3$. 
}

\subsection{$P_2,P_3$ irreducible}\label{Sec P2 P3 irreducible}
Here we handle the case when $P_2,P_3$ are both  
irreducible. A representative example is the equation $x^2+y^2=2z^2$, in which  
case \eqref{quadratic_forms_a+b=c} gives
$$
P_1(m,n)=m^2+n^2,   \quad P_2(m,n) = m^2-2mn-n^2, \quad  
P_3(m,n)=m^2+2mn-n^2.
$$

To prove the result, given $Q_1 \in \Phi_{s,r,P_1}$, 
$Q_2 \in \Phi_{r,K,P_2}$ and 
$Q_3 \in \Phi_{K,L, P_3}$,
we will choose an integer $v$ depending on $s,r,K,L,Q_1,Q_2, Q_3$, and 
we will average over the grid 
$(Q_{L} m+1, Q_{L} n +v)$ in \eqref{T: main a+b=c eq}.

From now on, we assume that $s$ is large enough so that for $p>s$ we have 
$\omega_{P_j}(p)\in \{0,2\}$ for $j\in\{1,2,3\}$ and 
\begin{equation}\label{a+b=c_new_E1: p not dividing}
p\nmid 2ab(a+b)=2abc.    
\end{equation}

\begin{lemma}\label{a+b=c_new_lemma_choice_of_v_1}
Let $s,r,K,L\in \N$ with $s<r<K<L/2$,  $P_1,P_2,P_3$ be as in \eqref{quadratic_forms_a+b=c},  and let 
$Q_1 \in \Phi_{s,r,P_1}$, 
$Q_2 \in \Phi_{r,K,P_2}$ and $Q_3 \in \Phi_{K,L,P_3}$. 
Then there is an integer
$v$ with $0\leq v \leq Q_{ L}-1$ such that 
\begin{enumerate}[(i)]
    \item \label{a+b=c_new_E1: lemma_choice_of_v_1} 
    $(Q_{ L},P_1(1,v))=Q_1$,
    $\frac{P_1(1,v)}{Q_1} \equiv Q_1^{-1} \pmod{{Q_s}}$,
    \item \label{a+b=c_new_E2: lemma_choice_of_v_1} 
    $(Q_{ L},P_2(1,v))=Q_2$,
    $\frac{P_2(1,v)}{Q_2} \equiv Q_2^{-1} \pmod{{Q_s}}$,
    \item \label{a+b=c_new_E3: lemma_choice_of_v_1} 
    $(Q_{ L},P_3(1,v))=Q_3$,
    $\frac{P_3(1,v)}{Q_3} \equiv Q_3^{-1} \pmod{{Q_s}}$.
\end{enumerate}
\end{lemma}

\begin{proof}
Let $p\in \P$ with $s<p \leq r$ and $\omega_{P_1}(p)=2$, and consider the 
polynomial $g_1(x):=P_1(1,x)-p^{\theta_p(Q_1)}$. 
Then $\textup{mod}\: p$ the polynomial
becomes $g_1(x)=P_1(1,x)$, and it has two simple roots
$z_1, z_1'$ with $z_1 \not \equiv z_1' \pmod p$. Using  
Hensel's lemma, we may lift $z_1$ to a unique root 
$\zeta_1 = \zeta_1 (p, Q_1)$ of $g_1 \mod {p^{\theta_p(Q_1)+1}}$.

Analogously, for $p\in \P$ with $r<p \leq K$ and $\omega_{P_2}(p)=2$, 
we may find $\zeta_2 = \zeta_2 (p, Q_2)$ so that if 
$g_2(x):=P_2(1,x)-p^{\theta_p(Q_2)}$, then $g_2(\zeta_2)\equiv 0 
\pmod {p^{\theta_p(Q_2)+1}}$, and for 
$p\in \P$ with $K<p \leq L$ and $\omega_{P_3}(p)=2$, 
we may find $\zeta_3 = \zeta_3 (p, Q_3)$ so that if 
$g_3(x)=P_3(1,x)-p^{\theta_p(Q_3)}$, then $g_3(\zeta_3)\equiv 0 
\pmod {p^{\theta_p(Q_3)+1}}$.

Using the Chinese Remainder theorem we may then choose $v$ with 
$0\leq v \leq Q_{L}-1$
such that 
\begin{itemize}
     \item $v\equiv 0 \pmod {p^{2s}}$ for   $p\leq s$,
    \item $v\equiv 0$ $\pmod {p}$ for $s< p\leq r$ with 
    $\omega_{P_1}(p)=0$,
    \item $v\equiv \zeta_1(p,Q_1) \pmod {p^{\theta_p(Q_1)+1}}$
    for $s< p\leq r$ with    $\omega_{P_1}(p)=2$,
    \item $v\equiv 0$ $\pmod {p}$ for $r< p\leq K$ with 
    $\omega_{P_2}(p)=0$,
    \item $v\equiv \zeta_2(p,Q_2) \pmod {p^{\theta_p(Q_2)+1}}$ for 
    $r<p\leq K$ with $\omega_{P_2}(p)=2$,
    \item $v\equiv 0 \pmod {p}$ for 
    $K<p\leq L$ with $\omega_{P_3}(p)=0$,
    \item $v\equiv \zeta_3(p,Q_3) \pmod {p^{\theta_p(Q_3)+1}}$ for 
    $K<p\leq L$ with $\omega_{P_3}(p)=2$.
\end{itemize}  
Recall that $P_1(1,v)=1+abv^2$. For $p\leq s$, $P_1(1,v)\equiv 1 \pmod p$, 
so $(p, P_1(1,v))=1$. For $s<p\leq r$ with $\omega_{P_1}(p)=0$, 
$P_1(1,v)\equiv 1 \pmod p$, so $(p, P_1(1,v))=1$.
For $s<p\leq r$ with $\omega_{P_1}(p)=2$, 
$P_1(1,v)\equiv P_1(1,\zeta_1)\equiv p^{\theta_p(Q_1)}  
\pmod {p^{\theta_p(Q_1)+1}}$, so $p^{\theta_p(Q_1)}\parallel P_1(1,v)$.

For $r<p\leq K$ with $\omega_{P_2}(p)=0$, 
$P_1(1,v)\equiv 1 \pmod p$, so $(p, P_1(1,v))=1$.
For $r<p\leq K$ with $\omega_{P_2}(p)=2$, suppose that $P_1(1,v)\equiv 0 
\pmod p$. 
We know that $P_2(1,v)\equiv 0 \pmod p$, so adding we get that
$p\mid P_1(1,v)+P_2(1,v)=2(1-bv)$. Since $p>r$ we have $p\neq 2$, and
we infer 
that $p\mid (1-bv)$, which implies that $v\equiv b^{-1} \pmod p$. This 
implies that $P_2(1,v)\equiv -1 -ab^{-1} \pmod p$, and since 
$P_2(1,v) \equiv 0 \pmod p$, we obtain that $a+b \equiv 0 \pmod p$, which 
contradicts \eqref{a+b=c_new_E1: p not dividing}. Therefore, 
again we have $(p, P_1(1,v))=1$.

For $K<p\leq L$ with $\omega_{P_3}(p)=0$, 
$P_1(1,v)\equiv 1 \pmod p$, so $(p, P_1(1,v))=1$.
For $K<p\leq L$ with $\omega_{P_3}(p)=2$, suppose that 
$P_1(1,v)\equiv 0 \pmod p$. 
We know that $P_3(1,v)\equiv 0 \pmod p$, so adding we get that
$p\mid P_1(1,v)+P_3(1,v)=2(1+av)$. Since $p>r$ we have $p\neq 2$, and
we infer 
that $p\mid (1+av)$, which implies that $v\equiv -a^{-1} \pmod p$. This 
implies that $P_3(1,v)\equiv -1 -a^{-1}b \pmod p$, and since 
$P_3(1,v) \equiv 0 \pmod p$, we obtain that $a+b \equiv 0 \pmod p$, again
contradicting \eqref{a+b=c_new_E1: p not dividing}. Therefore, 
we have $(p, P_1(1,v))=1$. Combining all of the above we get that 
$(Q_{ L}, P_1(1,v))=Q_1$. For $p\leq s$, 
$P_1(1,v) \equiv 1 \pmod {p^{2s}}$, so $P_1(1,v) \equiv 1 \pmod {Q_s}$,
and since $(Q_s, Q_1)=1$, we have that 
$\frac{P_1(1,v)}{Q_1} \equiv Q_1^{-1} \pmod{{Q_s}}$. This proves 
\eqref{a+b=c_new_E1: lemma_choice_of_v_1}. The proofs of 
\eqref{a+b=c_new_E2: lemma_choice_of_v_1}, 
\eqref{a+b=c_new_E3: lemma_choice_of_v_1} are identical, so
they are omitted.
\end{proof}

For the rest of Section \ref{Sec P2 P3 irreducible}, we suppress 
the dependence of $v$ on the parameters 
$s,r,K,L,Q_1,Q_2, Q_3$, and we simply write $v$
instead of $v(s,r,K,L,Q_1,Q_2, Q_3)$. 

Given $s,r, K, L, N\in \N$ with 
$s<r<K<L<N $, $Q_1 \in \Phi_{s,r,P_1}$, $Q_2\in \Phi_{r,K,P_2}$, 
$Q_3\in \Phi_{K,L,P_3}$ and $m,n\in \N$, for $j\in \{1,2,3\}$, let 
\begin{equation}\label{a+b=c_new_R_i def}
    R_j (s, r, K, L, N, Q_1, Q_2, Q_3,m,n):= 
\frac{P_j(Q_{L} m +1, Q_{L} n +v )}{Q_j},
\end{equation}
where $v$ is as in Lemma \ref{a+b=c_new_lemma_choice_of_v_1}.
Then for $j\in \{1,2,3\}$ we have
\begin{equation}\label{a+b=c_new_eq_extr_1} 
    P_j(Q_{L} m+1, Q_{L} n +v)= Q_j 
    R_j (s, r, K, L, N, Q_1, Q_2, Q_3,m,n).  
\end{equation}
Sometimes in the notation we suppress the dependence of $R_j$ on 
the parameters 
$s, r$, $K, L, N, Q_1, Q_2, Q_3$, and we only write $R_j(m,n)$.

When $P_2,P_3$ are irreducible, Theorem \ref{T: main a+b=c} follows from 
the next result. 
{
\begin{proposition}\label{a+b=c_new_P: main a=c P_2 reducible}
Let $(X,\CX,\mu,T_n)$  be a pretentious multiplicative action, let 
$F\in L^{\infty}(\mu)$ with $F\geq 0$ and $F\neq 0$, 
let $P_1,P_2,P_3$ be as in 
\eqref{quadratic_forms_a+b=c}, and suppose that $P_2,P_3$ are irreducible. 
Also, let $S_{\delta}$ be as in \eqref{def S_d a+b=c 1} and $v=v(s, r,K,L,Q_1,Q_2,Q_3)$ be as in 
Lemma \ref{a+b=c_new_lemma_choice_of_v_1}.
Then 
    \begin{multline}\label{a+b=c_new_main_term_1}
    \inflim_{\, \delta, s,r,K,L,N}\:
    \cesE_{Q_{1}\in \Phi_{s,r,P_1}} \cesE_{Q_{2}\in \Phi_{r,K,P_2}}
    \cesE_{Q_{3}\in \Phi_{K,L,P_3}}
    \cesE_{(m,n) \in S_{\delta} \cap [N]^2} \\
     \int_{X}  \prod_{j=1} ^3
    T_{P_j(Q_{L} m+1, Q_{L} n +v)} F 
    \, d \mu \geq    \lVert F \rVert _{L^{\infty}(\mu)}^{-1} \Big(\int_X F \, d \mu \Big)^4.
    \end{multline}
Furthermore, if there are $i, j \in \{1,2,3\}$ with 
$i\neq j$ such that $\omega_{P_i}(p) = \omega_{P_j}(p)$ for all 
but finitely many primes $p$,
then we get the following, which is stronger: 
\begin{multline}\label{a+b=c_new_main_term_20}
    \inflim_{\, \delta, s,r,K,L,N}\:
    \cesE_{Q_{1}\in \Phi_{s,r,P_1}} \cesE_{Q_{2}\in \Phi_{r,K,P_2}}
    \cesE_{Q_{3}\in \Phi_{K,L,P_3}}
    \cesE_{(m,n) \in S_{\delta} \cap [N]^2} \\
     \int_{X}  \prod_{j=1} ^3
    T_{P_j(Q_{L} m+1, Q_{L} n +v)} F 
    \, d \mu \geq   \Big(\int_X F \, d \mu \Big)^3.
\end{multline}
\end{proposition}}

It remains to prove Proposition \ref{a+b=c_new_P: main a=c P_2 reducible}. 
For that we need the following lemma.

\begin{lemma}\label{a+b=c_new_L: main a=c P_2 reducible}
Let $(X, \X, \mu, T_n)$, $F$  be as in Proposition~\ref{a+b=c_new_P: main a=c P_2 reducible}, and 
let $R_j$, $j\in \{1,2,3\}$,  be as in \eqref{a+b=c_new_R_i def}.
Then  for every $\delta$ we have 
\begin{align}
& \lim_{s,r,K, L, N} \, 
\max_{Q_1 \in \Phi_{s,r,P_1}, \: Q_2, Q_{2}' \in \Phi_{r,K,P_2}, \:Q_3, Q_{3}' \in \Phi_{K,L,P_3} } 
 \cesE_{(m,n) \in S_{\delta} \cap  [N]^2} 
  \label{a+b=c_new_L: main a=c P_2 reducible eq 1}\\ 
& \hspace{10mm} \lVert 
 T_{R_1(s, r, K, L, N, Q_1, Q_2, Q_3,m,n)} F - 
 T_{R_1(s, r, K, L, N, Q_1, Q_2 ', Q_3 ',m,n)} F
\rVert _{L^2(\mu)}=0, \nonumber \\
& \lim_{s,r,K, L, N} \, 
\max_{Q_1, Q_1 ' \in \Phi_{s,r,P_1}, \: Q_2 \in \Phi_{r,K,P_2}, \:Q_3, Q_{3}' \in \Phi_{K,L,P_3} }
 \cesE_{(m,n) \in S_{\delta} \cap  [N]^2} 
   \label{a+b=c_new_L: main a=c P_2 reducible eq 2} \\
& \hspace{10mm}  \lVert 
T_{R_2(s, r, K, L, N, Q_1, Q_2, Q_3,m,n)} F - 
T_{R_2(s, r, K, L, N, Q_1 ', Q_2, Q_3 ',m,n)} F
\rVert _{L^2(\mu)}=0, \nonumber  \\
& \lim_{s,r,K, L, N}\, 
\max_{Q_1, Q_1 ' \in \Phi_{s,r,P_1}, \: Q_2, Q_2 ' \in \Phi_{r,K,P_2},
\:Q_3 
\in \Phi_{K,L, P_3} } 
 \cesE_{(m,n) \in S_{\delta} \cap  [N]^2} 
   \label{a+b=c_new_L: main a=c P_2 reducible eq 3} \\
& \hspace{10mm} \lVert 
T_{R_3(s, r, K, L, N, Q_1, Q_2, Q_3,m,n)} F - 
T_{R_3(s, r, K, L, N, Q_1 ', Q_2 ', Q_3,m,n)} F
\rVert _{L^2(\mu)}=0. \nonumber 
\end{align}
\end{lemma}

\begin{remark}
Similarly to Lemma \ref{L: main a=c P_2 reducible}, 
Lemma \ref{a+b=c_new_L: main a=c P_2 reducible}
essentially says that 
in the integral in \eqref{a+b=c_new_main_term_1},
the first term is independent of $Q_2, Q_3$, the second term 
is independent of $Q_1, Q_3$ and the third term 
is independent of $Q_1, Q_2$.
\end{remark}

\begin{proof}
We omit the proof, as it closely follows the proof of Lemma \ref{L: main 
	a=c P_2 reducible}.  The essential step is to apply the 
	concentration estimates in Propositions \ref{linear_conc} and 
    \ref{quad_conc}, with 
	the only change being the use of Lemma \ref{a+b=c_new_lemma_choice_of_v_1} 
	in place of Lemma \ref{lemma_choice_of_v_1}.	
\end{proof}

We are now ready to prove Proposition \ref{a+b=c_new_P: main a=c P_2 reducible}. { Recall from Section~\ref{SS:factors} the definition of the factors
$\X_{p,\textup{fin.supp.},P_j}$, $j=1,2,3$,
whose corresponding $L^2(\mu)$ subspaces are defined in
\eqref{def_X_p_fin_supp_P}.}
{
\begin{proof}[Proof of Proposition \ref{a+b=c_new_P: main a=c P_2 reducible} (general case-lower bound \eqref{a+b=c_new_main_term_1})]
By homogeneity, we may assume that
$\lVert F\rVert_{L^{\infty}(\mu)}= 1$. In the proof, it is always 
implicit that $s<r<K<L<N$, and that each of the previous variables is 
sufficiently large depending on the smaller ones.

In view of \eqref{a+b=c_new_eq_extr_1}, the quantity in 
$\inflim_{\, \delta, s,r,K,L,N}$ on the left-hand side in 
\eqref{a+b=c_new_main_term_1} equals 
\begin{multline*}
\cesE_{Q_{1}\in \Phi_{s,r,P_1}} \cesE_{Q_{2}\in \Phi_{r,K,P_2}}
    \cesE_{Q_{3}\in \Phi_{K,L,P_3}}  
    \cesE_{(m,n) \in S_{\delta}\cap  [N]^2} \\
    \int_{X}  T_{Q_1 R_1(m,n)} F \cdot 
    T_{Q_2 R_2(m,n)} F \cdot 
    T_{{Q}_3 R_3(m,n)} F
    \, d \mu.  
\end{multline*}
For each $s,r,K,L\in \N$, choose any $Q_{1,s,r}\in \Phi_{s,r,P_1}$,  
$Q_{2, r, K} \in \Phi_{r,K,P_2}$ and $Q_{3,K,L}\in \Phi_{K,L,P_3}$. 
Using Lemma \ref{a+b=c_new_L: main a=c P_2 reducible}, 
one sees that in order to prove
\eqref{a+b=c_new_main_term_1}, it suffices to prove that 
\begin{align}\label{a+b=c_new_main_term_3}
\inflim_{\,\delta ,s,r, K, L,N}&
    \int_{X}  \cesE_{(m,n) \in S_{\delta} \cap [N]^2} 
     \Big( \cesE_{Q_{1}\in \Phi_{s,r,P_1}}  T_{Q_1 R_1(s, r, K, L, N, Q_1, Q_{2,r,K} , Q_{3,K,L} ,m,n)} F \Big)
    \cdot \nonumber \\ &\Big( \cesE_{Q_{2}\in \Phi_{r,K,P_2}}  
     T_{Q_2 R_2(s, r, K, L, N, Q_{1,s,r}, 
    Q_2 , Q_{3,K,L},m,n)} F \Big) \cdot  
    \\ &\Big( \cesE_{Q_{3}\in \Phi_{K,L,P_3}}
    T_{{Q}_3 R_3(s, r, K, L, N, Q_{1,s,r}, Q_{2,r,K} , Q_3 ,m,n)} F \Big) \,
    d \mu \geq \Big(\int_X F \, d \mu\Big)^4 . \nonumber 
\end{align}

The first step is similar to Step 1 in the proof  of Proposition
\ref{P: main a=c P_2 reducible}.

\smallskip 

{\bf Step 1 (Reduction to 
$\X_{p, \textnormal{\textup{fin.supp.}},P_1}$ in 
the average over $Q_1$).}
Write $F=F_1 + F_4$, where $F_1:=\E(F | {\X}_{p, \textup{fin.supp.},P_1})$ 
and $F_4 \perp F_1$. 
Then the main term in \eqref{a+b=c_new_main_term_3} is equal to 
\begin{align}\label{a+b=c_new_main_term_4}
    \sum_{j\in \{1,4\}}\int_{X}  \cesE_{(m,n) \in S_{\delta} \cap [N]^2} 
     &\Big( \cesE_{Q_{1}\in \Phi_{s,r,P_1}} T_{Q_1 R_1(s, r, K, L, N, Q_1, Q_{2,r,K} , Q_{3,K,L} ,m,n)} F_j \Big)
    \cdot \\ &\Big( \cesE_{Q_{2}\in \Phi_{r,K,P_2}}  
     T_{Q_2 R_2(s, r, K, L, N, Q_{1,s,r}, 
    Q_2 , Q_{3,K,L},m,n)} F \Big) \cdot \nonumber \\ & \Big( \cesE_{Q_{3}\in \Phi_{K,L,P_3}}
    T_{{Q}_3 R_3(s, r, K, L, N, Q_{1,s,r}, Q_{2,r,K} , Q_3 ,m,n)} F \Big) \,
    d \mu . \nonumber
\end{align}

The term corresponding to $j=4$ in \eqref{a+b=c_new_main_term_4} 
is bounded by 
\begin{equation}\label{a+b=c_new_aux_quant_1}
  \cesE_{(m,n) \in S_{\delta} \cap [N]^2}  \big \lVert 
\cesE_{Q_{1}\in \Phi_{s,r,P_1}} 
T_{Q_1 R_1(s, r, K, L, N, Q_1, Q_{2,r,K} , Q_{3,K,L} ,m,n)} F_4
\big \rVert_{L^2(\mu)}.
\end{equation}
We will show that the quantity in \eqref{a+b=c_new_aux_quant_1} goes to 
$0$ as $N\to \infty, L \to \infty, K \to \infty, r \to \infty, 
s \to \infty$. 
Since $\sigma_{F_4}$ is supported on
$\M_p \setminus \M_{p, \textup{fin.supp.},P_1}$, using \eqref{spec_eq_1} and 
applying Fatou's lemma multiple times, one sees 
that it suffices to prove that for every 
$f\in \M_p \setminus \M_{p, \textup{fin.supp.},P_1}$,  we have 
\begin{equation*}
\lim _{s,r,K,L,N} \:
\cesE_{(m,n) \in S_{\delta} \cap [N]^2} \big | 
\cesE_{Q_{1}\in \Phi_{s,r,P_1}} f({Q_1 
R_1(s, r, K, L, N, Q_1, Q_{2,r,K} , Q_{3,K,L} ,m,n)}) \big |=0.
\end{equation*}
Let $f\in \M_p \setminus \M_{p, \textup{fin.supp.},P_1}$, 
let $\chi, t$ so that $f\sim \chi \cdot n^{it}$,
and suppose that $s,L$ are large enough so that the conductor $q_{\chi}$ of 
$\chi$ divides $Q_s$ and $\frac{Q_{ L}}{Q_1}$ for all 
$Q_1 \in \Phi_{s,r,P_1}$. 
Using Lemma \ref{a+b=c_new_lemma_choice_of_v_1} \eqref{a+b=c_new_E1: lemma_choice_of_v_1} 
and applying Proposition \ref{quad_conc} and Remark \ref{uniform_in_Q} we deduce that
\begin{align}\label{a+b=c_new_eq_useful_1}
&\limsup_{N\to \infty} \cesE_{(m,n) \in S_{\delta} \cap  [N]^2}  \big | 
\cesE_{Q_{1}\in \Phi_{s,r,P_1}} f({Q_1 
R_1(s, r, K, L, N, Q_1, Q_{2,r,K} , Q_{3,K,L} ,m,n)})
\big | \ll_{P_1, \delta} \\
&\limsup_{N\to \infty} \cesE_{(m,n) \in S_{\delta} \cap [N]^2}  \big | 
\cesE_{Q_{1}\in \Phi_{s,r,P_1}}  f(Q_1)\, \overline{\chi(Q_1)} \, 
Q_1 ^{-it}  \, Q_{L}^{2it} \, (P_{1}(m,n))^{it} \, 
\exp(G_{P_1,N}(f,L)) \big |   \nonumber \\
& \hspace{20mm} + \D(f, \chi \cdot n^{it}; L, \infty)
+L^{-1/2}
= \nonumber \\
& \limsup_{N\to \infty} |\exp(G_{P_1,N}(f,L))|\cdot 
\big| \cesE_{Q_1 \in \Phi_{s,r,P_1}} 
f(Q_1) \, \overline{\chi(Q_1)}  \, Q_1 ^{-it}
\big|+ \D(f, \chi \cdot n^{it}; L, \infty)
+L^{-1/2}. \nonumber
\end{align}
From $f\sim \chi \cdot n^{it}$ we deduce that   
$\lim_{L\to \infty} \lim_{N\to \infty} |\exp(G_{P_1,N}(f,L))|=1$.
Since $f \in \M_p \setminus \M_{p, \textup{fin.supp.},P_1}$, for each 
$s$ there is 
$p>s$ such that $\omega_{P_1}(p)>0$ and $f(p) \neq \chi(p) p^{it}$. Using 
that $\Phi_{s,r,P_1}$ is asymptotically
invariant under dilation by $p$ as $r\to \infty$ we have 
\begin{multline*}
    {f(p)}\, \overline{\chi(p)}\, p^{-it} 
    \cesE_{Q_{1} \in \Phi_{s,r,P_1}} f(Q_{1})\, 
    \overline{\chi(Q_{1})}\, Q_{1} ^{-it}
    =
    \cesE_{Q_{1} \in p \Phi_{s,r,P_1}} f(Q_{1})\,  \overline{\chi(Q_{1})}\, 
    Q_{1} ^{-it}
    \\ =
    \cesE_{Q_{1} \in  \Phi_{s,r,P_1}} f(Q_{1}) \, \overline{\chi(Q_{1})}\, 
    Q_{1} ^{-it}
   + \oh_{r\to \infty}(1),
\end{multline*}
and since $f(p) \neq \chi(p)\, p^{it}$, the above implies that 
\begin{equation*}
   \lim_{r\to \infty} \cesE_{Q_{1} \in  \Phi_{s,r,P_1}} f(Q_{1}) 
   \cdot \overline{\chi(Q_{1})} \cdot Q_{1} ^{-it} = 0.
\end{equation*}
Combining this with \eqref{a+b=c_new_eq_useful_1} we infer that 
\begin{equation}\label{a+b=c_new_aux_eq_3}
    \lim_{s,r,K,L,N}
    \cesE_{(m,n) \in S_{\delta} \cap [N]^2}  \big \lVert 
\cesE_{Q_{1}\in \Phi_{s,r,P_1}} 
T_{Q_1 R_1(s, r, K, L, N, Q_1, Q_{2,r,K} , Q_{3,K,L} ,m,n)} F_4
\big \rVert_{L^2(\mu)}=0.
\end{equation}

\smallskip

{\bf Steps 2 and 3 (Dealing with the averages over $Q_2$ and $Q_3$).}
Writing $F=F_2 + F_5$, where $F_2:=\E(F | \X_{p, \textup{fin.supp.},P_2})$ 
and $F_5 \perp F_2$, and utilizing a similar argument we get 
\begin{equation}\label{a+b=c_new_aux_eq_4}
    \lim_{s,r,K,L,N}
    \cesE_{(m,n) \in S_{\delta} \cap [N]^2}  \big \lVert 
\cesE_{Q_{2}\in \Phi_{r,K,P_2}} 
T_{Q_2 R_2(s, r, K, L, N, Q_{1,s,r}, Q_{2} , Q_{3,K,L} ,m,n)} F_5
\big \rVert_{L^2(\mu)}=0,
\end{equation}
and similarly, if we write $F=F_3 + F_6$, where 
$F_3:=\E(F | \X_{p, \textup{fin.supp.}, P_3})$
and $F_6 \perp F_3$, then 
\begin{equation}\label{a+b=c_new_aux_eq_5}
    \lim_{s,r,K,L,N}
    \cesE_{(m,n) \in S_{\delta} \cap [N]^2}  \big \lVert 
\cesE_{Q_{3}\in \Phi_{K,L,P_3}} 
T_{Q_3 R_3(s, r, K, L, N, Q_{1,s,r}, Q_{2,r,K} , Q_{3} ,m,n)} F_6
\big \rVert_{L^2(\mu)}=0. 
\end{equation}

From \eqref{a+b=c_new_aux_eq_3}, \eqref{a+b=c_new_aux_eq_4} and 
\eqref{a+b=c_new_aux_eq_5}, we see that in order to prove 
\eqref{a+b=c_new_main_term_3}, it suffices to show that 
\begin{align}\label{a+b=c_new_main_term_3_fix}
\inflim_{\,\delta ,s,r, K, L,N}&
    \int_{X}  \cesE_{(m,n) \in S_{\delta} \cap [N]^2} 
     \Big( \cesE_{Q_{1}\in \Phi_{s,r,P_1}}  T_{Q_1 R_1(s, r, K, L, N, Q_1, Q_{2,r,K} , Q_{3,K,L} ,m,n)} F_1 \Big)
    \cdot \nonumber \\ &\Big( \cesE_{Q_{2}\in \Phi_{r,K,P_2}}  
     T_{Q_2 R_2(s, r, K, L, N, Q_{1,s,r}, 
    Q_2 , Q_{3,K,L},m,n)} F_2 \Big) \cdot  
    \\ &\Big( \cesE_{Q_{3}\in \Phi_{K,L,P_3}}
    T_{{Q}_3 R_3(s, r, K, L, N, Q_{1,s,r}, Q_{2,r,K} , Q_3 ,m,n)} F_3 \Big) \,
    d \mu \geq \Big(\int_X F \, d \mu\Big)^4. \nonumber 
\end{align}
\begin{comment}
and using the $T$-invariance of $\mu$ we may rewrite the previous 
as 
\begin{align}\label{a+b=c_new_main_term_4_fix}
&\inflim_{\,\delta ,s,r, K, L,N} \cesE_{Q_{1}\in \Phi_{s,r,P_1}}
\cesE_{Q_{2}\in \Phi_{r,K,P_2}} \cesE_{Q_{3}\in \Phi_{K,L,P_3}}
\cesE_{(m,n) \in S_{\delta} \cap [N]^2} \\
&\int_{X}  
     \Big(   T_{Q_1 R_1(s, r, K, L, N, Q_1, Q_{2,r,K} , Q_{3,K,L} ,m,n)} F_1 \Big)
    \cdot \nonumber \Big(   
     T_{Q_2 R_2(s, r, K, L, N, Q_{1,s,r}, 
    Q_2 , Q_{3,K,L},m,n)} F_2 \Big) \cdot  
    \\ & \hspace{6mm}\Big( 
    T_{{Q}_3 R_3(s, r, K, L, N, Q_{1,s,r}, Q_{2,r,K} , Q_3 ,m,n)} F_3 \Big) \,
    d \mu \geq \Big(\int_X F \, d \mu\Big)^4, \nonumber 
\end{align}
\end{comment}

\smallskip

{\bf Step 4 (A concentration estimate).}
The main idea that enables us to prove 
 \eqref{a+b=c_new_main_term_3_fix} is to 
 introduce an auxiliary average over $y\in\N$ and show that all three 
terms in \eqref{a+b=c_new_main_term_3_fix} can be approximated,
in averaged $L^2(\mu)$ norm, by
$T_{\frac{m^2Q_L^3y+1}{Q_Ly+1}}F_j$.\footnote{In the preceding cases, dividing by one polynomial iterate
	worked because the relevant spectral functions satisfied the
	finite-support condition for both forms in each quotient. Here this
	condition is only known for the form $P_j$ corresponding to $F_j$,
	so dividing by the third iterate may leave uncancelled factors
	involving $f(Q_3)$ and $G_{P_3,N}(f,L)$. The auxiliary average over $y$
	provides a common approximating transformation through linear
	concentration. We use $m^2$ to simplify the prime-divisor error estimate,
	but any fixed form $H(m,n)=m^2+\beta mn+\gamma n^2$, with $\beta,\gamma \in\Z$, would work; what matters is that 
	for sufficiently small $\delta$, it is positive on $S_\delta$ and
	$|\log(H(m,n)/P_j(m,n))|\ll_{a,b,H}\delta$ for $j=1,2,3$.}
	We can then remove
	this common transformation after multiplying and integrating,
	by invariance of $\mu$, and obtain $\int_X F_1\cdot F_2\cdot F_3\,d\mu$,
	for which we have the required lower bound. 

For the term involving $F_1$, we show that
\begin{multline}\label{a+b=c_new_main_term_3_fix_ld}
	\lim_{\delta,s,r,K,L,N}\limsup_{Y\to\infty}
	\max_{Q_1\in\Phi_{s,r,P_1}}
	\cesE_{(m,n)\in S_\delta\cap[N]^2}\,\cesE_{y\leq Y}
	\\
	\big\lVert
	T_{Q_1R_1(s,r,K,L,N,Q_1,Q_{2,r,K},Q_{3,K,L},m,n)}F_1
	-
	T_{\frac{m^2Q_L^3y+1}{Q_Ly+1}}F_1
	\big\rVert_{L^2(\mu)}=0.
\end{multline}
Here the limit  in $Y$ is taken first, with all other
parameters fixed.
Passing the  $\limsup_{Y\to\infty}$ through the finite maximum and inside
the average over $(m,n)$ gives an upper bound for the left-hand side
of \eqref{a+b=c_new_main_term_3_fix_ld}. It therefore suffices to show that
\begin{multline}\label{a+b=c_new_main_term_4_fix_ld}
    \lim_{\delta ,s,r, K, L,N} \max_{Q_{1} \in \Phi_{s,r,P_1}} 
    \cesE_{(m,n) \in S_{\delta} \cap [N]^2}\,
    \limsup_{Y \to \infty} \cesE_{y \leq Y}
    \\ \big \lVert 
    T_{Q_1 R_1(s, r, K, L, N, Q_1, Q_{2,r,K} , Q_{3,K,L} ,m,n)} F_1 -
    T_{\frac{m^2Q_L ^3 y +1}{Q_L y +1}} F_1 
    \big \rVert_{L^2(\mu)}=0. 
\end{multline}
Since the spectral measure of $F_1$ is supported on 
$\M_{p, \textup{fin.supp.},P_1}$, to prove 
\eqref{a+b=c_new_main_term_4_fix_ld}, it suffices to show that for 
all $f \in \M_{p, \textup{fin.supp.},P_1}$, 
\begin{multline}\label{a+b=c_new_main_term_5_fix_ld}
    \lim_{\delta ,s,r, K, L,N} \max_{Q_{1} \in \Phi_{s,r,P_1}} 
    \cesE_{(m,n) \in S_{\delta} \cap [N]^2}\,
    \limsup_{Y \to \infty} \cesE_{y \leq Y}
    \\ \big | f(Q_1 R_1(s, r, K, L, N, Q_1, Q_{2,r,K} , Q_{3,K,L} ,m,n)) - f(m^2Q_L ^3 y +1)\overline{f(Q_L y +1)}
    \big |=0. 
\end{multline}
Let $f\in \M_{p, \textup{fin.supp.},P_1}$, and
let $\chi, t$ and $s_0 \in \N$ so that $f\sim \chi \cdot n^{it}$ and
$f(p)=\chi(p)\, p^{it}$ for all primes $p>s_0$ with $\omega_{P_1}(p)>0$. 
Also, suppose that $s,L$ are sufficiently
large so that $s>s_0$ and the conductor $q_{\chi}$ of $\chi$ 
divides $Q_s$ and $\frac{Q_{ L}}{Q_1}$  for every
$Q_{1} \in \Phi_{s,r,P_1}$. 
Let 
$$
e_L(m):=\sum_{p>L, \, p\mid m}\frac{1}{p}.
$$
From Proposition~\ref{linear_conc}
we have that 
\begin{align}\label{a+b=c_new_main_term_6_fix_ld}
    &\limsup_{Y\to \infty} 
    \cesE_{y \leq Y} | f(Q_1 R_1(s, r, K, L, N, Q_1, Q_{2,r,K} , Q_{3,K,L} ,m,n)) - \\
    & \hspace{55mm} \nonumber
    f(m^2Q_L ^3 y +1)\overline{f(Q_L y +1)}
    \big |  \ll \\ 
   &  \nonumber \limsup_{Y\to \infty} 
    \cesE_{y \leq Y} | f(Q_1 R_1(s, r, K, L, N, Q_1, Q_{2,r,K} , Q_{3,K,L} ,m,n)) - \\
    & \nonumber m^{2it} Q_L ^{3 it} 
    y^{it} \exp(F_{Y}(f,L)) Q_L ^{-it} 
    y^{-it} \overline{\exp(F_{Y}(f,L)) }
    \big | + \D(f, \chi \cdot n^{it}, L, \infty) +  L^{-1/2} +e_L(m) = \\
    &  \nonumber \D(f, \chi \cdot n^{it}, L, \infty) +  L^{-1/2} +
    \limsup_{Y\to \infty} 
     | f(Q_1 R_1(s, r, K, L, N, Q_1, Q_{2,r,K} , Q_{3,K,L} ,m,n)) - \\
    & \nonumber  Q_L ^{2 it} m^{2it}  \exp(2\Re (F_{Y}(f,L)) )
    \big | + e_L(m)= \\
    &  \nonumber \D(f, \chi \cdot n^{it}, L, \infty) +  L^{-1/2} + 
     | f(Q_1 R_1(s, r, K, L, N, Q_1, Q_{2,r,K} , Q_{3,K,L} ,m,n)) - \\
    & \nonumber  Q_L ^{2 it} m^{2it}  \exp(-2\D^2(f, \chi \cdot n^{it};L, \infty))    \big |+e_L(m),
\end{align}
where the last equality follows from 
$2\Re (F_{Y}(f,L))= -2\D^2(f, \chi \cdot n^{it};L, Y)$ and the 
fact that $f \sim \chi \cdot n^{it}$. 
For every fixed $\delta>0$, we have
$$
\limsup_{N\to\infty}\E_{(m,n)\in S_\delta\cap[N]^2}\, e_L(m)
	\leq\frac{1}{\textup{d}(S_\delta)}\sum_{p>L}\frac1{p^2}\ll_\delta L^{-1},
	$$
where $\textup{d}(S_\delta)>0$ denotes the density of $S_\delta$. Indeed, $p\mid m^2$ if and only if $p\mid m$, and hence
$$
\E_{m,n\in[N]} \, e_L(m)=
 \sum_{L<p\leq N}\frac{1}{p}\,
\frac{\big|\{m\in[N]\colon p\mid m\}\big|}{N}
=\sum_{L<p\leq N}\frac{1}{p}\frac{\lfloor N/p \rfloor}{N}
\leq\sum_{L<p\leq N}\frac{1}{p^2}
\leq L^{-1}.
$$
Dividing this average by $|S_\delta\cap[N]^2|/N^2$ and taking $\limsup_{N\to\infty}$ proves the claimed bound.\footnote{This is where the choice of $m^2$, instead of a more complicated binary quadratic form, substantially simplifies the counting.} 
Using \eqref{a+b=c_new_main_term_6_fix_ld}, the previous bound,  and the fact that 
$\D(f, \chi \cdot n^{it}, L, \infty) +  L^{-1/2}+ L^{-1} \to 0$ as 
$L \to \infty$, one sees that to prove  
\eqref{a+b=c_new_main_term_5_fix_ld}, it suffices to show that 
\begin{multline}\label{a+b=c_new_main_term_7_fix_ld}
    \lim_{\delta ,s,r, K, L,N} \max_{Q_{1} \in \Phi_{s,r,P_1}} 
    \cesE_{(m,n) \in S_{\delta} \cap [N]^2}\,
    \big | f(Q_1 R_1(s, r, K, L, N, Q_1, Q_{2,r,K} , Q_{3,K,L} ,m,n)) \\ -  Q_L ^{2 it} m^{2it}  
    \exp(-2\D^2(f, \chi \cdot n^{it};L, \infty))
    \big |=0. 
\end{multline}
Now using Lemma \ref{a+b=c_new_lemma_choice_of_v_1} \eqref{a+b=c_new_E1: lemma_choice_of_v_1},
and applying Proposition \ref{quad_conc} and Remark
\ref{uniform_in_Q}, we get that
\begin{align}\label{a+b=c_new_main_term_8_fix_ld}
    &\limsup_{N\to \infty} \, 
    \max_{Q_1 \in \Phi_{s,r,P_1}} \,
    \cesE_{(m,n) \in S_{\delta} \cap [N]^2} 
    \big|f(Q_1 R_1(s, r, K, L, N, Q_1, Q_{2,r,K} , Q_{3,K,L} ,m,n)) \\
    & \hspace{30mm}-  Q_L ^{2 it} m^{2it}  
    \exp(-2\D^2(f, \chi \cdot n^{it};L, \infty)) \big| \leq 
   \nonumber \\
   & \Oh_{P_1, \delta}\big(\D(f, \chi \cdot n^{it}, L, \infty) +  L^{-1/2}\big)+  \limsup_{N\to \infty} \, 
    \max_{Q_1 \in \Phi_{s,r,P_1}} \, \nonumber \\
    &\cesE_{(m,n) \in S_{\delta} \cap [N]^2} 
    |f(Q_{1}) \, \overline{\chi(Q_{1})}  \, Q_{1}^{-it}  
    P_1(m,n)^{it} \, Q_{ L}^{2it}  \,
      \exp(G_{P_1,N}(f,L)) -  \nonumber \\  
      & \hspace{25mm}  Q_L ^{2 it} m^{2it}  
    \exp(-2\D^2(f, \chi \cdot n^{it};L, \infty)) \big|.  \nonumber
\end{align}
For $p>s$ with $\omega_{P_1}(p)>0$ we have $f(p)=\chi(p)\, p^{it}$,
which implies that 
$f(Q_{1})= \chi(Q_{1}) \, Q_{1}^{it}$, and since $L>s_0$, we have
$G_{P_1, N}(f,L)=0$. Therefore, the last quantity in \eqref{a+b=c_new_main_term_8_fix_ld} is equal to 
\begin{align}\label{a+b=c_new_main_term_9_fix_ld}
& \oh_{\delta, L\to \infty}(1) + \limsup_{N\to \infty} \, 
    \cesE_{(m,n) \in S_{\delta} \cap [N]^2} 
    \big |  P_1(m,n)^{it} Q_L^{2it}-  \\ & \nonumber 
    \hspace{30mm} m^{2it}  Q_L^{2it} 
    \exp(-2\D^2(f, \chi \cdot n^{it};L, \infty))\big|=\\
     & \nonumber
\oh_{\delta, L\to \infty}(1) + \limsup_{N\to \infty} \, 
    \cesE_{(m,n) \in S_{\delta} \cap [N]^2} 
    \big|  P_1(m,n)^{it} - \\
     & \hspace{30mm} \nonumber m^{2it} 
    \exp(-2\D^2(f, \chi \cdot n^{it};L, \infty))\big| \leq \\
    & \nonumber
\oh_{\delta, L\to \infty}(1) + |1- \exp(-2\D^2(f, \chi \cdot n^{it};L, \infty))|+ \\ & \hspace{30mm} \nonumber \limsup_{N\to \infty} \, 
    \cesE_{(m,n) \in S_{\delta} \cap [N]^2} 
    \big|  P_1(m,n)^{it}  m^{-2it} -1  \big|.
\end{align}

As $L \to \infty$, we have that $\D^2(f, \chi \cdot n^{it};L, \infty)\to 0$, so $\exp(-2\D^2(f, \chi \cdot n^{it};L, \infty)) 
\to 1$. 
Note that on $S_\delta$ we have $0<P_2(m,n)<m^2<P_1(m,n)$.
Hence, the logarithmic bound used to obtain \eqref{prop S_d a+b=c 1} gives
$$
0\leq\log\frac{P_1(m,n)}{m^2}
\leq\log\frac{P_1(m,n)}{P_2(m,n)}
\leq\delta.
$$
Using $|e^{iu}-1|\leq |u|$ for real $u$, we deduce that
$$
|P_1(m,n)^{it}m^{-2it}-1|
=\Big|e^{it\log(P_1(m,n)/m^2)}-1\Big|
\leq |t|\delta,
$$
so the last expression in 
\eqref{a+b=c_new_main_term_9_fix_ld} is 
\begin{equation*}
    \leq \oh_{{\delta}, L\to \infty}(1) + |t|\delta = 
    \oh_{\delta, L\to \infty}(1) + \oh_{t, \delta \to 0^+}(1).
\end{equation*}
We infer that for each $f\in \M_{p, \textup{fin.supp.},P_1}$, 
\eqref{a+b=c_new_main_term_5_fix_ld} holds, and 
\eqref{a+b=c_new_main_term_3_fix_ld} follows. 

Utilizing the same argument for $F_2, F_3$, one can show that 
\begin{multline}\label{a+b=c_new_main_term_10_fix_ld}
    \lim_{\delta ,s,r, K, L,N, Y} \max_{Q_{2} \in \Phi_{r, K, P_2}} 
    \cesE_{(m,n) \in S_{\delta} \cap [N]^2}\, \cesE_{y \leq Y}
    \\ \big \lVert 
    T_{Q_2 R_2(s, r, K, L, N, Q_{1,s,r}, Q_{2} , Q_{3,K,L} ,m,n)} F_2 -
    T_{\frac{m^2Q_L ^3 y +1}{Q_L y +1}} F_2 
    \big \rVert_{L^2(\mu)}=0,
\end{multline}
and
\begin{multline}\label{a+b=c_new_main_term_11_fix_ld}
    \lim_{\delta ,s,r, K, L,N, Y} \max_{Q_{3} \in \Phi_{K,L,P_3}} 
    \cesE_{(m,n) \in S_{\delta} \cap [N]^2}\, \cesE_{y \leq Y}
    \\ \big \lVert 
    T_{Q_3 R_3(s, r, K, L, N, Q_{1,s,r}, Q_{2,r,K} , Q_{3} ,m,n)} F_3 -
    T_{\frac{m^2Q_L ^3 y +1}{Q_L y +1}} F_3 
    \big \rVert_{L^2(\mu)}=0. 
\end{multline}
In view of \eqref{a+b=c_new_main_term_3_fix_ld}, 
\eqref{a+b=c_new_main_term_10_fix_ld} and 
\eqref{a+b=c_new_main_term_11_fix_ld}, the quantity in the 
left hand side of
\eqref{a+b=c_new_main_term_3_fix} is 
\begin{align*}
 \geq \inflim_{\,\delta ,s,r, K, L,N, Y}  \int_{X} \cesE_{y \leq Y}\,
     \cesE_{(m,n) \in S_{\delta} \cap [N]^2} 
      T_{\frac{m^2Q_L ^3 y +1}{Q_L y +1}} F_1 
    \cdot   
     T_{\frac{m^2Q_L ^3 y +1}{Q_L y +1}} F_2 \cdot
    T_{\frac{m^2Q_L ^3 y +1}{Q_L y +1}} F_3  \,
    d \mu   &  \\ 
    =\inflim_{\,\delta , K, L,N, Y}    \cesE_{y \leq Y}\,
     \cesE_{(m,n) \in S_{\delta} \cap [N]^2} \int_{X}
      T_{\frac{m^2Q_L ^3 y +1}{Q_L y +1}} F_1 
    \cdot 
     T_{\frac{m^2Q_L ^3 y +1}{Q_L y +1}} F_2 \cdot 
    T_{\frac{m^2Q_L ^3 y +1}{Q_L y +1}} F_3  \,
    d \mu  & 
    \\ = \int_{X}
       F_1  \cdot F_2 \cdot F_3 \, d \mu,& 
\end{align*}
where in the last equality we used the $T$-invariance of $\mu$.
Now, since $\lVert F \rVert_{L^{\infty}(\mu)}= 1$, we have that 
\begin{multline*}
    \int_{X} F_1  \cdot F_2 \cdot F_3 \, d \mu \geq 
    \int_{X} F \cdot F_1  \cdot F_2 \cdot F_3 \, d \mu = 
    \int_{X} F \cdot \E(F \mid \CX_{p, \textup{fin.supp.},P_1})  \cdot \\ \E(F \mid \CX_{p, \textup{fin.supp.},P_2}) \cdot 
    \E(F \mid \CX_{p, \textup{fin.supp.},P_3})\, d \mu
    \geq \Big(\int_X F \, d \mu\Big)^4,
\end{multline*}
where in the last inequality we used Proposition \ref{chu-lemma}.
This proves \eqref{a+b=c_new_main_term_3_fix}, 
and concludes the proof in the general case. 
\end{proof}

Finally we sketch  the proof of \eqref{a+b=c_new_main_term_20}. 

\begin{proof}[Sketch of proof of Proposition 
\ref{a+b=c_new_P: main a=c P_2 reducible} (lower bound \eqref{a+b=c_new_main_term_20})]
By homogeneity, we may assume that
	$\lVert F\rVert_{L^{\infty}(\mu)}= 1$.
Suppose also that
there are $i, j \in \{1,2,3\}$ with 
$i\neq j$ such that $\omega_{P_i}(p) = \omega_{P_j}(p)$ for all 
but finitely many primes $p$. 
For simplicity, we give the argument when $i=1$ and $j=3$;
the other cases are treated similarly, using the corresponding
factor projections.
In this case we have $\omega_{P_1}(p) = \omega_{P_3}(p)$ for all 
but finitely many primes $p$.

For each $s,r,K,L\in \N$, choose any $Q_{1,s,r}\in \Phi_{s,r,P_1}$,  
$Q_{2, r, K} \in \Phi_{r,K,P_2}$ and $Q_{3,K,L}\in \Phi_{K,L,P_3}$. 
As before, we see that in order to prove
\eqref{a+b=c_new_main_term_20}, it suffices to prove that 
\begin{align}\label{a+b=c_new_main_term_21}
\inflim_{\,\delta ,s,r, K, L,N}&
    \int_{X}  \cesE_{(m,n) \in S_{\delta} \cap [N]^2} 
     \Big( \cesE_{Q_{1}\in \Phi_{s,r,P_1}}  T_{Q_1 R_1(s, r, K, L, N, Q_1, Q_{2,r,K} , Q_{3,K,L} ,m,n)} F \Big)
    \cdot \nonumber \\ &\Big( \cesE_{Q_{2}\in \Phi_{r,K,P_2}}  
     T_{Q_2 R_2(s, r, K, L, N, Q_{1,s,r}, 
    Q_2 , Q_{3,K,L},m,n)} F \Big) \cdot  
    \\ &\Big( \cesE_{Q_{3}\in \Phi_{K,L,P_3}}
    T_{{Q}_3 R_3(s, r, K, L, N, Q_{1,s,r}, Q_{2,r,K} , Q_3 ,m,n)} F \Big) \,
    d \mu \geq \Big(\int_X F \, d \mu\Big)^3 . \nonumber 
\end{align}
Repeating Step $1$ in the proof of the general case of 
Proposition \ref{a+b=c_new_P: main a=c P_2 reducible}, and using the
$T$-invariance of $\mu$, we see that \eqref{a+b=c_new_main_term_21}
follows from 
\begin{multline}\label{a+b=c_new_main_term_22}
\inflim_{\,\delta, s,r, K, L,N}
    \int_{X} F \cdot \cesE_{(m,n) \in S_{\delta} \cap [N]^2}
    \, \cesE_{Q_{3}\in \Phi_{K,L,P_3}} \\
    \Big( \cesE_{Q_{1}\in \Phi_{s,r,P_1}} T_{\frac{Q_1 R_1(s, r, K, L, N, Q_1, Q_{2,r,K} , Q_{3,K,L} ,m,n)}{{Q}_3 R_3(s, r, K, L, N, Q_{1,s,r}, Q_{2,r,K} , Q_3 ,m,n)}} F_1 \Big) \cdot \\
    \Big( \cesE_{Q_{2}\in \Phi_{r,K,P_2}}  
     T_{\frac{Q_2 R_2(s, r, K, L, N, Q_{1,s,r}, 
    Q_2 , Q_{3,K,L},m,n)}{{Q}_3 R_3(s, r, K, L, N, Q_{1,s,r}, Q_{2,r,K} , Q_3 ,m,n)}} F \Big) \,
    d \mu   \geq \Big(\int_X F \, d \mu\Big)^3,
\end{multline}
where again $F_1=\E(F \mid \CX_{p, \textup{fin.supp.},P_1})$.

Next, we show that  
\begin{multline}\label{a+b=c_new_main_term_23}
    \lim_{\delta, s, r, K, L, N}\, 
    \max_{Q_1 \in \Phi_{s,r,P_1}, Q_3 \in \Phi_{K,L,P_3}} 
    \cesE_{(m,n) \in S_{\delta} \cap [N]^2} \\ 
    \Big \lVert 
    T_{\frac{Q_1 R_1(s, r, K, L, N, Q_1, 
    Q_{2,r,K} , Q_{3,K,L},m,n)}{{Q}_3 R_3(s, r, K, L, N, Q_{1,s,r}, Q_{2,r,K} , Q_3 ,m,n)}} F_1
    -F_1 \Big \rVert_{L^2(\mu)} =0.
\end{multline}

To see this, let $f\in \M_{p, \textup{fin.supp.},P_1}$, and
let $\chi, t$ and $s_0 \in \N$ so that 
$f(p)=\chi(p)\, p^{it}$ for all primes $p>s_0$ with $\omega_{P_1}(p)>0$. Since $\omega_{P_1}(p)= \omega_{P_3}(p)$ for all but finitely 
many primes $p$, we may assume that this equality holds for $p>s_0$. 
Also, suppose that $s,L$ are sufficiently
large so that $s>s_0$ and the conductor $q_{\chi}$ of $\chi$ 
divides $Q_{s}$, $\frac{Q_{ L}}{Q_1}$ and 
$\frac{Q_{ L}}{Q_3}$ for every
$Q_{1} \in \Phi_{s,r,P_1}, Q_{3} \in \Phi_{K, L, P_3}$.

Using Lemma \ref{a+b=c_new_lemma_choice_of_v_1} \eqref{a+b=c_new_E1: lemma_choice_of_v_1},
\eqref{a+b=c_new_E3: lemma_choice_of_v_1}, and
applying Proposition \ref{quad_conc} and Remark
\ref{uniform_in_Q} we get 
\begin{align}\label{eq_useful_30}
    &\limsup_{N\to \infty} \, 
    \max_{Q_1 \in \Phi_{s,r,P_1}, Q_3 \in \Phi_{K,L,P_3}} \,
    \cesE_{(m,n) \in S_{\delta} \cap [N]^2} 
    \big|f(Q_1) \overline{f(Q_3)} \cdot \\
    & f(R_1(s, r, K, L, N, Q_{1}, Q_{2, r, K} , Q_{3,K,L} ,m,n))
     \overline{f(R_3(s, r, K, L, N, Q_{1,s,r}, Q_{2,r,K} , Q_3 ,m,n))}-1\big| 
   \nonumber \\
   & \leq \Oh_{P_1, P_3,\delta}\big(\D(f, \chi \cdot n^{it}, L, \infty) +  L^{-1/2}\big)+  \limsup_{N\to \infty} \, 
    \max_{Q_1 \in \Phi_{s,r,P_1}, Q_3 \in \Phi_{K,L,P_3}} \, \nonumber \\
    &\cesE_{(m,n) \in S_{\delta} \cap [N]^2} 
    |f(Q_{1}) \overline{f(Q_3)} \, \overline{\chi(Q_{1})}  \, 
    Q_{1}^{-it}  
    P_1(m,n)^{it} \, Q_{ L}^{2it}  \,
      \exp(G_{P_1,N}(f,L)) \chi(Q_{3}) Q_3^{it} \cdot \nonumber \\  
      & \hspace{25mm}  P_3(m,n)^{-it}
      Q_{ L}^{-2it}\,
      \overline{\exp(G_{P_3,N}(f,L))}-1|.  \nonumber
\end{align}

For $p>s$ with $\omega_{P_1}(p)>0$ we have that $f(p)=\chi(p)\, p^{it}$, and since $\omega_{P_1}(p)= \omega_{P_3}(p)$ for $p>s_0$, we 
infer that $f(Q_{1})= \chi(Q_{1}) \, Q_{1}^{it}$,  
$f(Q_{3})= \chi(Q_{3}) \, Q_{3}^{it}$, and since $L>s_0$, we have
$G_{P_1, N}(f,L)=G_{P_3, N}(f,L)=0$. Therefore, the last quantity in \eqref{eq_useful_30} 
is equal to 
\begin{equation}\label{eq_useful_31}
\oh_{\delta, L\to \infty}(1) + \limsup_{N\to \infty} \, 
    \max_{Q_1 \in \Phi_{s,r,P_1}, Q_3 \in \Phi_{K,L,P_3}} \,
    \cesE_{(m,n) \in S_{\delta} \cap [N]^2} 
    |  P_1(m,n)^{it} P_3(m,n)^{-it}   
       -1|.
\end{equation}
From \eqref{prop S_d a+b=c 1} we have that for $(m,n) \in S_{\delta}$,
$\big| P_1(m,n)^{it} P_3(m,n) ^{-it} -1 \big| \leq |t| \delta$,
so the expression in 
\eqref{eq_useful_31} is 
\begin{equation*}
    \leq \oh_{\delta, L\to \infty}(1) + |t| \delta = 
    \oh_{\delta, L\to \infty}(1) + \oh_{t, \delta \to 0^{+}}(1).
\end{equation*}
We infer that 
for $f\in \M_{p, \textup{fin.supp.},P_1}$, we have 
\begin{multline*}
    \lim_{\delta,s , r, K, L, N}\,
    \max_{Q_1 \in \Phi_{s, r,P_1}, Q_3 \in \Phi_{K,L,P_3}} 
    \cesE_{(m,n) \in S_{\delta} \cap [N]^2}
    \big|f(Q_1) \overline{f(Q_3)} \cdot \\
     f(R_1(s, r, K, L, N, Q_{1}, Q_{2,r,K} , Q_{3,K,L} ,m,n))
     \overline{f(R_3(s, r, K, L, N, Q_{1,s,r}, Q_{2,r,K} , Q_3 ,m,n))}-1\big|=0,
\end{multline*}
and since the spectral measure of $F_1$ is supported on 
$\M_{p, \textup{fin.supp.},P_1}$, \eqref{a+b=c_new_main_term_23} follows. 

Using \eqref{a+b=c_new_main_term_23}, we see that in order to prove 
\eqref{a+b=c_new_main_term_22}, it suffices to show that 
\begin{multline}\label{a+b=c_new_main_term_24}
\inflim_{\,\delta, s,r, K, L,N}
    \int_{X} F \cdot F_1 \cdot \cesE_{(m,n) \in S_{\delta} \cap [N]^2}
    \, \cesE_{Q_{3}\in \Phi_{K,L,P_3}}  \\
    \Big( \cesE_{Q_{2}\in \Phi_{r,K,P_2}}  
     T_{\frac{Q_2 R_2(s, r, K, L, N, Q_{1,s,r}, 
    Q_2 , Q_{3,K,L},m,n)}{{Q}_3 R_3(s, r, K, L, N, Q_{1,s,r}, Q_{2,r,K} , Q_3 ,m,n)}} F \Big) \,
    d \mu   \geq \Big(\int_X F \, d \mu\Big)^3. 
\end{multline}
Finally, writing $F= F_2 + F_3$, with 
$F_2=\E( F \mid \X_{p, \textnormal{\textup{fin.supp.}},P_2}\cap
\X_{p, \textnormal{\textup{fin.supp.}},P_3})$ and $F_3 \perp F_2$, 
and following closely (up to obvious 
modifications) the argument in Steps $3, 4$ in the proof of 
Proposition \ref{P: main a=c P_2 irreducible}, one sees that the 
quantity in the left hand side of \eqref{a+b=c_new_main_term_24} is
\begin{equation*}
\geq    \int_{X} F \cdot F_1\cdot   F_2  \,
    d \mu  \geq \Big(\int_X F \, d \mu\Big)^3,
\end{equation*}
where in the last inequality we used Proposition 
\ref{chu-lemma}, and that $F_1=\E(F \mid \CX_{p, \textup{fin.supp.},P_1})$, 
$F_2=\E( F \mid \X_{p, \textnormal{\textup{fin.supp.}},P_2}\cap
\X_{p, \textnormal{\textup{fin.supp.}},P_3})$. This concludes the 
proof.
\end{proof}}

\subsection{$P_2$ reducible, $P_3$ irreducible or $P_2$ irreducible, $P_3$ reducible}\label{Sec : a+b=c_P_2_reducible P_3 irreducible}
Those cases are symmetric, so we may assume $P_2$ is reducible and 
$P_3$ is irreducible. A representative example is $3x^2+y^2=4z^2$, in 
which case \eqref{quadratic_forms_a+b=c} gives
$$
P_1(m,n)=m^2+3n^2, \quad P_2(m,n)=(m-3n)(m+n), \quad P_3(m,n)=m^2+6mn-3n^2.
$$

Since $P_2$ is reducible, it is easy to see that there are 
$\lambda_1, \lambda_2 \in \Z\setminus \{0\}$ such that 
$\lambda_1+ \lambda_2 =-2b$, $\lambda_1\lambda_2 =-ab$, and 
\begin{equation}\label{E:P2reducible}
P_2(m,n)=(m+\lambda_1 n)(m+\lambda_2 n).
\end{equation}
 Since $\lambda_1 \lambda_2<0$,
we have $\lambda_1 \neq \lambda_2$.
From now on we assume that $s$ is large enough so that
for $p>s$ we have $\omega_{P_1}(p), \omega_{P_3}(p) \in \{0,2\}$, and 
\begin{equation}\label{a+b=c_new_E2: p not dividing}
p\nmid 2 ab(a+b)\lambda_1 \lambda_2 (\lambda_1 - \lambda_2)
(ab + \lambda_1 ^2).    
\end{equation}

Given $Q_1 \in \Phi_{s,r,P_1}$, $Q_2 \in \Phi_{r,K}$ and 
$Q_3 \in \Phi_{K,L, P_3}$,
we will choose an integer $v$ depending on $s,r,K,L,Q_1,Q_2, Q_3$, and 
we will average over the grid 
$(Q_{L} m+1, Q_{L} n +v)$ in \eqref{T: main a+b=c eq}.

\begin{lemma}\label{a+b=c_new_lemma_choice_of_v_1_a=c_P_3_irred}
Let $s,r,K,L\in \N$ with $s<r<K<L/2$, 
 $P_1,P_2,P_3$ be as in 
\eqref{quadratic_forms_a+b=c}, with $P_2$ as in \eqref{E:P2reducible}, and let $Q_1 \in 
\Phi_{s,r,P_1}$, $Q_2 \in \Phi_{r,K}$ and $Q_3 \in \Phi_{K,L,P_3}$. 
Then there is an integer
$v$ with $0\leq v \leq Q_{ L}-1$ such that 
\begin{enumerate}[(i)]
    \item $(Q_{ L}, P_1(1,v))=Q_1$, 
    $\frac{P_1(1,v)}{Q_1}\equiv Q_1 ^{-1} \pmod {Q_s}$,
    \item $(Q_{ L},1+\lambda_1 v)= Q_2$, 
    $\frac{1+\lambda_1 v}{Q_2} \equiv Q_2 ^{-1} 
    \pmod{{Q_s}}$,
     \item $(Q_{ L},1+\lambda_2 v)= 1$, 
    $1+\lambda_2 v \equiv 1    \pmod{{Q_s}}$,
    \item $(Q_{ L},P_3(1,v))=Q_3$, 
    $\frac{P_3(1,v)}{Q_3} \equiv Q_3^{-1} \pmod{{Q_s}}$.
\end{enumerate}
\end{lemma}

\begin{proof}
Let $p\in \P$ with $s<p \leq r$ and $\omega_{P_1}(p)=2$, and consider the 
polynomial $g_1(x):=P_1(1,x)-p^{\theta_p(Q_1)}$. 
Then $\textup{mod}\: p$ the polynomial
becomes $g_1(x)=P_1(1,x)$, and it has two simple roots
$z_1, z_1'$ with $z_1 \not \equiv z_1' \pmod p$. Using  
Hensel's lemma, we may lift $z_1$ to a unique root 
$\zeta_1 = \zeta_1 (p, Q_1)$ of $g_1 \mod {p^{\theta_p(Q_1)+1}}$.

Analogously, for $p\in \P$ with $K<p \leq L$ and $\omega_{P_3}(p)=2$, 
we may find $\zeta_3 = \zeta_3 (p, Q_3)$ so that if 
$g_3(x)=P_3(1,x)-p^{\theta_p(Q_3)}$, then $g_3(\zeta_3)\equiv 0 
\pmod {p^{\theta_p(Q_3)+1}}$.

Using the Chinese Remainder theorem we may then choose $v$ with 
$0\leq v \leq Q_{L}-1$
such that 
\begin{itemize}
     \item $v\equiv 0 \pmod {p^{2s}}$ for $p\leq s$,
     \item $v\equiv 0 \pmod {p}$ for $s<p \leq r$ with $\omega_{P_1}(p)=0$,
     \item $v\equiv \zeta_1(p, Q_1)$
    $\pmod {p^{\theta_p(Q_1)+1}}$ for $s< p\leq r$ with 
    $\omega_{P_1}(p)=2$,
    \item $v\equiv \lambda_1 ^{-1} (p^{\theta_p(Q_2)} -1 )\pmod{p^{\theta_p(Q_2)+1}}$ for $r<p\leq K$,
    \item $v\equiv 0 \pmod {p}  $ for $K<p\leq L$ with $\omega_{P_3}(p)=0$,
    \item $v\equiv \zeta_3(p,Q_3) \pmod {p^{\theta_p(Q_3)+1}}$ for 
    $K<p\leq L$ with $\omega_{P_3}(p)=2$.
\end{itemize}
The rest of the proof is very similar to the proof of Lemma 
\ref{a+b=c_new_lemma_choice_of_v_1}, so it is omitted.
\end{proof}

When $P_2$ is reducible and $P_3$ is irreducible, 
Theorem \ref{T: main a+b=c} follows from the next result.

\begin{proposition}\label{a+b=c_new_P: main a=c P_3 irreducible}
Let $(X,\CX,\mu,T_n)$  be a pretentious multiplicative action, let 
$F\in L^{\infty}(\mu)$ with $F\geq 0$, 
let $P_1,P_2,P_3$ be as in 
\eqref{quadratic_forms_a+b=c},  and suppose that 
$P_2$ is reducible and $P_3$ is irreducible.
Also, let $S_{\delta}$ be as in \eqref{def S_d a+b=c 1} and  $v=v(s, r, K,L,Q_1,Q_2,Q_3)$ be  as in 
Lemma \ref{a+b=c_new_lemma_choice_of_v_1_a=c_P_3_irred}.
Then 
    \begin{multline}\label{a+b=c_new_E: P: main a=c P_3 irreducible}
    \inflim_{\,\delta,s, r,K,L,N}\:
    \cesE_{Q_{1}\in \Phi_{s,r,P_1}} \cesE_{Q_{2}\in \Phi_{r,K}}
    \cesE_{Q_{3}\in \Phi_{K,L,P_3}}
   \cesE_{(m,n) \in S_{\delta} \cap [N]^2}  \\
     \int_{X}  \prod_{j=1} ^3
    T_{P_j(Q_{L} m+1, Q_{L} n +v)} F 
    \, d \mu \geq \Big(\int_X F \, d \mu\Big)^3.
    \end{multline}
\end{proposition}

To prove Proposition \ref{a+b=c_new_P: main a=c P_3 irreducible}, one must first 
establish an analog of Lemma \ref{a+b=c_new_L: main a=c P_2 reducible}, 
which states that in the integral \eqref{a+b=c_new_E: P: main a=c P_3 
	irreducible}, the first term is asymptotically independent of $Q_2, Q_3$, the second 
term is asymptotically independent of $Q_1, Q_3$, and the third term is asymptotically independent 
of $Q_1, Q_2$. Both the statement and the proof
of this lemma are identical to those of  Lemma 
\ref{a+b=c_new_L: main a=c P_2 reducible},
so they are omitted. Finally, the proof of Proposition 
\ref{a+b=c_new_P: main a=c P_3 irreducible} is very similar to the 
proof of the special case of Proposition 
\ref{a+b=c_new_P: main a=c P_2 reducible}, so it is omitted.

\subsection{$P_2$, $P_3$ reducible}\label{Sec : a+b=c P_2 P_3 reducible}
Finally, we handle the case when both $P_2$ and $P_3$ are reducible.
A representative example is $9x^2+16y^2=25z^2$, in 
which case \eqref{quadratic_forms_a+b=c} gives
$$
P_1(m,n)=m^2+144n^2, \quad P_2(m,n)=(m-36n)(m+4n), \quad P_3(m,n)=(m-6n)(m+24n).
$$
Again, since the proof is very similar to the proofs we have 
already presented, we only note the modifications needed
to obtain the result in this case.

Using that $P_2$, $P_3$ are reducible, we see that there are 
$\lambda_1, \lambda_2, \lambda_3, \lambda_4  \in \Z\setminus \{0\}$
such that 
$\lambda_1+ \lambda_2 =-2b$, $\lambda_1\lambda_2 =-ab$,
$\lambda_3+ \lambda_4 =2a$, $\lambda_3\lambda_4 =-ab$
and 
\begin{equation}\label{E:P23reducible}
P_2(m,n)=(m+\lambda_1 n)(m+\lambda_2 n),\quad 
P_3(m,n)=(m+\lambda_3 n)(m+\lambda_4 n).
\end{equation}
Since $\lambda_1 \lambda_2<0$ and $\lambda_3 \lambda_4<0$,
we have $\lambda_1 \neq \lambda_2$ and $\lambda_3 \neq \lambda_4$.
From now on we assume that $s$ is large enough so that
for $p>s$ we have  $\omega_{P_1}(p) \in \{0,2\}$ and 
\begin{equation}\label{new_a+b=c_new_E2: p not dividing}
p\nmid 2 ab(a+b)\lambda_1 \lambda_2 (\lambda_1 - \lambda_2)
(ab + \lambda_1 ^2)
\lambda_3 \lambda_4 (\lambda_3 - \lambda_4)
(ab + \lambda_3 ^2).    
\end{equation}

Given $Q_1 \in \Phi_{s,r,P_1}$, $Q_2 \in \Phi_{r,K}$ and 
$Q_3 \in \Phi_{K,L}$,
we will choose an integer $v$ depending on $s,r,K,L,Q_1,Q_2, Q_3$, and 
we will average over the grid 
$(Q_{L} m+1, Q_{L} n +v)$ in \eqref{T: main a+b=c eq}.

\begin{lemma}\label{a+b=c_new_lemma_choice_of_v_1_a=c_P_3_red}
Let $s,r,K,L\in \N$ with $s<r<K<L/2$, 
  $P_1,P_2,P_3$ be as in 
\eqref{quadratic_forms_a+b=c}, with $P_2,P_3$ as in \eqref{E:P23reducible}, 
and let $Q_1 \in 
\Phi_{s,r,P_1}$, $Q_2 \in \Phi_{r,K}$ and $Q_3 \in \Phi_{K,L}$. 
Then there is an integer
$v$ with $0\leq v \leq Q_{ L}-1$ such that 
\begin{enumerate}[(i)]
    \item $(Q_{ L}, P_1(1,v))=Q_1$, 
    $\frac{P_1(1,v)}{Q_1}\equiv Q_1 ^{-1} \pmod {Q_s}$,
     \item $(Q_{ L},1+\lambda_1 v)= Q_2$, 
    $\frac{1+\lambda_1 v}{Q_2} \equiv Q_2 ^{-1}    \pmod{{Q_s}}$,
    \item $(Q_{ L},1+\lambda_2 v)=1$ and 
    $1+\lambda_2 v \equiv 1 \pmod{{Q_s}}$,
        \item $(Q_{ L}, 1+\lambda_3 v)= Q_3$, 
    $\frac{1+\lambda_3 v}{Q_3} \equiv Q_3 ^{-1}    \pmod{{Q_s}}$,
    \item $(Q_{ L},1+\lambda_4 v)=1$ and 
    $1+\lambda_4 v \equiv 1 \pmod{{Q_s}}$.
\end{enumerate}
\end{lemma}

\begin{proof}
For $p\in \P$ with $s<p \leq r$ and $\omega_{P_1}(p)=2$, 
we may use Hensel's lemma to find $\zeta_1 = \zeta_1 (p, Q_1)$ so that if 
$g_1(x):=P_1(1,x)-p^{\theta_p(Q_1)}$, then $g_1(\zeta_1)\equiv 0 
\pmod {p^{\theta_p(Q_1)+1}}$.

Using the Chinese Remainder theorem, we may then choose $v$ with 
$0\leq v \leq Q_{L}-1$
such that 
\begin{itemize}
     \item $v\equiv 0 \pmod {p^{2s}}$ for $p\leq s$,
     \item $v\equiv 0 \pmod {p}$ for $s<p \leq r$ with $\omega_{P_1}(p)=0$,
     \item $v\equiv \zeta_1(p, Q_1)$
    $\pmod {p^{\theta_p(Q_1)+1}}$ for $s< p\leq r$ with 
    $\omega_{P_1}(p)=2$,
    \item $v\equiv \lambda_1 ^{-1} (p^{\theta_p(Q_2)} -1 )\pmod{p^{\theta_p(Q_2)+1}}$ for $r<p\leq K$,
    \item $v\equiv \lambda_3 ^{-1} (p^{\theta_p(Q_3)} -1 )\pmod{p^{\theta_p(Q_3)+1}}$ for $K<p\leq L$.
\end{itemize}
The rest of the proof is very similar to the proof of 
Lemma \ref{a+b=c_new_lemma_choice_of_v_1}, so it is omitted.
\end{proof}

The case when $P_2$, $P_3$ are reducible in 
Theorem \ref{T: main a+b=c} follows from the next 
result.

\begin{proposition}\label{a+b=c_new_P: main a=c P_3 reducible}
Let $(X,\CX,\mu,T_n)$  be a pretentious multiplicative action, let 
$F\in L^{\infty}(\mu)$ with $F\geq 0$, 
let $P_1,P_2,P_3$ be as in 
\eqref{quadratic_forms_a+b=c},  and suppose that 
$P_2,P_3$ are reducible.
Also, let $S_{\delta}$ be as in \eqref{def S_d a+b=c 1} and $v=v( s, r,K,L,Q_1,Q_2,Q_3)$ be  as in 
Lemma \ref{a+b=c_new_lemma_choice_of_v_1_a=c_P_3_red}.
Then 
    \begin{multline}\label{a+b=c_new_E: P: main a=c P_3 reducible}
    \inflim_{\, \delta,s, r,K,L,N}\:
    \cesE_{Q_{1}\in \Phi_{s,r,P_1}} \cesE_{Q_{2}\in \Phi_{r,K}}
    \cesE_{Q_{3}\in \Phi_{K,L}}
     \cesE_{(m,n) \in S_{\delta} \cap [N]^2} \\ 
     \int_{X} \prod_{j=1} ^3
    T_{P_j(Q_{L} m+1, Q_{L} n +v)} F 
    \, d \mu \geq \Big(\int_X F \, d \mu\Big)^3.
    \end{multline}
\end{proposition}

Again, to prove Proposition \ref{a+b=c_new_P: main a=c P_3 reducible}, one must 
first establish an analog of Lemma \ref{a+b=c_new_L: main a=c P_2 reducible}. Using that result, the proof of 
Proposition \ref{a+b=c_new_P: main a=c P_3 reducible} is very similar
to the special case of Proposition 
\ref{a+b=c_new_P: main a=c P_2 reducible}, so it is omitted.

 {
\section*{Appendix - Correspondence principle for partition regularity of  dilation-invariant patterns} 

A family $\CP$ of non-empty finite subsets of $\N$ is
\emph{dilation invariant} if 
$$aF:=\{an:n\in F\}\in\CP$$ whenever
$F\in\CP$ and $a\in\N$, and \emph{partition regular} if every
finite coloring of $\N$ has a monochromatic member of $\CP$.
A set $\Lambda\subseteq\N$ is \emph{multiplicatively syndetic} if
$$\bigcup_{j=1}^k j^{-1}\Lambda=\N$$ for some $k\in\N$, where
$j^{-1}\Lambda:=\{n\in\N:jn\in\Lambda\}$.
In the next statement we also allow non-invertible multiplicative actions on a
probability space $(X,\X,\mu)$, that is,  the maps
$T_n$ are measurable and measure-preserving, with $T_1=\mathrm{id}$
and $T_mT_n=T_{mn}$.

\begin{proposition*}
	Let $\CP$ be a dilation-invariant family of nonempty finite subsets
	of $\N$. The following statements are equivalent:
	\begin{enumerate}[\upshape(i)]
		\item $\CP$ is partition regular.
		\item Every multiplicatively syndetic subset of $\N$ contains a member
		of $\CP$.
		\item For every multiplicative action $(X,\X,\mu,T_n)$ and every 
		$A\in\X$ with $\mu\big(\bigcup_{j=1}^k T_j^{-1}A\big)=1$ for some
		$k\in\N$, there exists $F\in\CP$ such that
		$$
		\mu\big(\bigcap_{n\in F}T_n^{-1}A\big)>0.
		$$
	\end{enumerate}
\end{proposition*}
\begin{remark}
	The equivalence remains valid if \text{(iii)} is restricted to
	invertible actions. Indeed, an invertible natural extension preserves the
	covering condition and the measures of all finite intersections.
%	Also,  for invertible
%	actions, inverse images may also be replaced by %forward images.
%	Indeed, applying $T_q$, with $q=k!$, we get 
%	$\mu\big(\bigcup_{j=1}^k T_{q/j}A\big)=1$, so %\text{(iii)} applies
%	to the  action $S_n:=T_n^{-1}$.
\end{remark}

\begin{proof}
	The equivalence of \text{(i)} and \text{(ii)} follows from the
	argument of \cite[Theorem~1.1]{ChapmanSyndetic}.
	The implication \text{(iii)}$\implies$\text{(ii)} follows from 
	Furstenberg's correspondence principle for dilation-invariant
	densities~\cite{Bergelson_multiplicative}.

	Thus, we only have to prove the implication
	\text{(ii)}$\implies$\text{(iii)}.
	Set $U:=\bigcup_{j=1}^k T_j^{-1}A$ and
	$X_0:=\bigcap_{n\in\N}T_n^{-1}U$. Our covering assumption gives
	$\mu(X_0)=1$. For $x\in X_0$, let
	$\Lambda_x:=\{n\in\N:T_nx\in A\}$.
	For every $n\in\N$, the inclusion $T_nx\in U$ gives
	$jn\in\Lambda_x$ for some $j\leq k$, so $\Lambda_x$ is
	multiplicatively syndetic. By \text{(ii)}, it contains a member
	of $\CP$. Hence
	$$
	X_0\subseteq\bigcup_{F\in\CP}\ \bigcap_{n\in F}T_n^{-1}A.
	$$
	Since $\CP$ is countable and $\mu(X_0)=1$, one of these
	intersections has positive measure.
\end{proof}

For each fixed Rado triple $(a,b,c)$, we can apply the proposition to
the dilation-invariant family of sets $\{x,y,z\}$ with pairwise distinct
$x,y,z\in\N$ satisfying $ax^2+by^2=cz^2$.
Together with the necessity of the Rado condition noted in the
introduction, this proves the equivalence of
Conjectures~\ref{Rado} and~\ref{RadoErgodic}.
The covering condition may be taken to hold exactly by restricting
the action to an invariant conull set.}

\bibliographystyle{abbrv}
\bibliography{refs}

\end{document}